\documentclass[12pt]{amsart}
\usepackage{amssymb}

\usepackage{amsmath}
\usepackage{amscd}
\usepackage{amsthm}
\usepackage{graphicx}
\usepackage[all]{xy}

\numberwithin{equation}{section}

\theoremstyle{plain}
 \newtheorem{thm}{Theorem}[section]
 \newtheorem{lem}[thm]{Lemma}
 \newtheorem{pro}[thm]{Proposition}
 \newtheorem{cor}[thm]{Corollary}
 \newtheorem{conj}[thm]{Conjecture}
 
 \newtheorem{exa}[thm]{Example}
\theoremstyle{definition}

 \newtheorem{defn}[thm]{Definition}

\begin{document}

\title{On a Lomonaco-Kauffman conjecture}

\author{Takahito Kuriya}
\date{}
\address{Kyushu University, Faculty of Mathematics, 6-10-1 Hakozaki, Higasiku, Fukuoka 812-8581, Japan}
\email{marron@math.kyushu-u.ac.jp}
\maketitle
\begin{abstract} 
Samuel J. Lomonaco Jr and Louis H. Kauffman conjectured that tame knot theory and knot mosaic theory are equivalent. 
We give a proof of the Lomonaco-Kauffman conjecture.
\end{abstract}

\section{Introduction}{\label {intro}}
In {\cite{LK}}, Samuel J. Lomonaco Jr and Louis H. Kauffman created a formal system $(\mathbb{K},\mathbb{A})$ consisting of a graded set $K$ of symbol strings, called knot mosaics, and a graded subgroup $A$, called the knot mosaic ambient group, of the
group of all permutations of the set of knot mosaics $K$ , and they conjectured that the formal system $(\mathbb{K},\mathbb{A})$ fully captures the entire structure
of tame knot theory. 
\begin{conj}
Let $k_{1}$ and $k_{2}$ be two tame knots (or links), and let $K_{1}$ and
$K_{2}$ be two arbitrary chosen mosaic representatives of $k_{1}$ and $k_{2}$,
respectively. \ Then $k_{1}$ and $k_{2}$ are of the same knot type if and only
if the representative mosaics $K_{1}$ and $K_{2}$ are of the same knot mosaic
type. \ In other words, knot mosaic type is a complete invariant of tame knots.
\end{conj}
The goal of this paper is to prove the conjecture.

\begin{thm}[Main Theorem]{\label{LK}}
The Lomonaco-Kauffman conjecture is true.
\end{thm}

The paper is organized as follows. In Section \ref{knot_mosaic}, \ref{mosaic_move}, \ref{knot_type}, we review some notations related to knot mosaic. In Section \ref{grid_diagram}, we recall the definition of the grid diagram. In Section \ref{grid_mosaic_diagram}, \ref{grid_mosaic_move}, we will see a relationship between knot mosaic and grid diagram. In Section \ref{zoom}, we define $5 \times$ zoom map. Section \ref{proof} is devoted to proof of the Lomonaco-Kauffman conjecture. In Appendix, we give a list of $\mathbb{K}^{(n)}/\mathbb{A}(n) \ (n \leq 4)$. Using this, we can know the mosaic number of some knots.\\

{\bf Acknowledgment.}
I would like to thank many people for their help with this paper, especially, to Professor Osamu Saeki for his encouragement.

\section{Knot Mosaic}{\label {knot_mosaic}}
Let $\mathbb{T}^{(u)}$ denote the set of the following 11 symbols
\begin{figure}[ht]
\begin{tabular}{cccccccccc}
\begin{minipage}{1.0cm}
\includegraphics[width=\hsize]{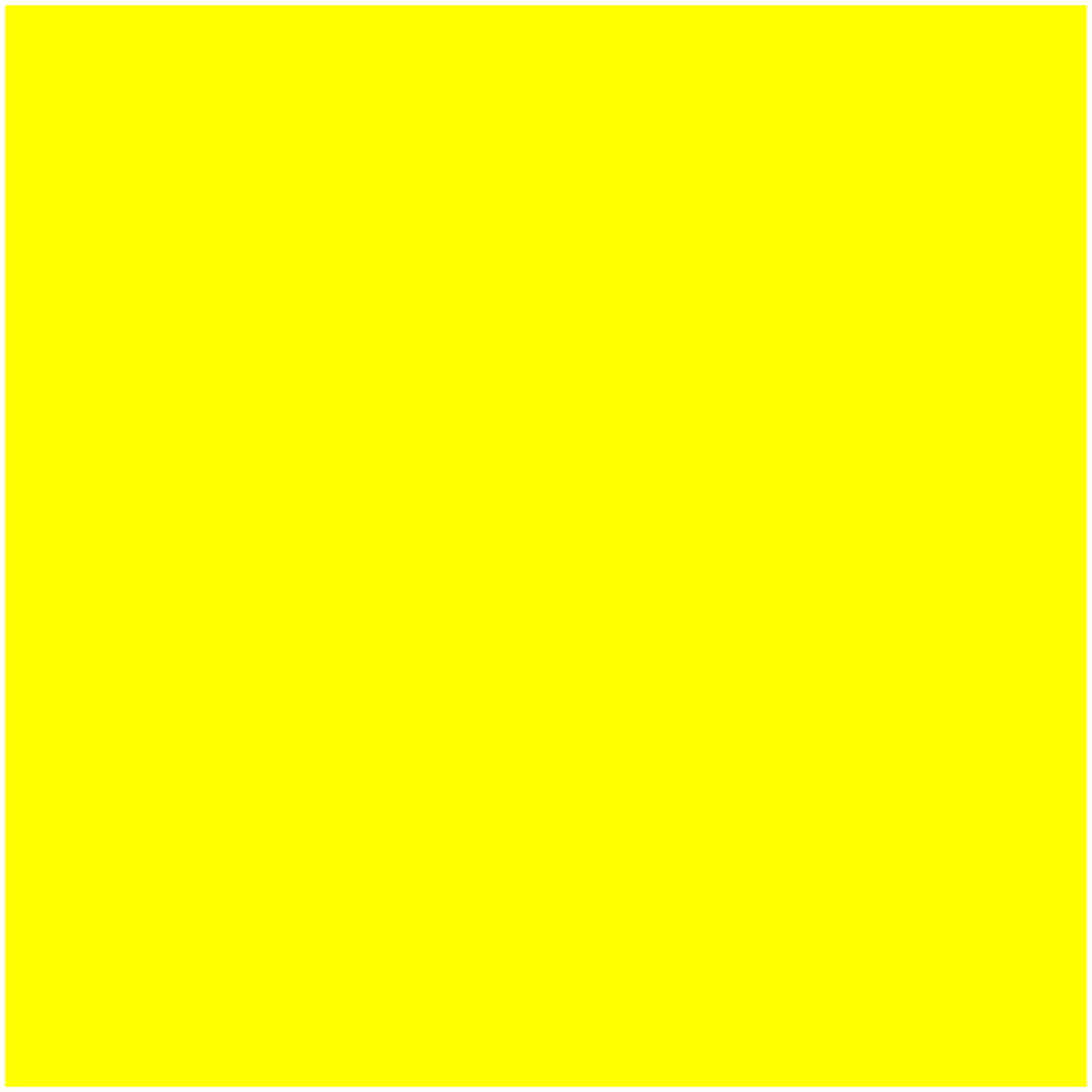}
\end{minipage}
\begin{minipage}{1.0cm}
\includegraphics[width=\hsize]{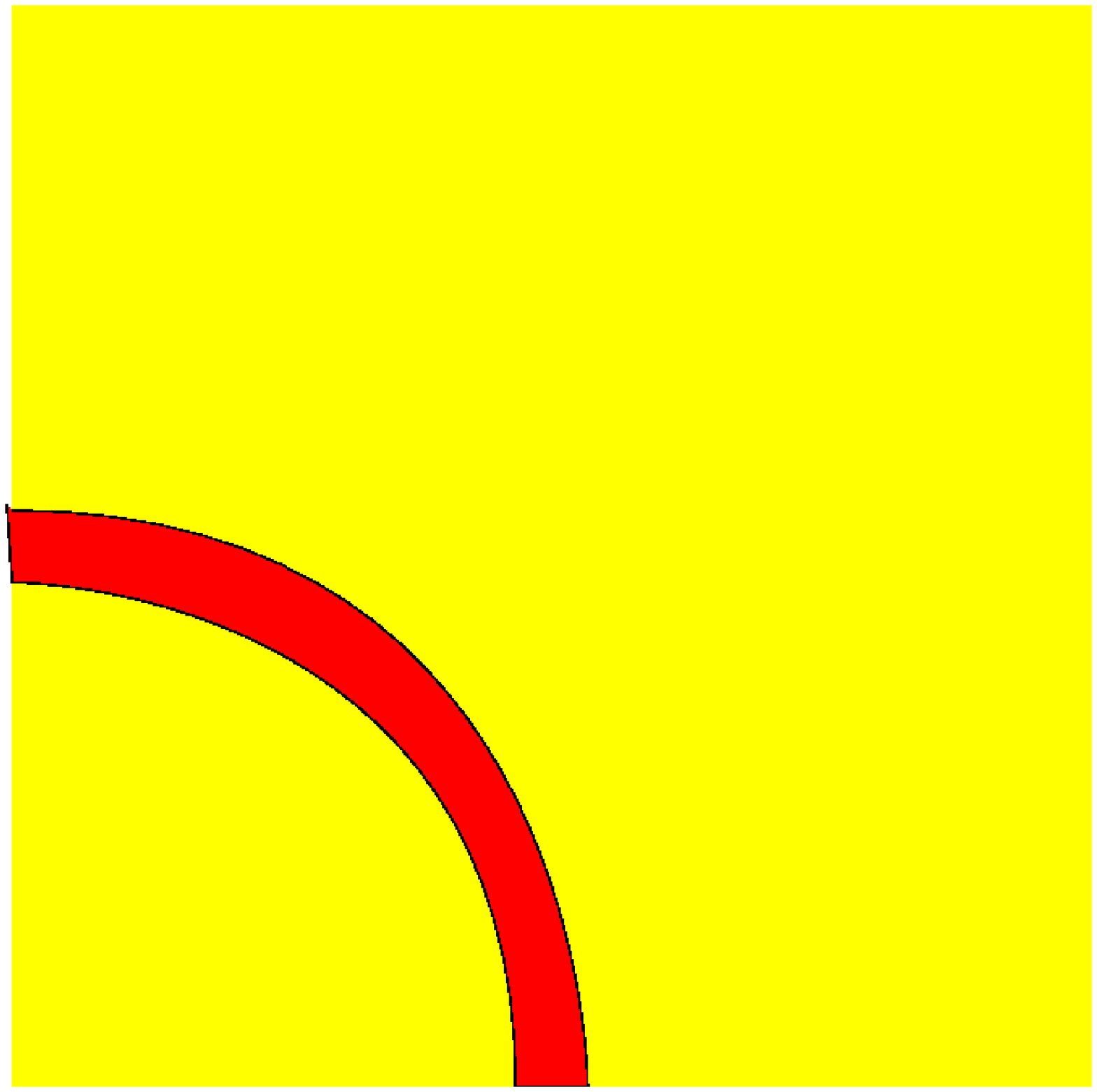}
\end{minipage}
\begin{minipage}{1.0cm}
\includegraphics[width=\hsize]{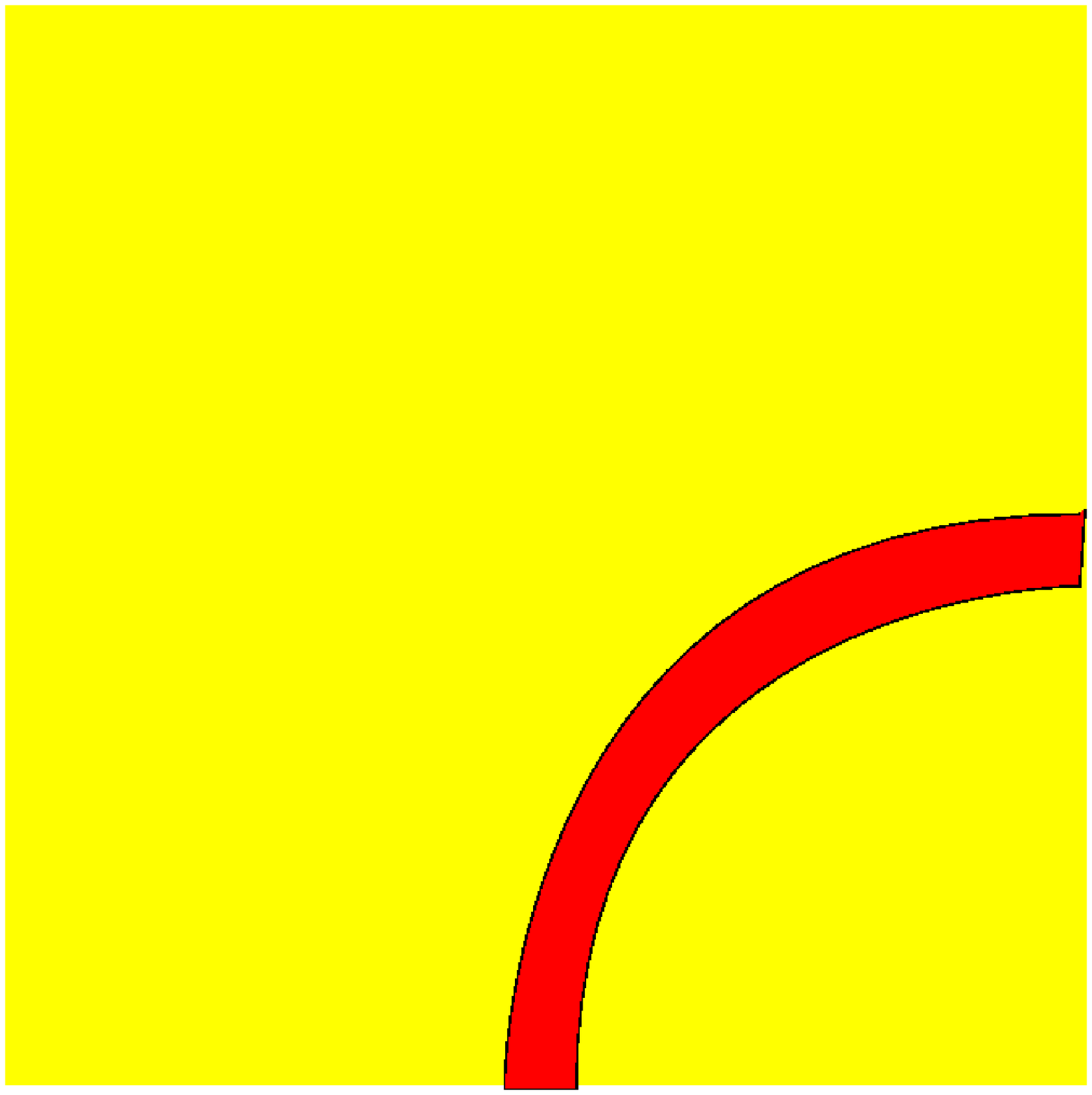}
\end{minipage}
\begin{minipage}{1.0cm}
\includegraphics[width=\hsize]{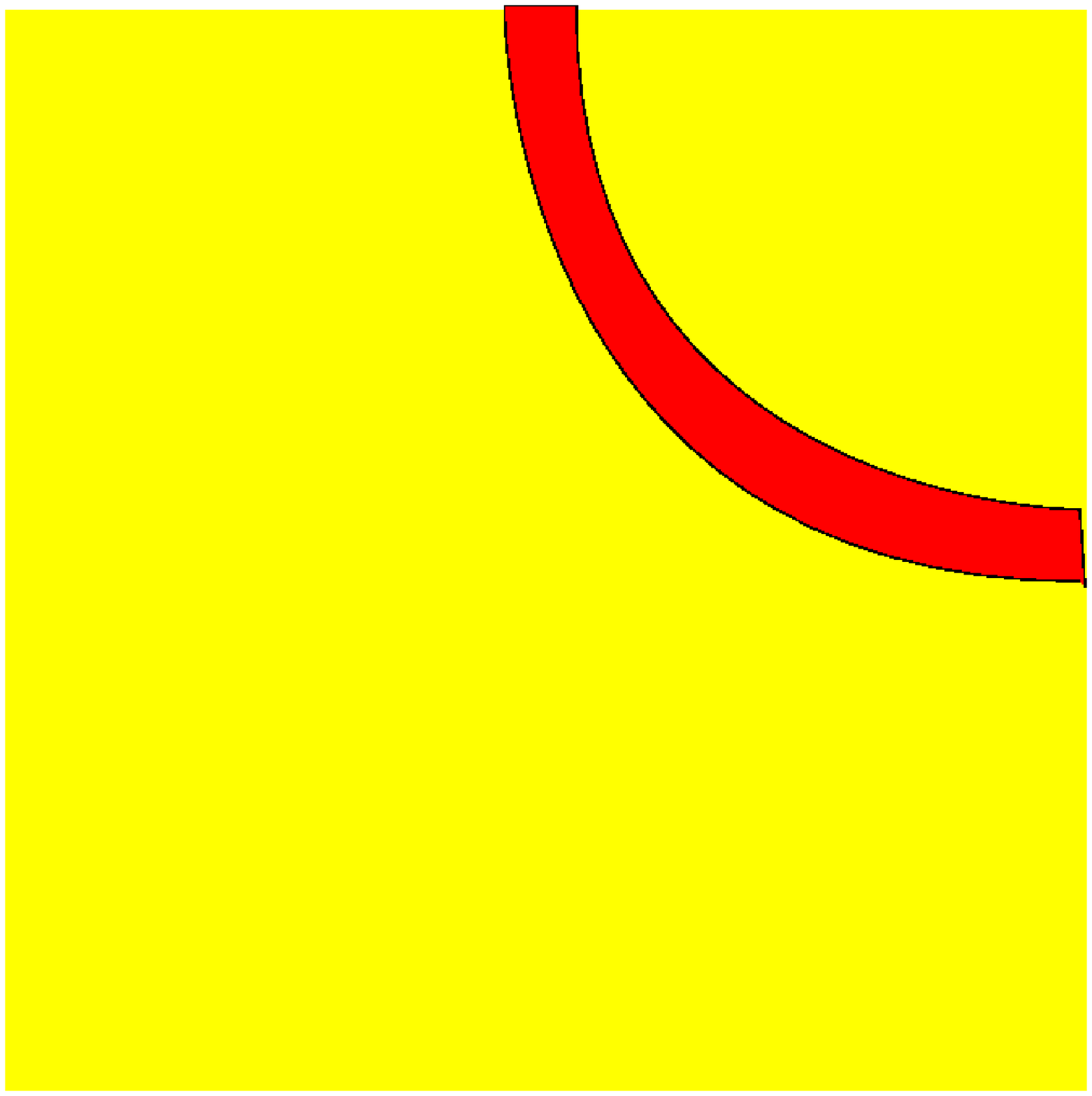}
\end{minipage}
\begin{minipage}{1.0cm}
\includegraphics[width=\hsize]{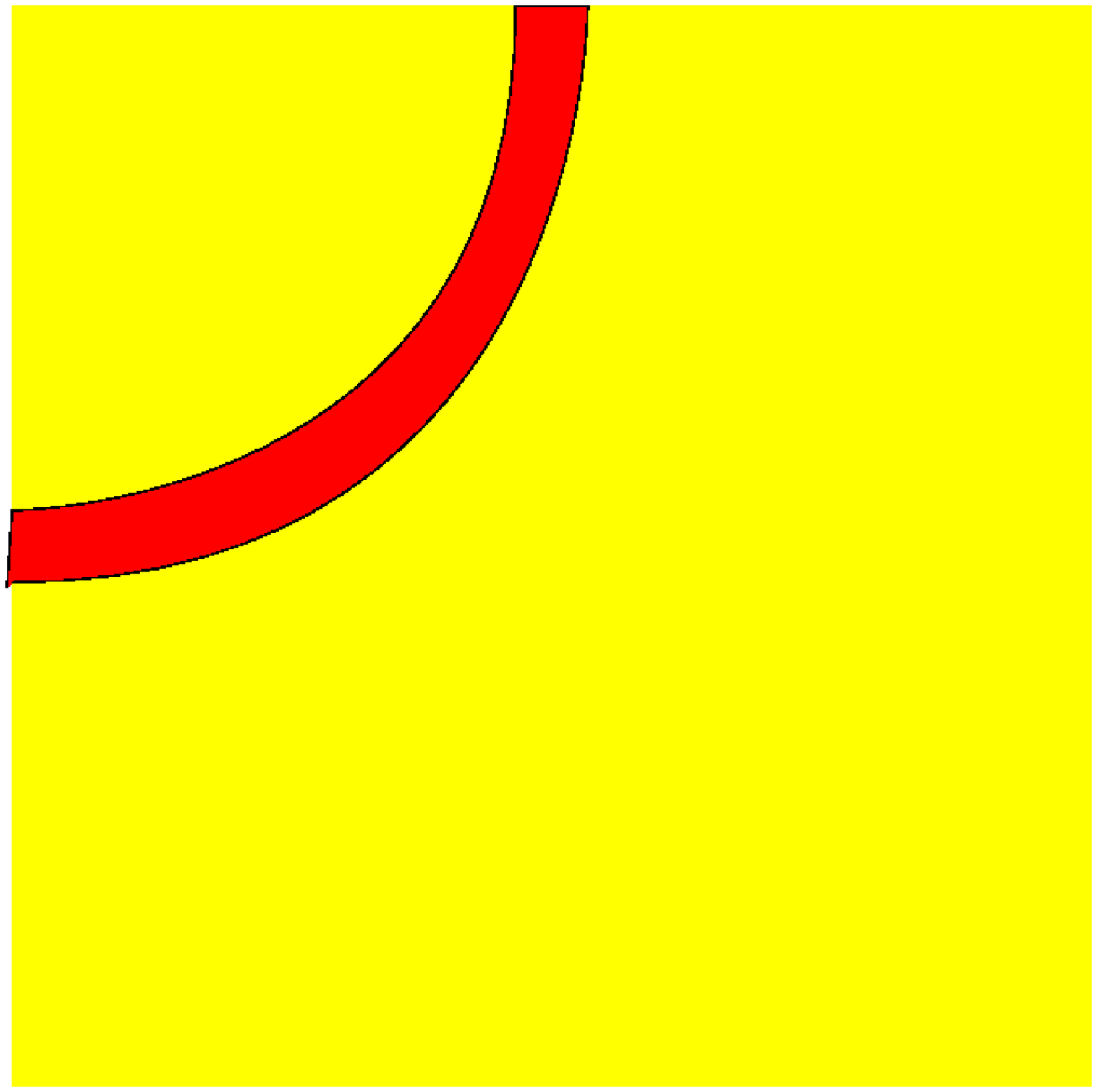}
\end{minipage}
\begin{minipage}{1.0cm}
\includegraphics[width=\hsize]{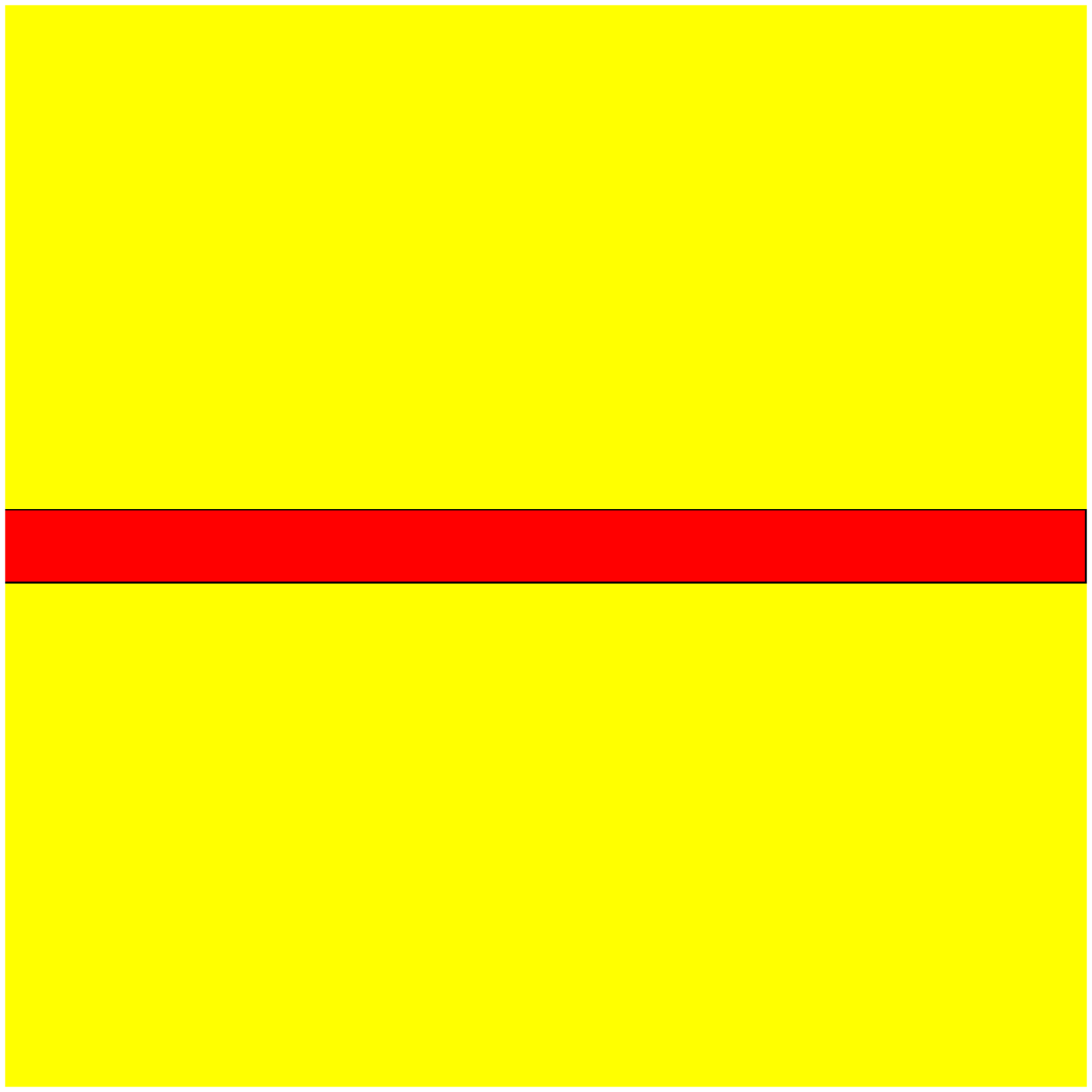}
\end{minipage}
\begin{minipage}{1.0cm}
\includegraphics[width=\hsize]{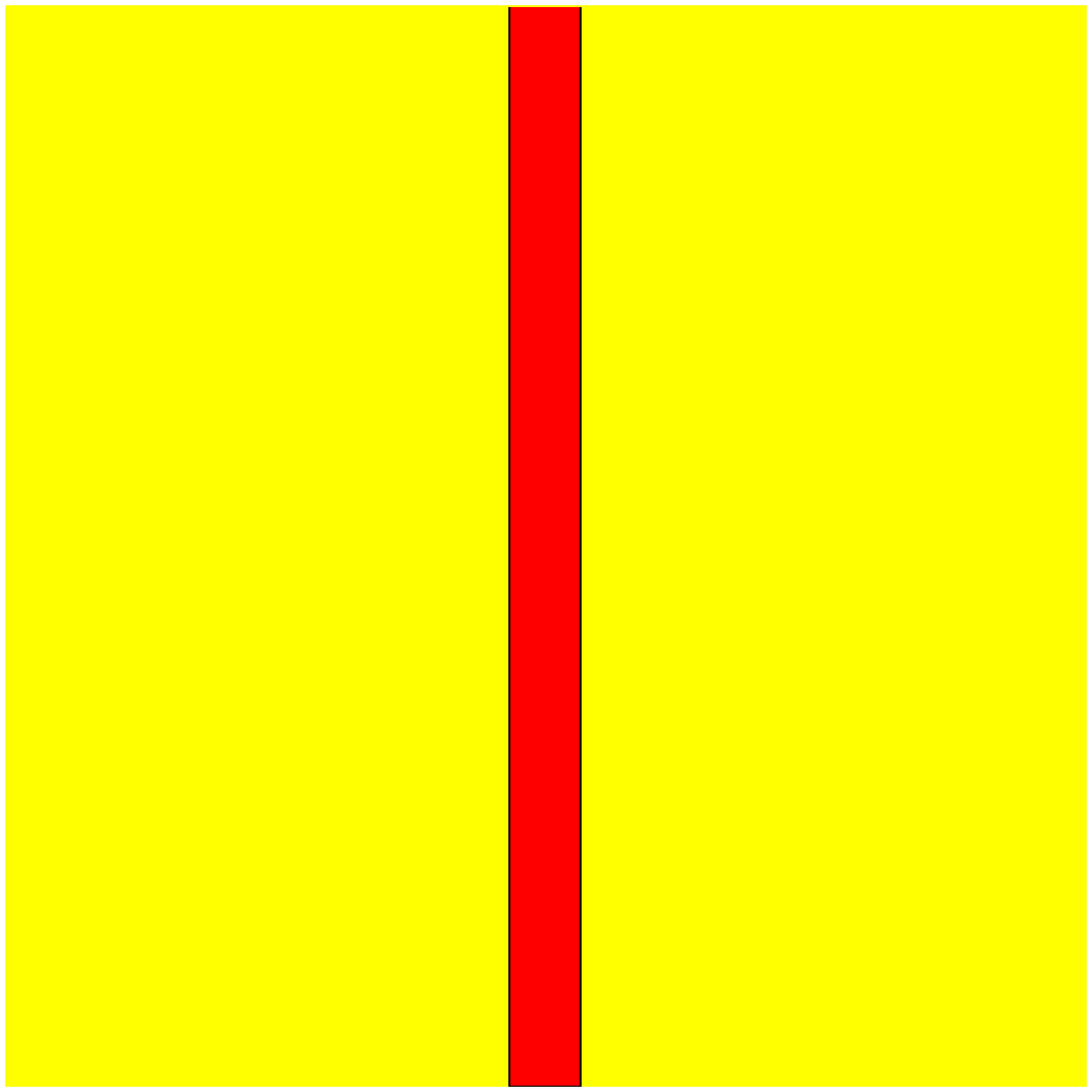}
\end{minipage}
\begin{minipage}{1.0cm}
\includegraphics[width=\hsize]{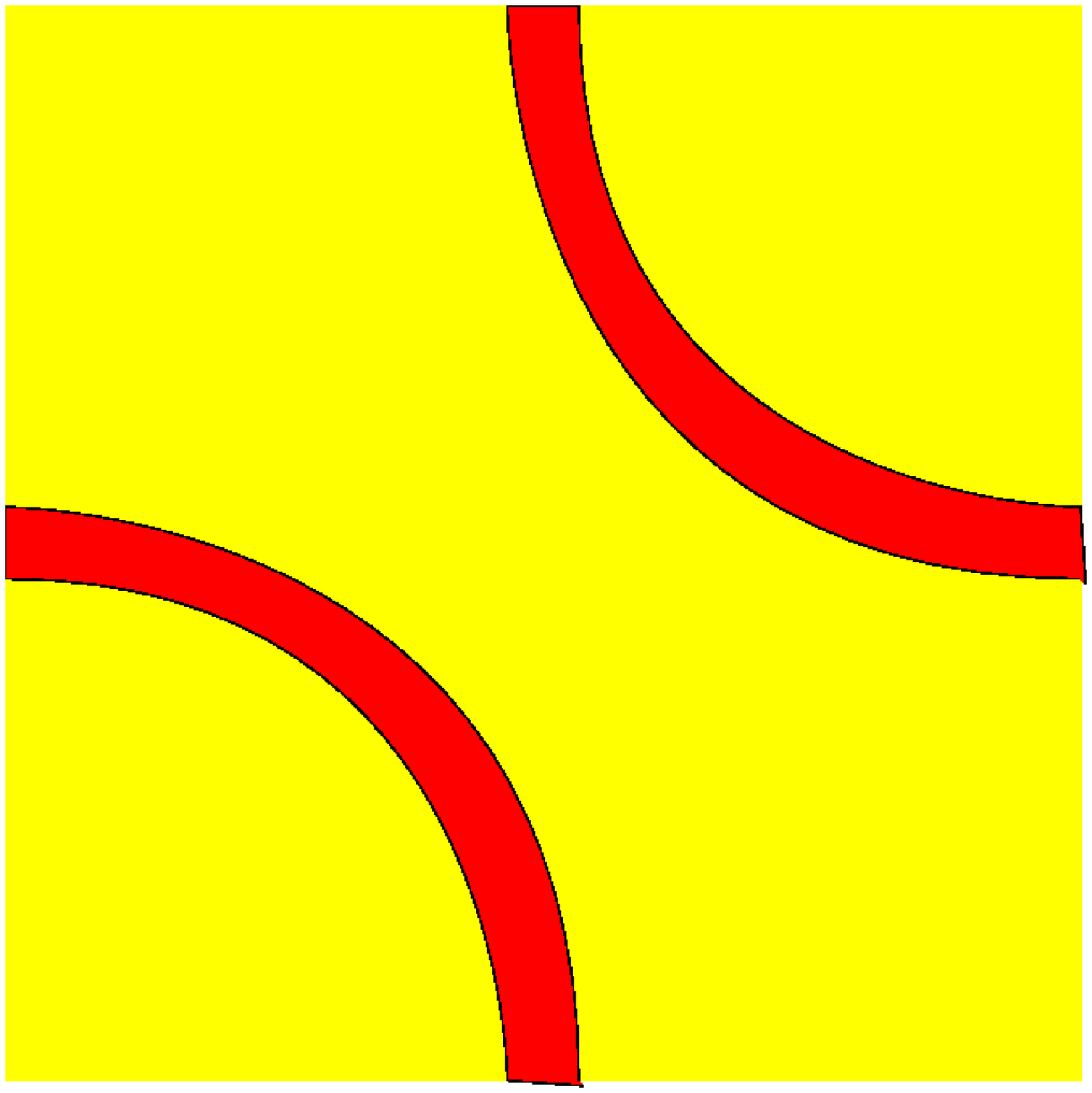}
\end{minipage}
\begin{minipage}{1.0cm}
\includegraphics[width=\hsize]{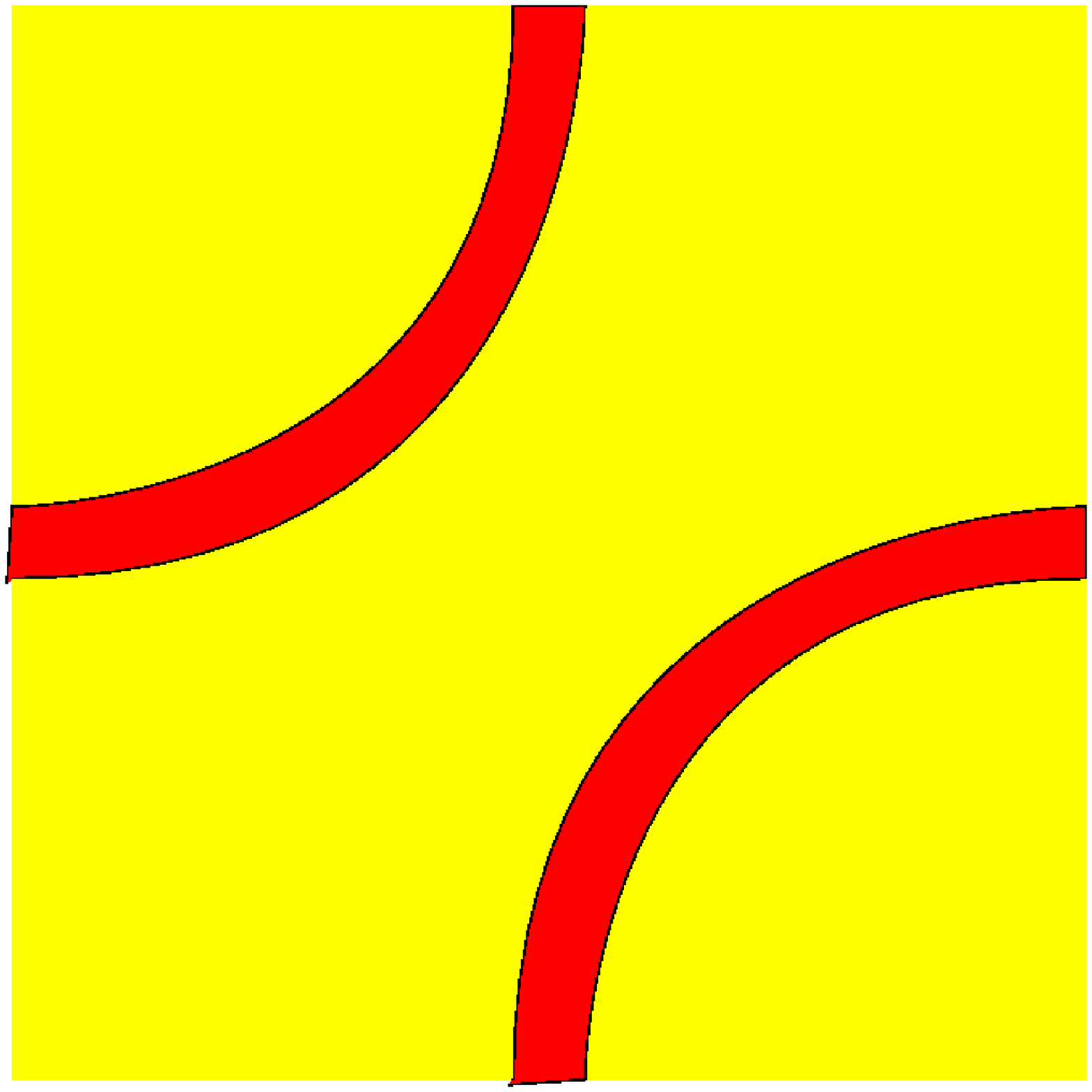}
\end{minipage}
\begin{minipage}{1.0cm}
\includegraphics[width=\hsize]{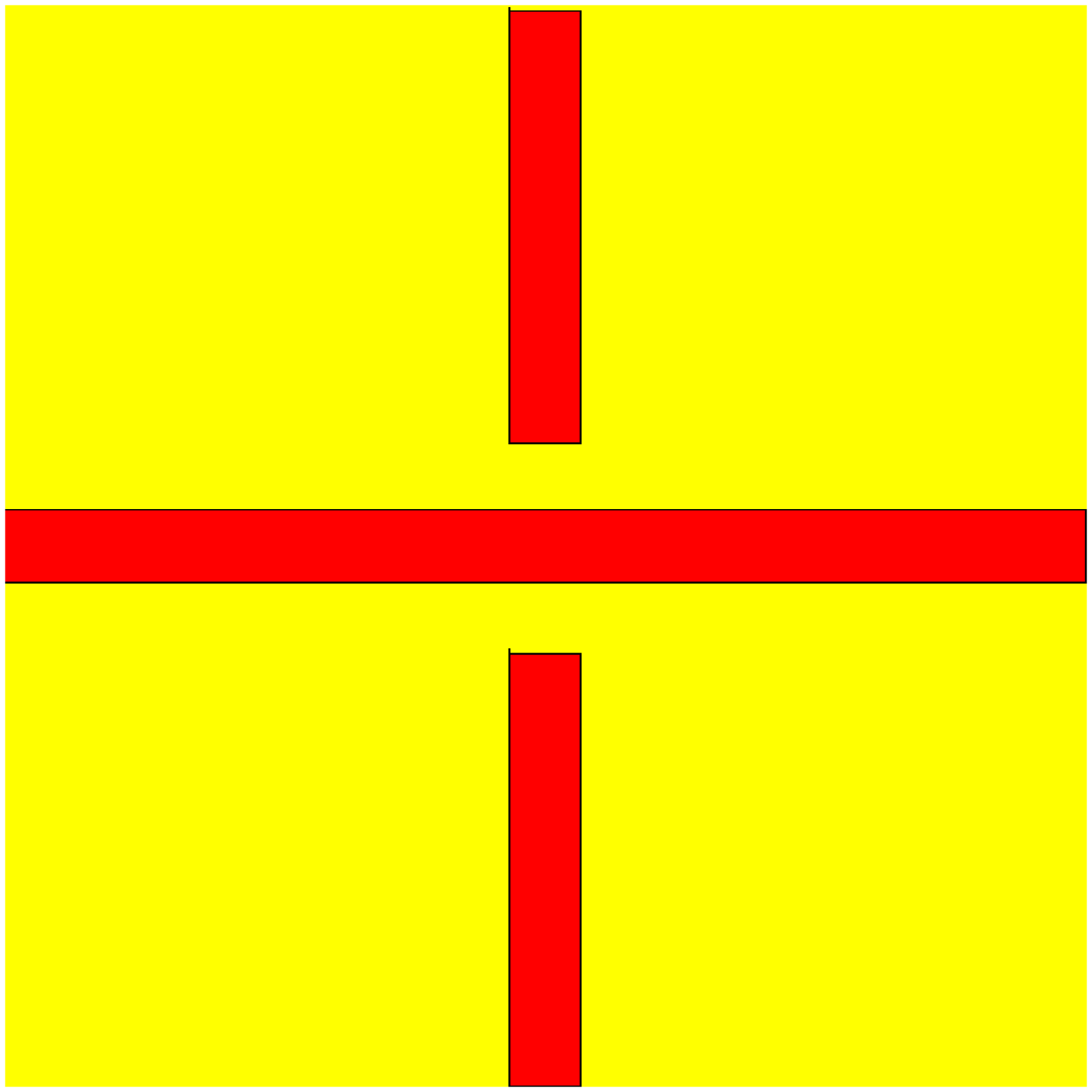}
\end{minipage}
\begin{minipage}{1.0cm}
\includegraphics[width=\hsize]{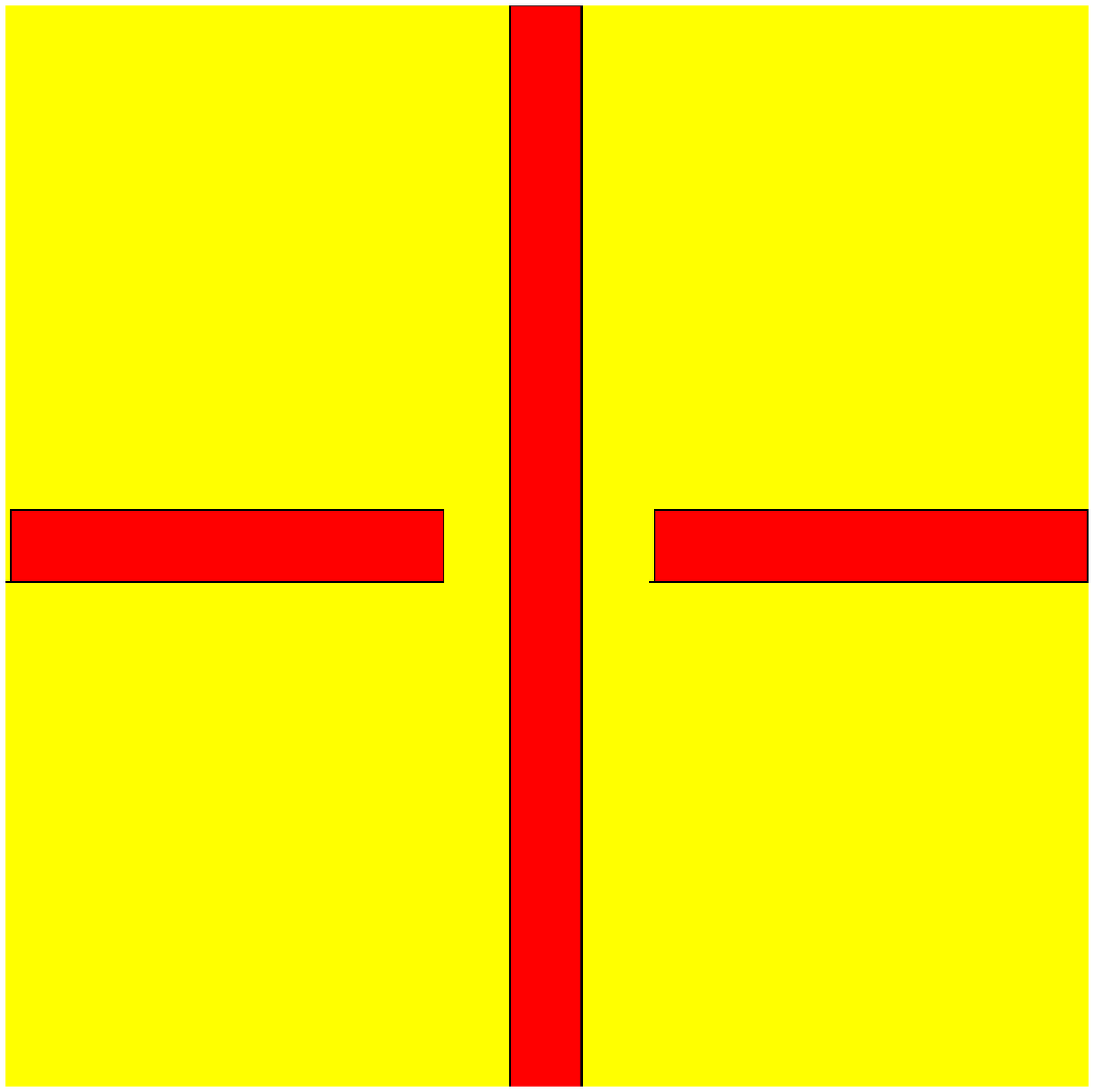}
\end{minipage}
\end{tabular}
\end{figure}
called (unoriented) tiles. We denote these tiles respectively by
the following symbols $T_0$ , $T_1$ , $T_2$ , $T_3$ , $T_4$ , $T_5$ , $T_6$ , $T_7$ , $T_8$ , $T_9$ , $T_{10}$.

\begin{defn}
Let $n$ be a positive integer. We define an (unoriented) $n$-mosaic
as an $n \times n$ matrix $M = (M_{ij}) = (T_{k(i,j)})$
of (unoriented) tiles with rows and columns indexed from $0$ to $n-1$. We denote the set of $n$-mosaics by $\mathbb{M}^{(n)}$.
\end{defn} 

\begin{defn}
A connection point of a tile is a midpoint of an edge which is also the endpoint of a curve drawn on a tile.

\begin{description}
\item[(1)] {\bf 0 connection point}
\begin{figure}[ht]
\begin{tabular}{c}
\begin{minipage}{1.0cm}
\includegraphics[width=\hsize]{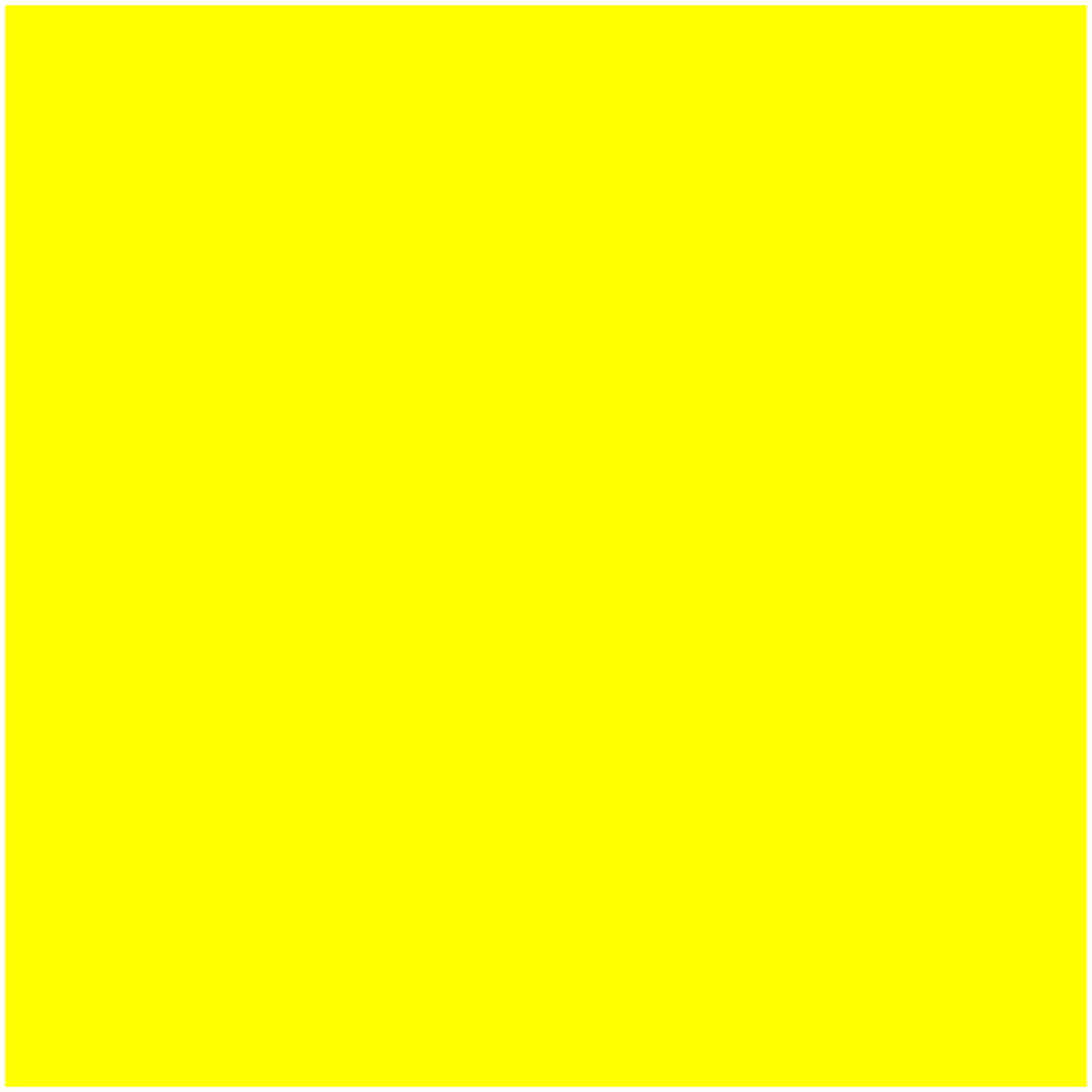}
\end{minipage}
\end{tabular}
\end{figure}

\item[(2)] {\bf 2 connection points}

\begin{figure}[ht]
\begin{tabular}{cccccc}
\begin{minipage}{1.0cm}
\includegraphics[width=\hsize]{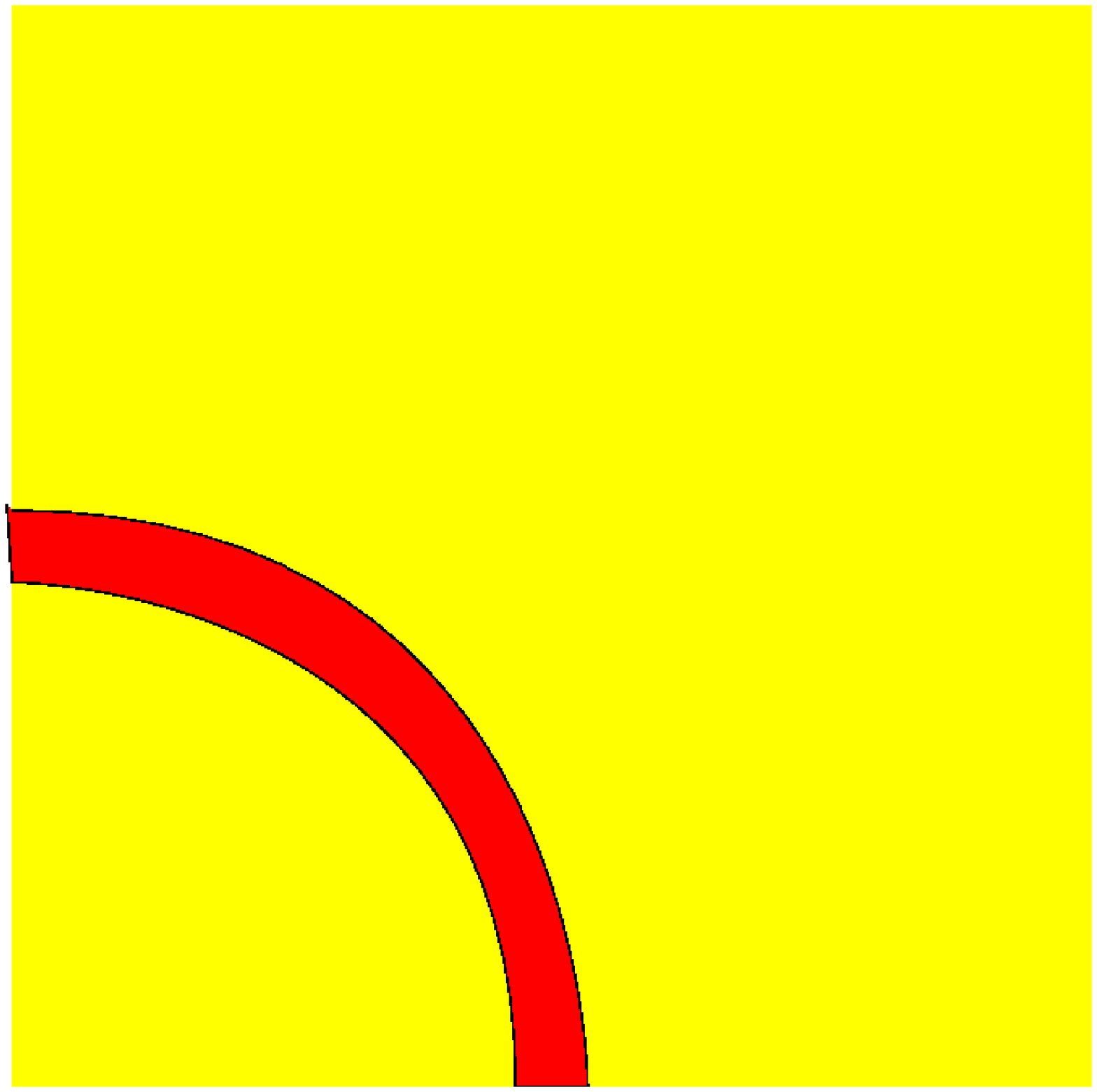}
\end{minipage}
\begin{minipage}{1.0cm}
\includegraphics[width=\hsize]{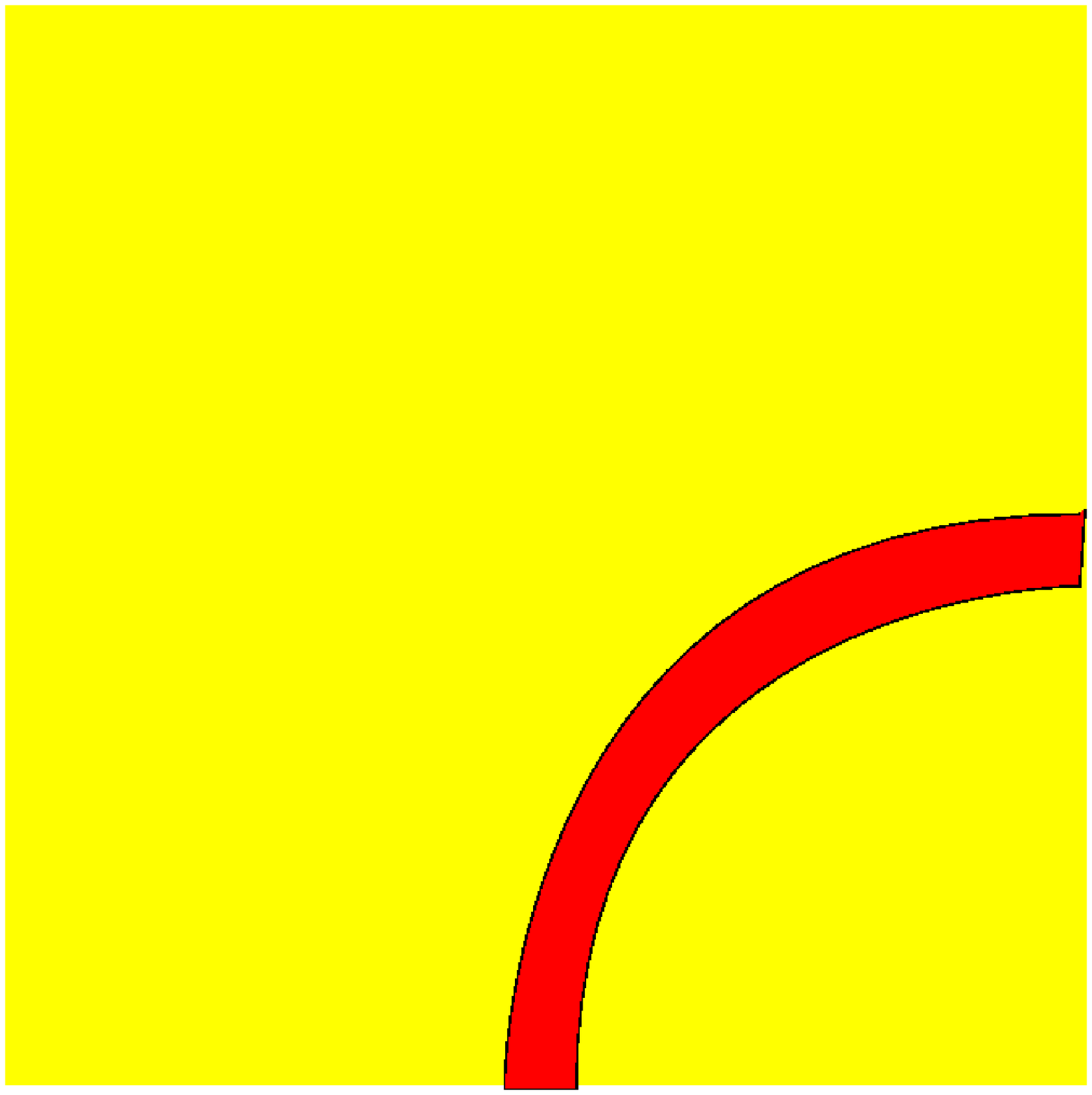}
\end{minipage}
\begin{minipage}{1.0cm}
\includegraphics[width=\hsize]{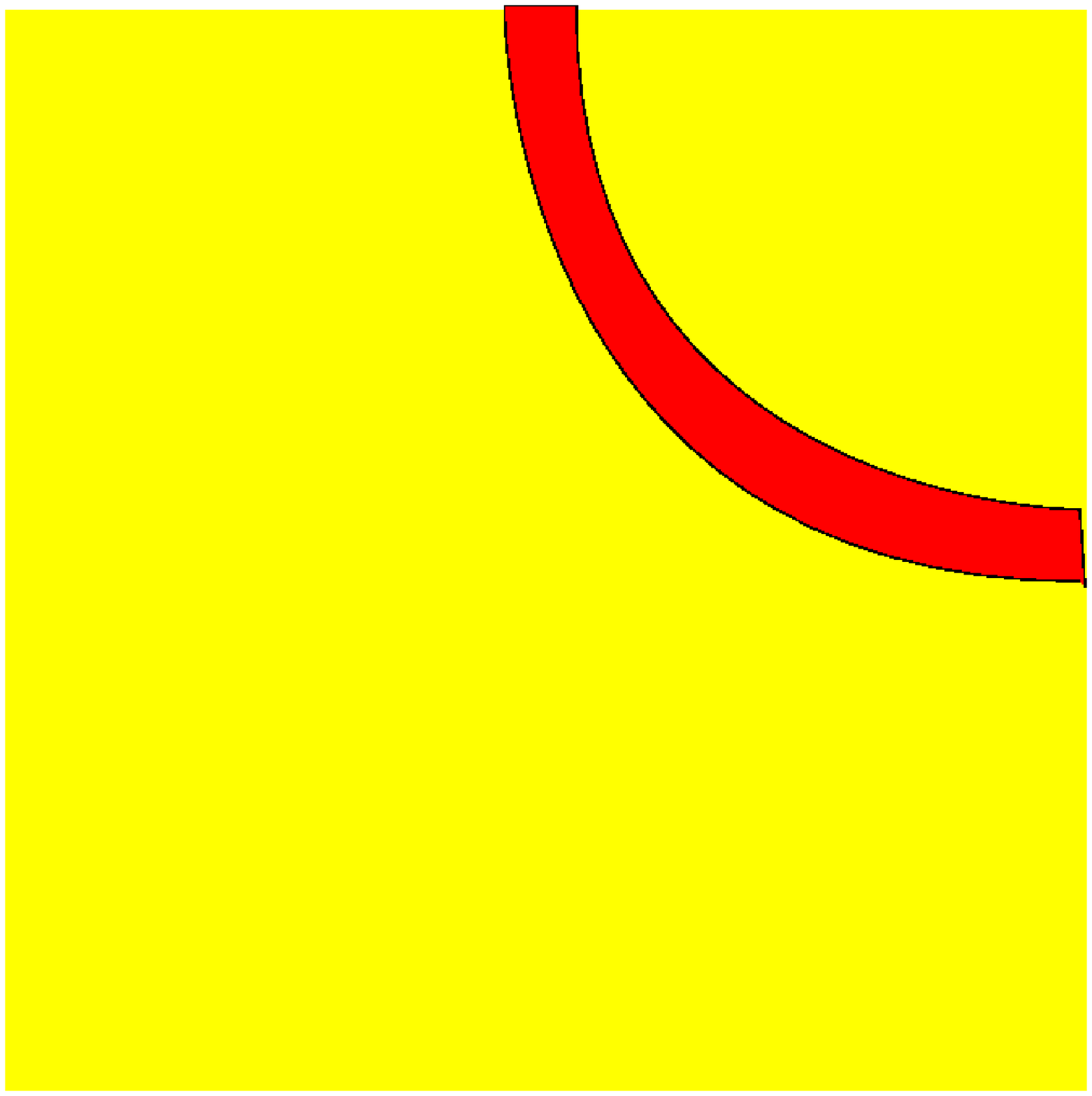}
\end{minipage}
\begin{minipage}{1.0cm}
\includegraphics[width=\hsize]{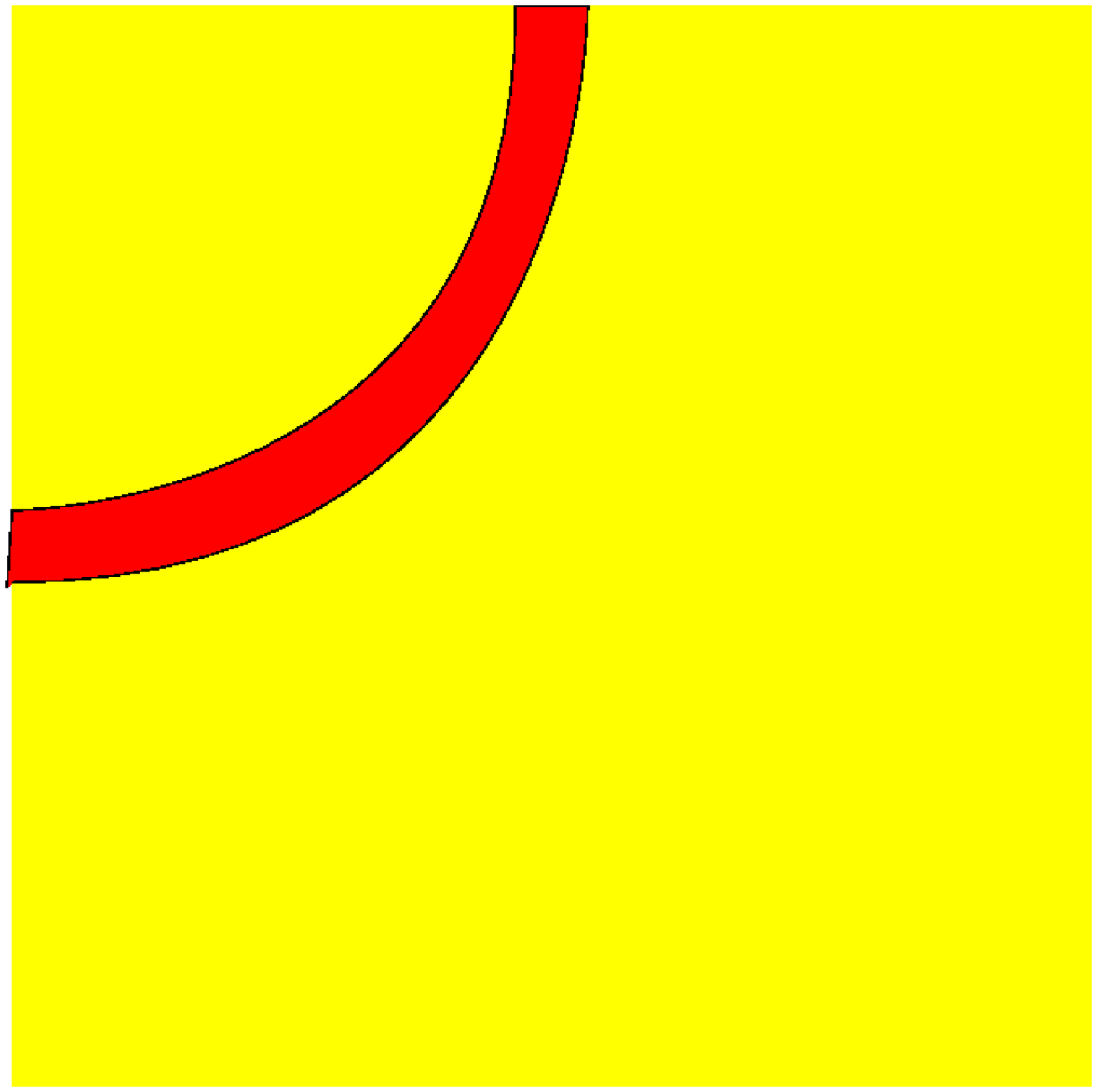}
\end{minipage}
\begin{minipage}{1.0cm}
\includegraphics[width=\hsize]{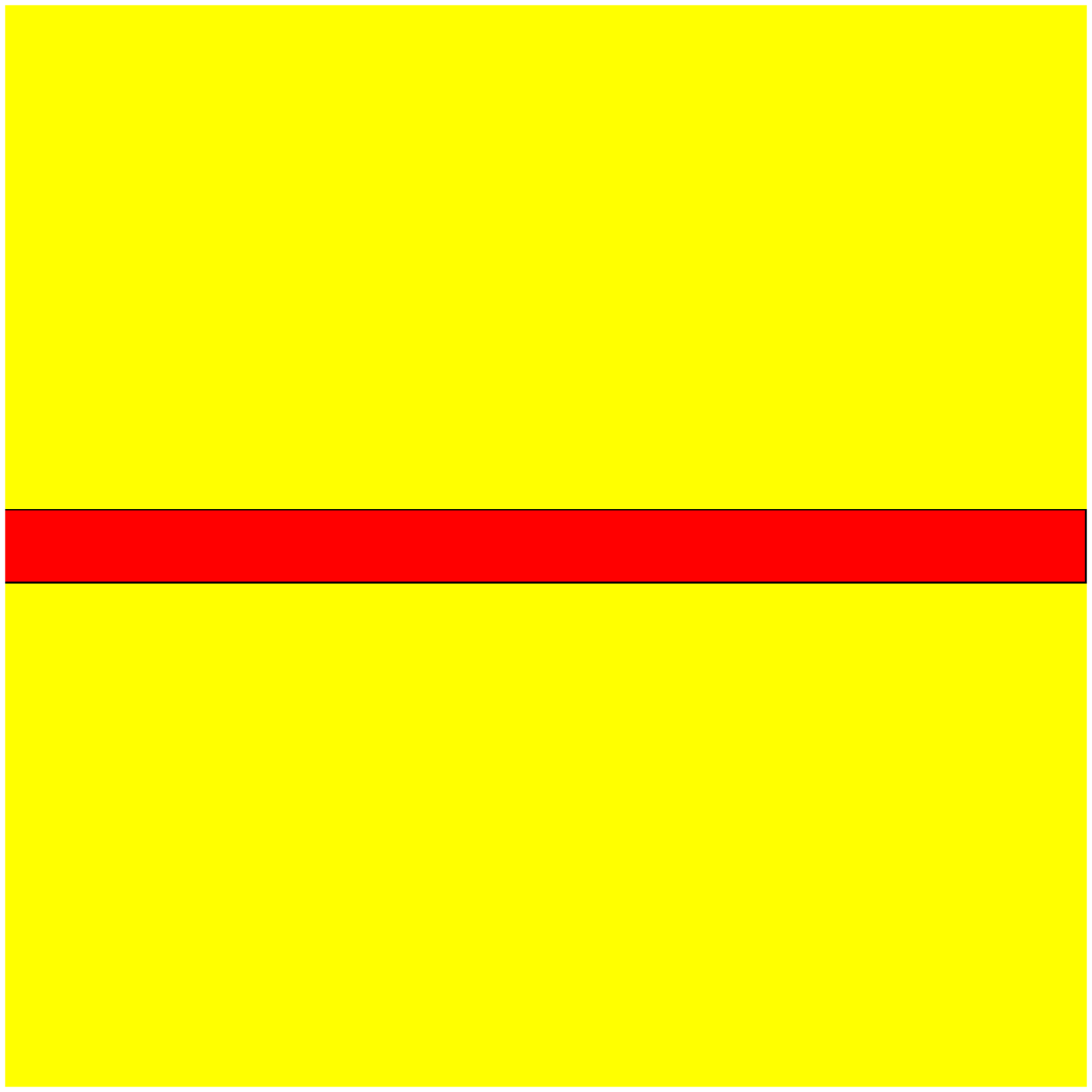}
\end{minipage}
\begin{minipage}{1.0cm}
\includegraphics[width=\hsize]{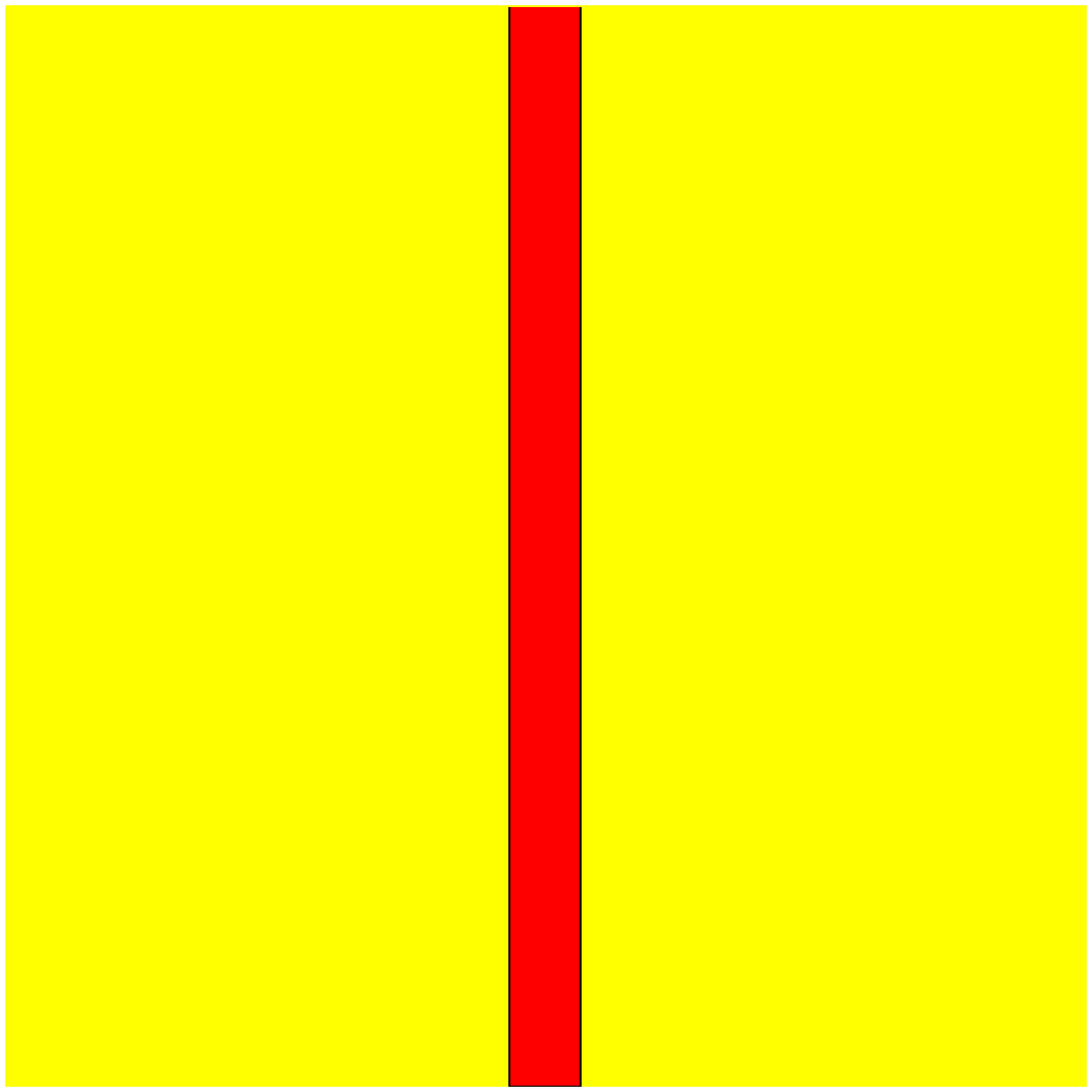}
\end{minipage}
\end{tabular}
\end{figure}

\item[(3)] {\bf 4 connection points}

\begin{figure}[ht]
\begin{tabular}{cccc}
\begin{minipage}{1.0cm}
\includegraphics[width=\hsize]{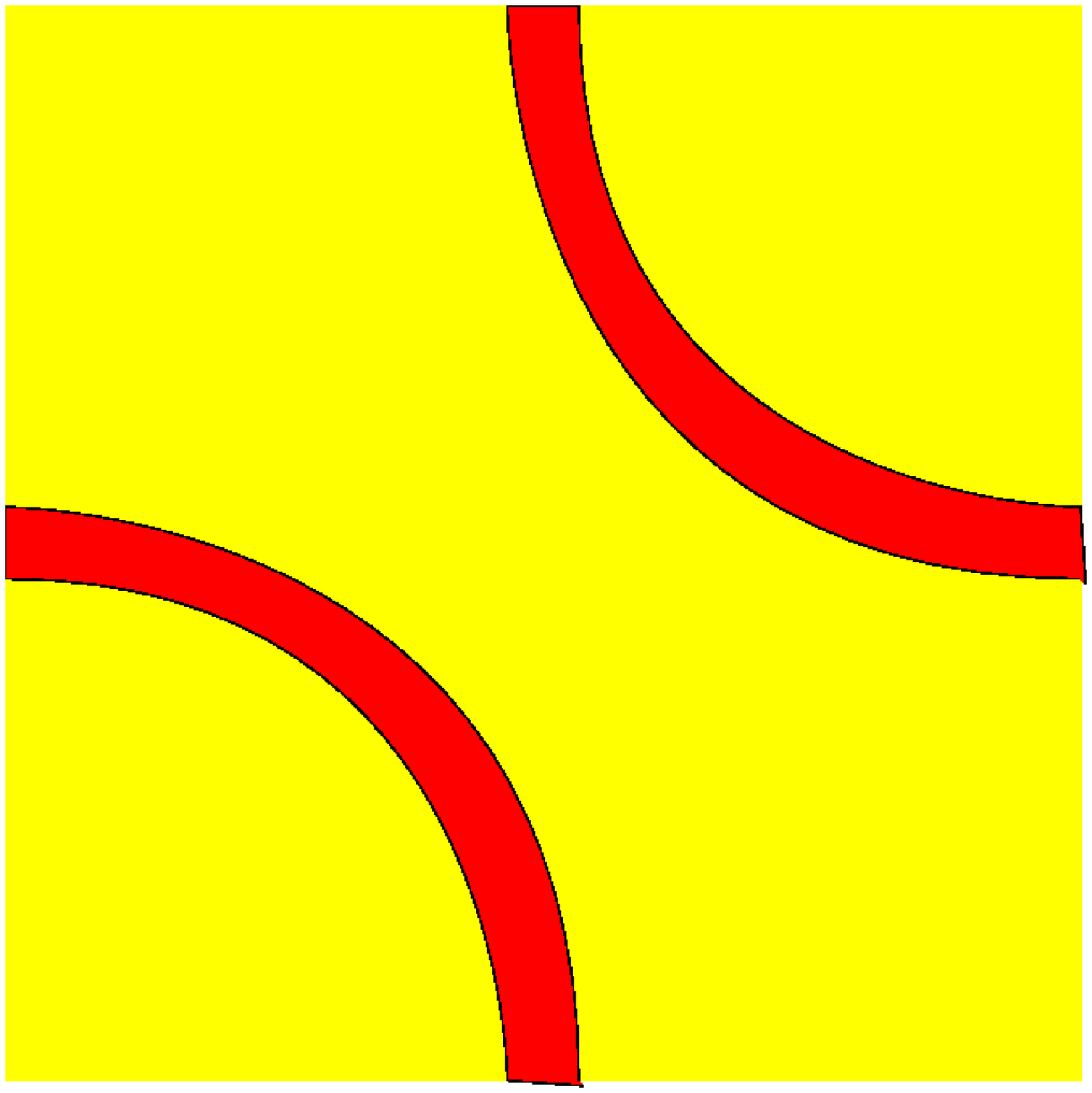}
\end{minipage}
\begin{minipage}{1.0cm}
\includegraphics[width=\hsize]{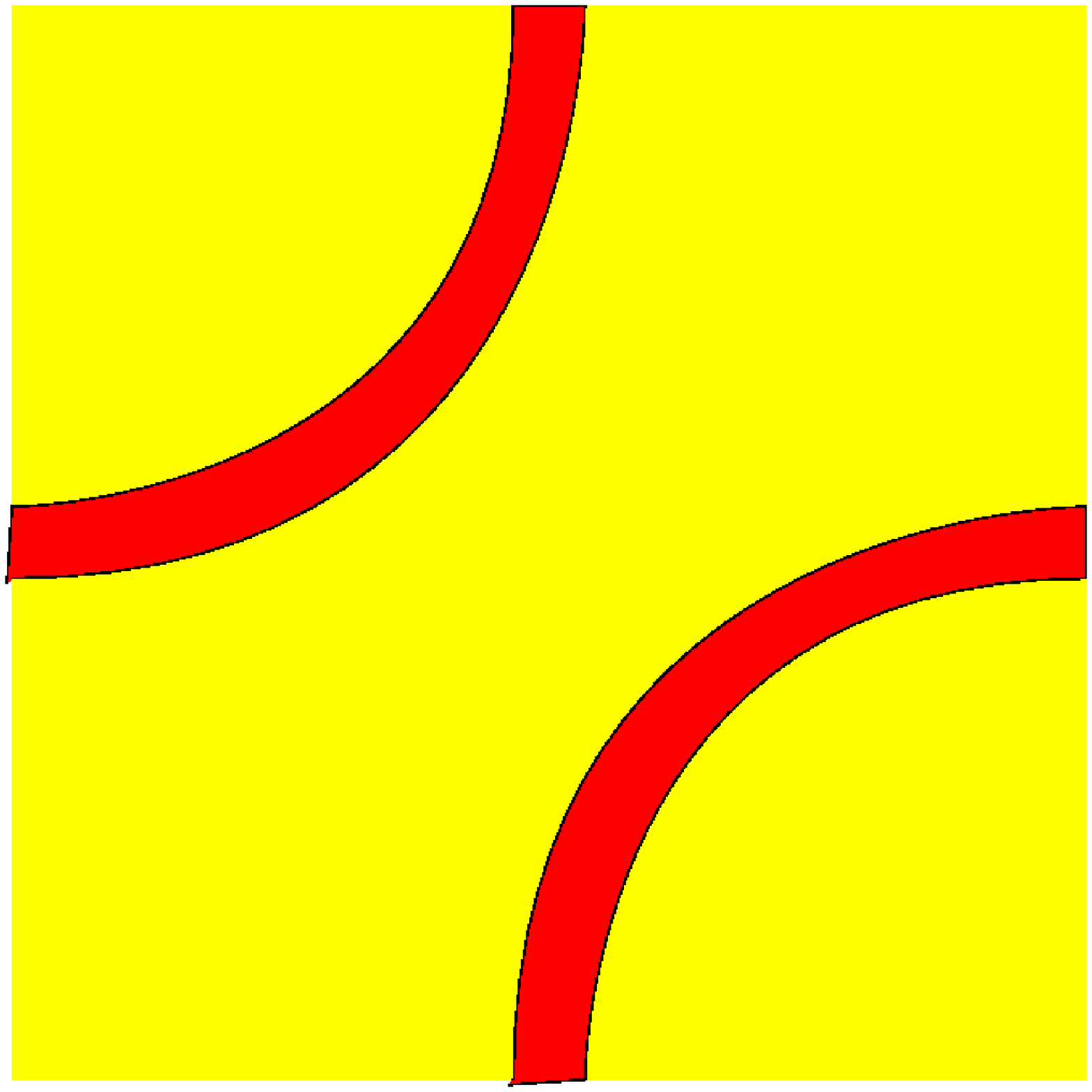}
\end{minipage}
\begin{minipage}{1.0cm}
\includegraphics[width=\hsize]{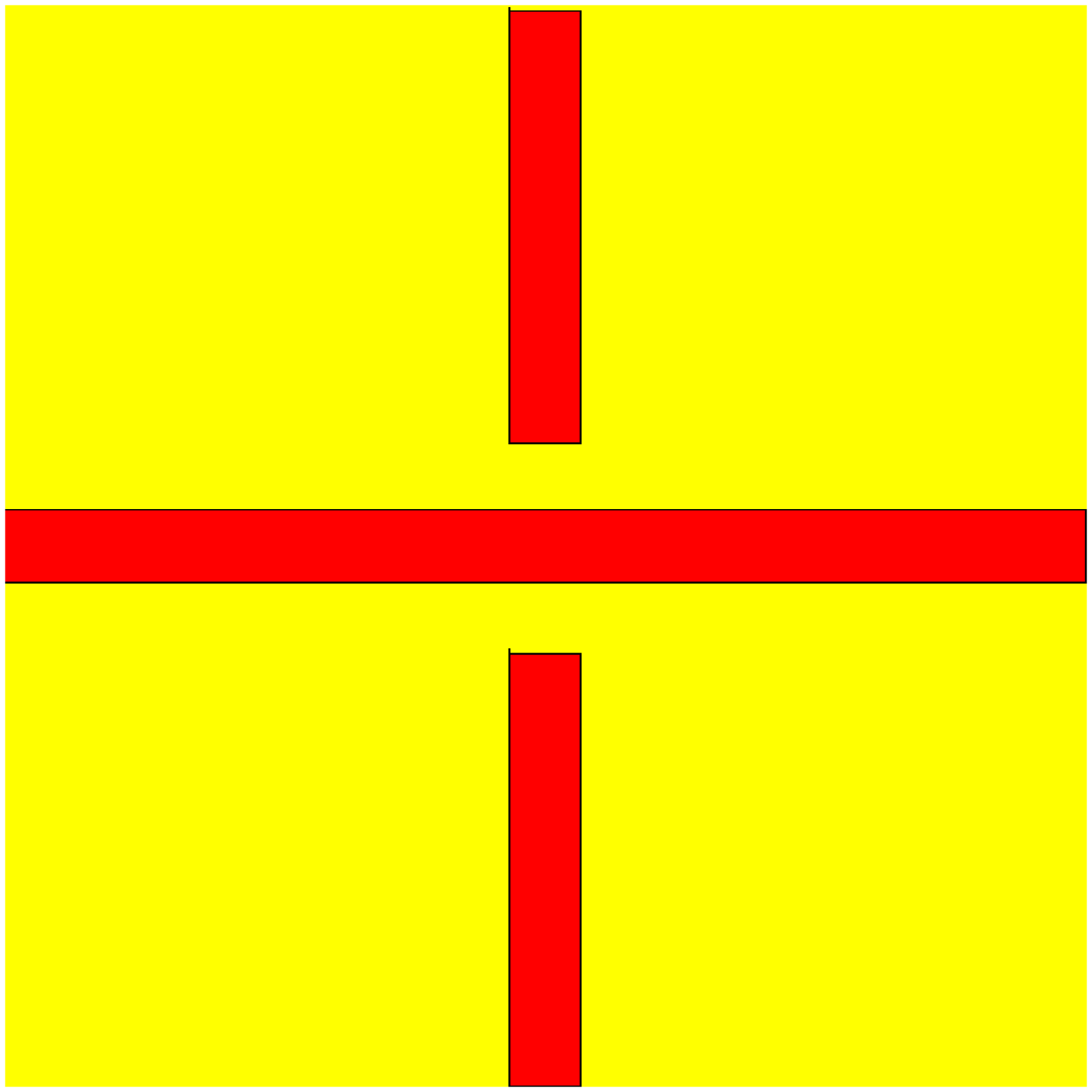}
\end{minipage}
\begin{minipage}{1.0cm}
\includegraphics[width=\hsize]{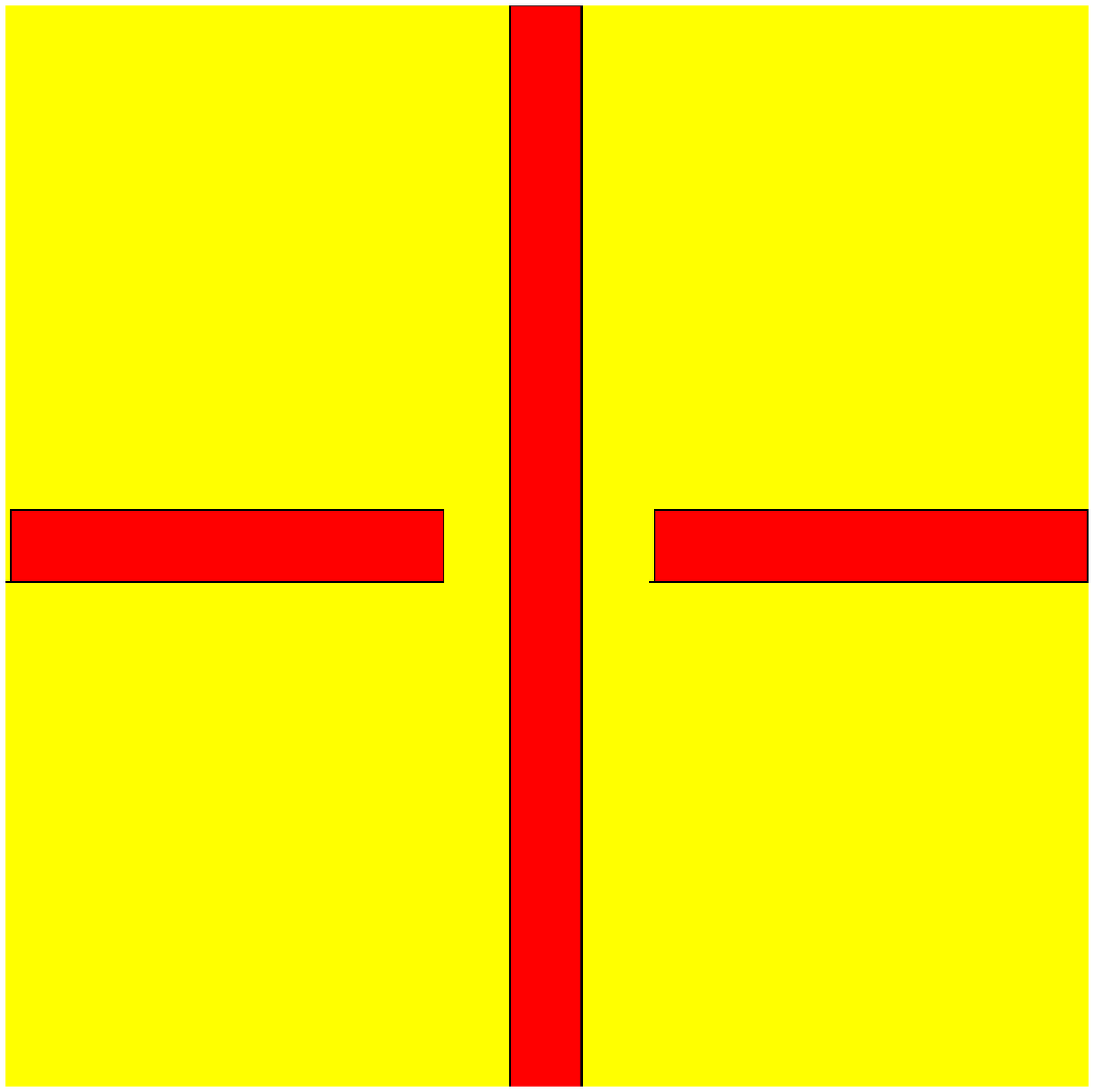}
\end{minipage}
\end{tabular}
\end{figure}
\end{description}

\end{defn}

\begin{defn}
A tile in a mosaic is said to be suitably connected if all its connection points touch the connection points of contiguous tiles.
\end{defn}

\begin{defn}
An (unoriented) knot $n$-mosaic is a mosaic in which all tiles are suitably connected. We let $\mathbb{K}^{(n)}$ denote the subset of $\mathbb{M}^{(n)}$ of all knot $n$-mosaics. For example,

\begin{figure}[ht]
\begin{minipage}{3.5cm}
\includegraphics[width=\hsize]{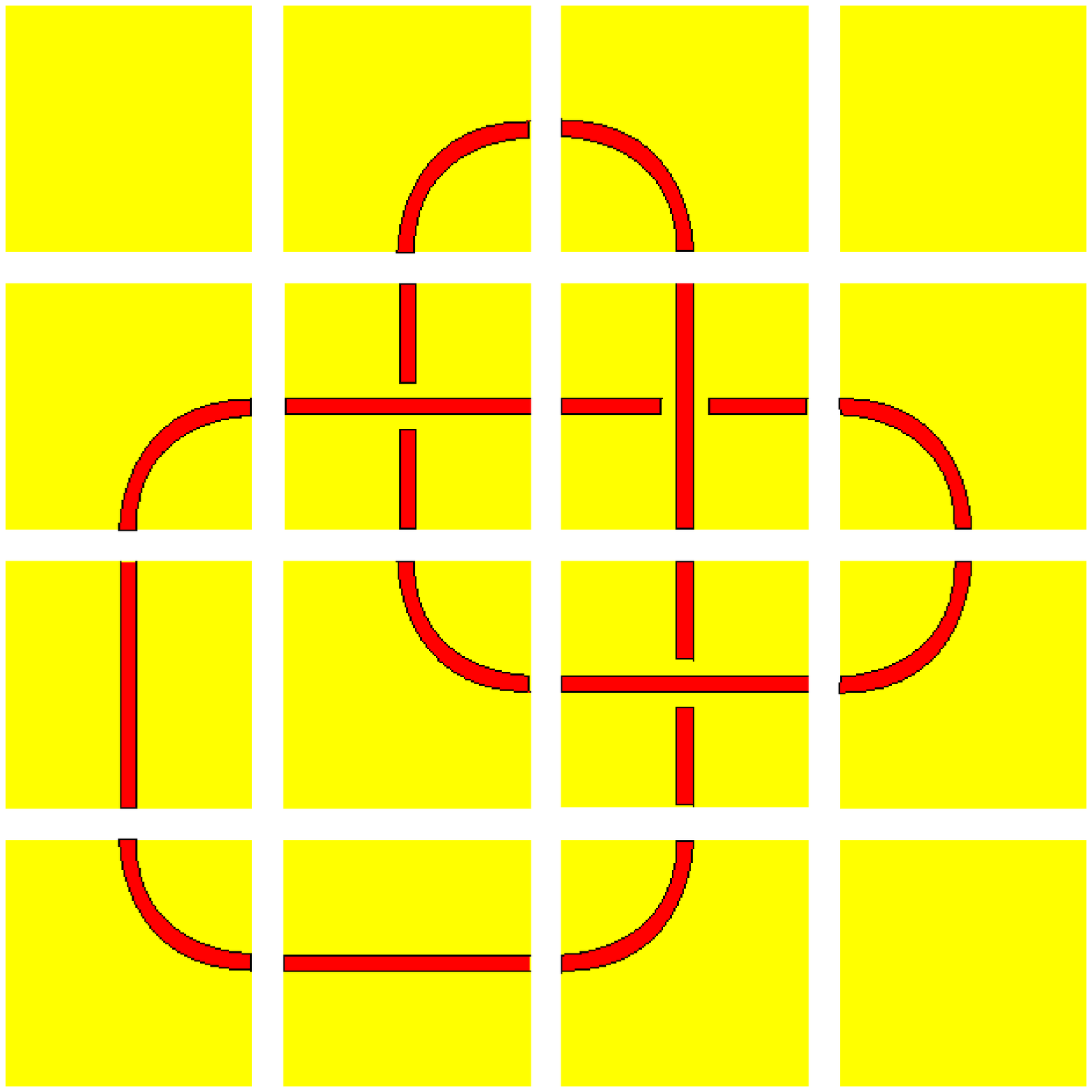}
\end{minipage}
$=
\begin{pmatrix}
 T_0 & T_2 & T_1 & T_0 \\
 T_2 & T_9 & T_{10} & T_1 \\
 T_6 & T_3 & T_9 & T_4 \\
 T_3 & T_5 & T_4 & T_0
\end{pmatrix}
=
\begin{pmatrix}
 0 & 2 & 1 & 0 \\
 2 & 9 & 10 & 1 \\
 6 & 3 & 9 & 4 \\
 3 & 5 & 4 & 0
\end{pmatrix}
\in \mathbb{K}^{(4)}$
\end{figure}

\end{defn}

\section{Mosaic moves}{\label {mosaic_move}}

\begin{defn}
Let $k$ and $n$ be positive integers such that $k\leq n$. A $k$-mosaic $N$ is
said to be a $k$\textbf{-submosaic} of an $n$-mosaic $M$ if it is a $k\times
k$ submatrix of $M$. \ The $k$-submosaic $N$ is said to be at
\textbf{location} $\left(  i,j\right)  $ in the $n$-mosaic $M$ if the top left
entry of $N$ lies in row $i$ and column $j$ of $M$. Let $\mathbf{M}^{\left(
k:i,j\right)  }$ denote the $k$\textbf{-submosaic of }$\mathbf{M}$\textbf{ at
location }$(i,j)$.
\end{defn}

\begin{defn}
Let $k$ and $n$ be positive integers such that $k\leq n$. \ For any two
$k$-mosaics $N$ and $N^{\prime}$, we define a $\mathbf{k}$\textbf{-move at
location }$\left(  \mathbf{i,j}\right)  $ on the set of $n$-mosaics
$\mathbb{M}^{(n)}$, denoted by
\[
N\overset{\left(  i,j\right)  }{\longleftrightarrow}N^{\prime}\text{ ,}%
\]
as the map from $\mathbb{M}^{(n)}$ to $\mathbb{M}^{(n)}$ defined by%
\[
\left(  N\overset{\left(  i,j\right)  }{\longleftrightarrow}N^{\prime}\right)
\left(  M\right)  =\left\{
\begin{array}
[c]{ll}%
M\text{ with }M^{\left(  k:i,j\right)  }\text{ replaced by }N^{\prime} &
\text{if }M^{\left(  k:i,j\right)  }=N\\
M\text{ with }M^{\left(  k:i,j\right)  }\text{ replaced by }N & \text{if
}M^{\left(  k:i,j\right)  }=N^{\prime}\\
M & \text{otherwise}%
\end{array}
\right.
\]

\end{defn}

\begin{pro}{\label{LK1}}
Each $k$-move $N\overset{\left(  i,j\right)  }{\longleftrightarrow}N^{\prime}$
is a permutation of $\mathbb{M}^{(n)}$. \ In fact, it is a permutation which
is the product of disjoint transpositions.
\end{pro}

\begin{defn}
As an analog to the planar isotopy moves for standard knot diagrams, we define
for mosaics the 11 \textbf{mosaic planar isotopy moves} given below:

\[
\begin{array}
[c]{cc}%
{\includegraphics[
height=0.3269in,
width=0.3269in
]%
{T_2.eps}
}%
&
{\includegraphics[
height=0.3269in,
width=0.3269in
]%
{T_4.eps}
}%
\\%
{\includegraphics[
height=0.3269in,
width=0.3269in
]%
{T_4.eps}%
}%
&
{\includegraphics[
height=0.3269in,
width=0.3269in
]%
{T_0.eps}%
}%
\end{array}
\Longleftrightarrow _{P_1}
\begin{array}
[c]{cc}%
{\includegraphics[
height=0.3269in,
width=0.3269in
]%
{T_0.eps}
}%
&
{\includegraphics[
height=0.3269in,
width=0.3269in
]%
{T_6.eps}
}%
\\%
{\includegraphics[
height=0.3269in,
width=0.3269in
]%
{T_5.eps}%
}%
&
{\includegraphics[
height=0.3269in,
width=0.3269in
]%
{T_4.eps}%
}%
\end{array}
\]

\[
\begin{array}
[c]{cc}%
{\includegraphics[
height=0.3269in,
width=0.3269in
]%
{T_5.eps}
}%
&
{\includegraphics[
height=0.3269in,
width=0.3269in
]%
{T_1.eps}
}%
\\%
{\includegraphics[
height=0.3269in,
width=0.3269in
]%
{T_2.eps}%
}%
&
{\includegraphics[
height=0.3269in,
width=0.3269in
]%
{T_4.eps}%
}%
\end{array}
\Longleftrightarrow _{P_2}
\begin{array}
[c]{cc}%
{\includegraphics[
height=0.3269in,
width=0.3269in
]%
{T_1.eps}
}%
&
{\includegraphics[
height=0.3269in,
width=0.3269in
]%
{T_0.eps}
}%
\\%
{\includegraphics[
height=0.3269in,
width=0.3269in
]%
{T_6.eps}%
}%
&
{\includegraphics[
height=0.3269in,
width=0.3269in
]%
{T_0.eps}%
}%
\end{array}
\]

\[
\begin{array}
[c]{cc}%
{\includegraphics[
height=0.3269in,
width=0.3269in
]%
{T_2.eps}
}%
&
{\includegraphics[
height=0.3269in,
width=0.3269in
]%
{T_1.eps}
}%
\\%
{\includegraphics[
height=0.3269in,
width=0.3269in
]%
{T_4.eps}%
}%
&
{\includegraphics[
height=0.3269in,
width=0.3269in
]%
{T_6.eps}%
}%
\end{array}
\Longleftrightarrow _{P_3}
\begin{array}
[c]{cc}%
{\includegraphics[
height=0.3269in,
width=0.3269in
]%
{T_0.eps}
}%
&
{\includegraphics[
height=0.3269in,
width=0.3269in
]%
{T_0.eps}
}%
\\%
{\includegraphics[
height=0.3269in,
width=0.3269in
]%
{T_5.eps}%
}%
&
{\includegraphics[
height=0.3269in,
width=0.3269in
]%
{T_1.eps}%
}%
\end{array}
\]

\[
\begin{array}
[c]{cc}%
{\includegraphics[
height=0.3269in,
width=0.3269in
]%
{T_0.eps}
}%
&
{\includegraphics[
height=0.3269in,
width=0.3269in
]%
{T_2.eps}
}%
\\%
{\includegraphics[
height=0.3269in,
width=0.3269in
]%
{T_5.eps}%
}%
&
{\includegraphics[
height=0.3269in,
width=0.3269in
]%
{T_4.eps}%
}%
\end{array}
\Longleftrightarrow _{P_4}
\begin{array}
[c]{cc}%
{\includegraphics[
height=0.3269in,
width=0.3269in
]%
{T_2.eps}
}%
&
{\includegraphics[
height=0.3269in,
width=0.3269in
]%
{T_5.eps}
}%
\\%
{\includegraphics[
height=0.3269in,
width=0.3269in
]%
{T_4.eps}%
}%
&
{\includegraphics[
height=0.3269in,
width=0.3269in
]%
{T_0.eps}%
}%
\end{array}
\]

\[
\begin{array}
[c]{cc}%
{\includegraphics[
height=0.3269in,
width=0.3269in
]%
{T_0.eps}
}%
&
{\includegraphics[
height=0.3269in,
width=0.3269in
]%
{T_0.eps}
}%
\\%
{\includegraphics[
height=0.3269in,
width=0.3269in
]%
{T_5.eps}%
}%
&
{\includegraphics[
height=0.3269in,
width=0.3269in
]%
{T_5.eps}%
}%
\end{array}
\Longleftrightarrow _{P_5}
\begin{array}
[c]{cc}%
{\includegraphics[
height=0.3269in,
width=0.3269in
]%
{T_2.eps}
}%
&
{\includegraphics[
height=0.3269in,
width=0.3269in
]%
{T_1.eps}
}%
\\%
{\includegraphics[
height=0.3269in,
width=0.3269in
]%
{T_4.eps}%
}%
&
{\includegraphics[
height=0.3269in,
width=0.3269in
]%
{T_3.eps}%
}%
\end{array}
\]

\[
\begin{array}
[c]{cc}%
{\includegraphics[
height=0.3269in,
width=0.3269in
]%
{T_1.eps}
}%
&
{\includegraphics[
height=0.3269in,
width=0.3269in
]%
{T_0.eps}
}%
\\%
{\includegraphics[
height=0.3269in,
width=0.3269in
]%
{T_4.eps}%
}%
&
{\includegraphics[
height=0.3269in,
width=0.3269in
]%
{T_0.eps}%
}%
\end{array}
\Longleftrightarrow _{P_6}
\begin{array}
[c]{cc}%
{\includegraphics[
height=0.3269in,
width=0.3269in
]%
{T_5.eps}
}%
&
{\includegraphics[
height=0.3269in,
width=0.3269in
]%
{T_1.eps}
}%
\\%
{\includegraphics[
height=0.3269in,
width=0.3269in
]%
{T_5.eps}%
}%
&
{\includegraphics[
height=0.3269in,
width=0.3269in
]%
{T_4.eps}%
}%
\end{array}
\]

\[
\begin{array}
[c]{cc}%
{\includegraphics[
height=0.3269in,
width=0.3269in
]%
{T_0.eps}
}%
&
{\includegraphics[
height=0.3269in,
width=0.3269in
]%
{T_0.eps}
}%
\\%
{\includegraphics[
height=0.3269in,
width=0.3269in
]%
{T_1.eps}%
}%
&
{\includegraphics[
height=0.3269in,
width=0.3269in
]%
{T_0.eps}%
}%
\end{array}
\Longleftrightarrow _{P_7}
\begin{array}
[c]{cc}%
{\includegraphics[
height=0.3269in,
width=0.3269in
]%
{T_2.eps}
}%
&
{\includegraphics[
height=0.3269in,
width=0.3269in
]%
{T_1.eps}
}%
\\%
{\includegraphics[
height=0.3269in,
width=0.3269in
]%
{T_8.eps}%
}%
&
{\includegraphics[
height=0.3269in,
width=0.3269in
]%
{T_4.eps}%
}%
\end{array}
\]

\[
\begin{array}
[c]{cc}%
{\includegraphics[
height=0.3269in,
width=0.3269in
]%
{T_2.eps}
}%
&
{\includegraphics[
height=0.3269in,
width=0.3269in
]%
{T_8.eps}
}%
\\%
{\includegraphics[
height=0.3269in,
width=0.3269in
]%
{T_9.eps}%
}%
&
{\includegraphics[
height=0.3269in,
width=0.3269in
]%
{T_4.eps}%
}%
\end{array}
\Longleftrightarrow _{P_8}
\begin{array}
[c]{cc}%
{\includegraphics[
height=0.3269in,
width=0.3269in
]%
{T_2.eps}
}%
&
{\includegraphics[
height=0.3269in,
width=0.3269in
]%
{T_9.eps}
}%
\\%
{\includegraphics[
height=0.3269in,
width=0.3269in
]%
{T_8.eps}%
}%
&
{\includegraphics[
height=0.3269in,
width=0.3269in
]%
{T_4.eps}%
}%
\end{array}
\]

\[
\begin{array}
[c]{cc}%
{\includegraphics[
height=0.3269in,
width=0.3269in
]%
{T_2.eps}
}%
&
{\includegraphics[
height=0.3269in,
width=0.3269in
]%
{T_8.eps}
}%
\\%
{\includegraphics[
height=0.3269in,
width=0.3269in
]%
{T_10.eps}%
}%
&
{\includegraphics[
height=0.3269in,
width=0.3269in
]%
{T_4.eps}%
}%
\end{array}
\Longleftrightarrow _{P_9}
\begin{array}
[c]{cc}%
{\includegraphics[
height=0.3269in,
width=0.3269in
]%
{T_2.eps}
}%
&
{\includegraphics[
height=0.3269in,
width=0.3269in
]%
{T_10.eps}
}%
\\%
{\includegraphics[
height=0.3269in,
width=0.3269in
]%
{T_8.eps}%
}%
&
{\includegraphics[
height=0.3269in,
width=0.3269in
]%
{T_4.eps}%
}%
\end{array}
\]

\[
\begin{array}
[c]{cc}%
{\includegraphics[
height=0.3269in,
width=0.3269in
]%
{T_7.eps}
}%
&
{\includegraphics[
height=0.3269in,
width=0.3269in
]%
{T_9.eps}
}%
\\%
{\includegraphics[
height=0.3269in,
width=0.3269in
]%
{T_3.eps}%
}%
&
{\includegraphics[
height=0.3269in,
width=0.3269in
]%
{T_4.eps}%
}%
\end{array}
\Longleftrightarrow _{P_{10}}
\begin{array}
[c]{cc}%
{\includegraphics[
height=0.3269in,
width=0.3269in
]%
{T_10.eps}
}%
&
{\includegraphics[
height=0.3269in,
width=0.3269in
]%
{T_8.eps}
}%
\\%
{\includegraphics[
height=0.3269in,
width=0.3269in
]%
{T_3.eps}%
}%
&
{\includegraphics[
height=0.3269in,
width=0.3269in
]%
{T_4.eps}%
}%
\end{array}
\]

\[
\begin{array}
[c]{cc}%
{\includegraphics[
height=0.3269in,
width=0.3269in
]%
{T_7.eps}
}%
&
{\includegraphics[
height=0.3269in,
width=0.3269in
]%
{T_10.eps}
}%
\\%
{\includegraphics[
height=0.3269in,
width=0.3269in
]%
{T_3.eps}%
}%
&
{\includegraphics[
height=0.3269in,
width=0.3269in
]%
{T_4.eps}%
}%
\end{array}
\Longleftrightarrow _{P_{11}}
\begin{array}
[c]{cc}%
{\includegraphics[
height=0.3269in,
width=0.3269in
]%
{T_9.eps}
}%
&
{\includegraphics[
height=0.3269in,
width=0.3269in
]%
{T_8.eps}
}%
\\%
{\includegraphics[
height=0.3269in,
width=0.3269in
]%
{T_3.eps}%
}%
&
{\includegraphics[
height=0.3269in,
width=0.3269in
]%
{T_4.eps}%
}%
\end{array}
\]

\end{defn}

\begin{defn}
As an analog to the Reidemeister moves for standard knot diagrams, we create
for mosaics the \textbf{mosaic Reidemeister moves}. 

The \textbf{mosaic Reidemeister 1 moves} are the following:
\[
\begin{array}
[c]{cc}%
{\includegraphics[
height=0.3269in,
width=0.3269in
]%
{T_2.eps}
}%
&
{\includegraphics[
height=0.3269in,
width=0.3269in
]%
{T_1.eps}
}%
\\%
{\includegraphics[
height=0.3269in,
width=0.3269in
]%
{T_8.eps}%
}%
&
{\includegraphics[
height=0.3269in,
width=0.3269in
]%
{T_4.eps}%
}%
\end{array}
\Longleftrightarrow _{R_1}
\begin{array}
[c]{cc}%
{\includegraphics[
height=0.3269in,
width=0.3269in
]%
{T_2.eps}
}%
&
{\includegraphics[
height=0.3269in,
width=0.3269in
]%
{T_1.eps}
}%
\\%
{\includegraphics[
height=0.3269in,
width=0.3269in
]%
{T_10.eps}%
}%
&
{\includegraphics[
height=0.3269in,
width=0.3269in
]%
{T_4.eps}%
}%
\end{array}
\]
and similar moves (see {\cite{LK}} for detail).\\

the \textbf{mosaic Reidemeister 2 moves} are given below:
\[
\begin{array}
[c]{cc}%
{\includegraphics[
height=0.3269in,
width=0.3269in
]%
{T_7.eps}%
}%
&
{\includegraphics[
height=0.3269in,
width=0.3269in
]%
{T_1.eps}%
}%
\\%
{\includegraphics[
height=0.3269in,
width=0.3269in
]%
{T_8.eps}%
}%
&
{\includegraphics[
height=0.3269in,
width=0.3269in
]%
{T_4.eps}%
}%
\end{array}
\Longleftrightarrow _{R_2}
\begin{array}
[c]{cc}%
{\includegraphics[
height=0.3269in,
width=0.3269in
]%
{T_10.eps}%
}%
&
{\includegraphics[
height=0.3269in,
width=0.3269in
]%
{T_1.eps}%
}%
\\%
{\includegraphics[
height=0.3269in,
width=0.3269in
]%
{T_10.eps}%
}%
&
{\includegraphics[
height=0.3269in,
width=0.3269in
]%
{T_4.eps}%
}%
\end{array}
\]

the \textbf{mosaic Reidemeister 3 moves} are given below:
\[
\begin{array}
[c]{ccc}%
{\includegraphics[
height=0.3269in,
width=0.3269in
]%
{T_0.eps}
}%
&
{\includegraphics[
height=0.3269in,
width=0.3269in
]%
{T_6.eps}
}%
&
{\includegraphics[
height=0.3269in,
width=0.3269in
]%
{T_2.eps}
}%
\\%
{\includegraphics[
height=0.3269in,
width=0.3269in
]%
{T_5.eps}%
}%
&
{\includegraphics[
height=0.3269in,
width=0.3269in
]%
{T_10.eps}%
}%
&
{\includegraphics[
height=0.3269in,
width=0.3269in
]%
{T_10.eps}%
}%
\\%
{\includegraphics[
height=0.3269in,
width=0.3269in
]%
{T_5.eps}%
}%
&
{\includegraphics[
height=0.3269in,
width=0.3269in
]%
{T_10.eps}%
}%
&
{\includegraphics[
height=0.3269in,
width=0.3269in
]%
{T_4.eps}%
}%
\end{array}
\Longleftrightarrow _{R_3}
\begin{array}
[c]{ccc}%
{\includegraphics[
height=0.3269in,
width=0.3269in
]%
{T_2.eps}
}%
&
{\includegraphics[
height=0.3269in,
width=0.3269in
]%
{T_10.eps}
}%
&
{\includegraphics[
height=0.3269in,
width=0.3269in
]%
{T_5.eps}
}%
\\%
{\includegraphics[
height=0.3269in,
width=0.3269in
]%
{T_10.eps}%
}%
&
{\includegraphics[
height=0.3269in,
width=0.3269in
]%
{T_10.eps}%
}%
&
{\includegraphics[
height=0.3269in,
width=0.3269in
]%
{T_5.eps}%
}%
\\%
{\includegraphics[
height=0.3269in,
width=0.3269in
]%
{T_4.eps}%
}%
&
{\includegraphics[
height=0.3269in,
width=0.3269in
]%
{T_6.eps}%
}%
&
{\includegraphics[
height=0.3269in,
width=0.3269in
]%
{T_0.eps}%
}%
\end{array}
\]

\end{defn}

The planar isotopy moves and the Reidemeister moves lie in the permutation group of the set of mosaics.
\ It easily follows that the planar isotopy moves and the Reidemeister moves
also lie in the group of all permutations of the set of knot mosaics
$\mathbb{K}^{(n)}$. \ Hence, we can make the following definition:

\begin{defn}
We define the (\textbf{knot mosaic}) \textbf{ambient group} $\mathbb{A}(n)$ as
the group of all permutations of the set of knot $n$-mosaics $\mathbb{K}%
^{(n)}$ generated by the mosaic planar isotopy and the mosaic Reidemeister
moves. \ 
\end{defn}

\section{Knot mosaic type}{\label {knot_type}}

We define the \textbf{mosaic injection }%
\[%
\begin{array}
[c]{rrr}%
\iota:\mathbb{M}^{(n)} & \longrightarrow & \mathbb{M}^{(n+1)}\\
M^{(n)} & \longmapsto & M^{(n+1)}%
\end{array}
\]
as%
\[
M_{ij}^{(n+1)}=\left\{
\begin{array}
[c]{cl}%
M_{ij}^{(n)} & \text{if }0\leq i,j<n\\
& \\
\raisebox{-0.1003in}{\includegraphics[
height=0.3269in,
width=0.3269in
]%
{T_0.eps}%
}%
& \text{otherwise}%
\end{array}
\right.
\]

Thus,
\[
M^{(n)}=%
\begin{array}
[c]{cccc}%
{\includegraphics[
height=0.3269in,
width=0.3269in
]%
{T_0.eps}
}%
&
{\includegraphics[
height=0.3269in,
width=0.3269in
]%
{T_2.eps}
}%
&
{\includegraphics[
height=0.3269in,
width=0.3269in
]%
{T_1.eps}
}%
&
{\includegraphics[
height=0.3269in,
width=0.3269in
]%
{T_0.eps}%
}%
\\%
{\includegraphics[
height=0.3269in,
width=0.3269in
]%
{T_2.eps}%
}%
&
{\includegraphics[
height=0.3269in,
width=0.3269in
]%
{T_9.eps}%
}%
&
{\includegraphics[
height=0.3269in,
width=0.3269in
]%
{T_10.eps}%
}%
&
{\includegraphics[
height=0.3269in,
width=0.3269in
]%
{T_1.eps}%
}%
\\%
{\includegraphics[
height=0.3269in,
width=0.3269in
]%
{T_6.eps}%
}%
&
{\includegraphics[
height=0.3269in,
width=0.3269in
]%
{T_3.eps}%
}%
&
{\includegraphics[
height=0.3269in,
width=0.3269in
]%
{T_9.eps}%
}%
&
{\includegraphics[
height=0.3269in,
width=0.3269in
]%
{T_4.eps}%
}%
\\%
{\includegraphics[
height=0.3269in,
width=0.3269in
]%
{T_3.eps}%
}%
&
{\includegraphics[
height=0.3269in,
width=0.3269in
]%
{T_5.eps}%
}%
&
{\includegraphics[
height=0.3269in,
width=0.3269in
]%
{T_4.eps}%
}%
&
{\includegraphics[
height=0.3269in,
width=0.3269in
]%
{T_0.eps}%
}%
\end{array}
\overset{\iota}{\longrightarrow}M^{(n+1)}=%
\begin{array}
[c]{ccccc}%

{\includegraphics[
height=0.3269in,
width=0.3269in
]%
{T_0.eps}%
}%
&
{\includegraphics[
height=0.3269in,
width=0.3269in
]%
{T_2.eps}%
}%
&
{\includegraphics[
height=0.3269in,
width=0.3269in
]%
{T_1.eps}%
}%
&
{\includegraphics[
height=0.3269in,
width=0.3269in
]%
{T_0.eps}%
}%
&
{\includegraphics[
height=0.3269in,
width=0.3269in
]%
{T_0.eps}%
}%

\\%

{\includegraphics[
height=0.3269in,
width=0.3269in
]%
{T_2.eps}%
}%

&
{\includegraphics[
height=0.3269in,
width=0.3269in
]%
{T_9.eps}%
}%

&
{\includegraphics[
height=0.3269in,
width=0.3269in
]%
{T_10.eps}%
}%

&
{\includegraphics[
height=0.3269in,
width=0.3269in
]%
{T_1.eps}%
}%
&
{\includegraphics[
height=0.3269in,
width=0.3269in
]%
{T_0.eps}%
}%
\\%

{\includegraphics[
height=0.3269in,
width=0.3269in
]%
{T_6.eps}%
}%

&
{\includegraphics[
height=0.3269in,
width=0.3269in
]%
{T_3.eps}%
}%

&
{\includegraphics[
height=0.3269in,
width=0.3269in
]%
{T_9.eps}%
}%

&

{\includegraphics[
height=0.3269in,
width=0.3269in
]%
{T_4.eps}%
}%

&

{\includegraphics[
height=0.3269in,
width=0.3269in
]%
{T_0.eps}%
}%

\\%

{\includegraphics[
height=0.3269in,
width=0.3269in
]%
{T_3.eps}%
}%

&

{\includegraphics[
height=0.3269in,
width=0.3269in
]%
{T_5.eps}%
}%

&

{\includegraphics[
height=0.3269in,
width=0.3269in
]%
{T_4.eps}%
}%
&

{\includegraphics[
height=0.3269in,
width=0.3269in
]%
{T_0.eps}%
}%
&

{\includegraphics[
height=0.3269in,
width=0.3269in
]%
{T_0.eps}%
}%
\\%

{\includegraphics[
height=0.3269in,
width=0.3269in
]%
{T_0.eps}%
}%
&

{\includegraphics[
height=0.3269in,
width=0.3269in
]%
{T_0.eps}%
}%
&
{\includegraphics[
height=0.3269in,
width=0.3269in
]%
{T_0.eps}%
}%
&
{\includegraphics[
height=0.3269in,
width=0.3269in
]%
{T_0.eps}%
}%
&
{\includegraphics[
height=0.3269in,
width=0.3269in
]%
{T_0.eps}%
}%
\end{array}
\]

\begin{defn}
Two $n$-mosaics $M$ and $N$ are said to be of the \textbf{same knot }%
$n$\textbf{-type}, written%
\[
M\underset{n}{\sim}N\text{ ,}%
\]
provided there is an element of the ambient isotopy group $\mathbb{A}(n)$
which transforms $M$ into $N$.
\end{defn}

\bigskip

\begin{defn}
An $m$-mosaic $M$ and an $n$-mosaic $N$ are said to be of the\textbf{ same
knot mosaic type}, written%
\[
M\sim N\text{ ,}%
\]
provided there exists a non-negative integer $\ell$ such that, if $m\leq n$,
then
\[
\iota^{\ell+n-m}M\sim_{\ell+n}\iota^{\ell}N\text{ ,}%
\]
or if $m>n$, then
\[
\iota^{\ell}M\sim_{\ell+m}\iota^{\ell+m-n}N\text{ ,}%
\]
where, for each non-negative integer $p$, $\iota^{p}$ denotes the $p$-fold
composition $\underset{p}{\underbrace{\iota\circ\iota\circ\cdots\circ\iota}}$ .
\end{defn}

\begin{exa}

Let 
\[
K_1=%
\begin{array}
[c]{ccc}%
{\includegraphics[
height=0.3269in,
width=0.3269in
]%
{T_2.eps}
}%
&
{\includegraphics[
height=0.3269in,
width=0.3269in
]%
{T_1.eps}
}%
&
{\includegraphics[
height=0.3269in,
width=0.3269in
]%
{T_0.eps}
}%
\\%
{\includegraphics[
height=0.3269in,
width=0.3269in
]%
{T_3.eps}%
}%
&
{\includegraphics[
height=0.3269in,
width=0.3269in
]%
{T_8.eps}%
}%
&
{\includegraphics[
height=0.3269in,
width=0.3269in
]%
{T_1.eps}%
}%
\\%
{\includegraphics[
height=0.3269in,
width=0.3269in
]%
{T_0.eps}%
}%
&
{\includegraphics[
height=0.3269in,
width=0.3269in
]%
{T_3.eps}%
}%
&
{\includegraphics[
height=0.3269in,
width=0.3269in
]%
{T_4.eps}%
}%
\end{array}
, \ K_2=%
\begin{array}
[c]{ccc}%
{\includegraphics[
height=0.3269in,
width=0.3269in
]%
{T_0.eps}
}%
&
{\includegraphics[
height=0.3269in,
width=0.3269in
]%
{T_2.eps}
}%
&
{\includegraphics[
height=0.3269in,
width=0.3269in
]%
{T_1.eps}
}%
\\%
{\includegraphics[
height=0.3269in,
width=0.3269in
]%
{T_2.eps}%
}%
&
{\includegraphics[
height=0.3269in,
width=0.3269in
]%
{T_7.eps}%
}%
&
{\includegraphics[
height=0.3269in,
width=0.3269in
]%
{T_4.eps}%
}%
\\%
{\includegraphics[
height=0.3269in,
width=0.3269in
]%
{T_3.eps}%
}%
&
{\includegraphics[
height=0.3269in,
width=0.3269in
]%
{T_4.eps}%
}%
&
{\includegraphics[
height=0.3269in,
width=0.3269in
]%
{T_0.eps}%
}%
\end{array}
\]

, then $K_1 \not\sim_3 K_2$ but $K_1 \sim K_2 \ (\iota K_1 \sim_4 \iota K_2)$.

\[
\iota K_1=%
\begin{array}
[c]{cccc}%
{\includegraphics[
height=0.3in,
width=0.3in
]%
{T_2.eps}
}%
&
{\includegraphics[
height=0.3in,
width=0.3in
]%
{T_1.eps}
}%
&
{\includegraphics[
height=0.3in,
width=0.3in
]%
{T_0.eps}
}%
&
{\includegraphics[
height=0.3in,
width=0.3in
]%
{T_0.eps}%
}%
\\%
{\includegraphics[
height=0.3in,
width=0.3in
]%
{T_3.eps}%
}%
&
{\includegraphics[
height=0.3in,
width=0.3in
]%
{T_8.eps}%
}%
&
{\includegraphics[
height=0.3in,
width=0.3in
]%
{T_1.eps}%
}%
&
{\includegraphics[
height=0.3in,
width=0.3in
]%
{T_0.eps}%
}%
\\%
{\includegraphics[
height=0.3in,
width=0.3in
]%
{T_0.eps}%
}%
&
{\includegraphics[
height=0.3in,
width=0.3in
]%
{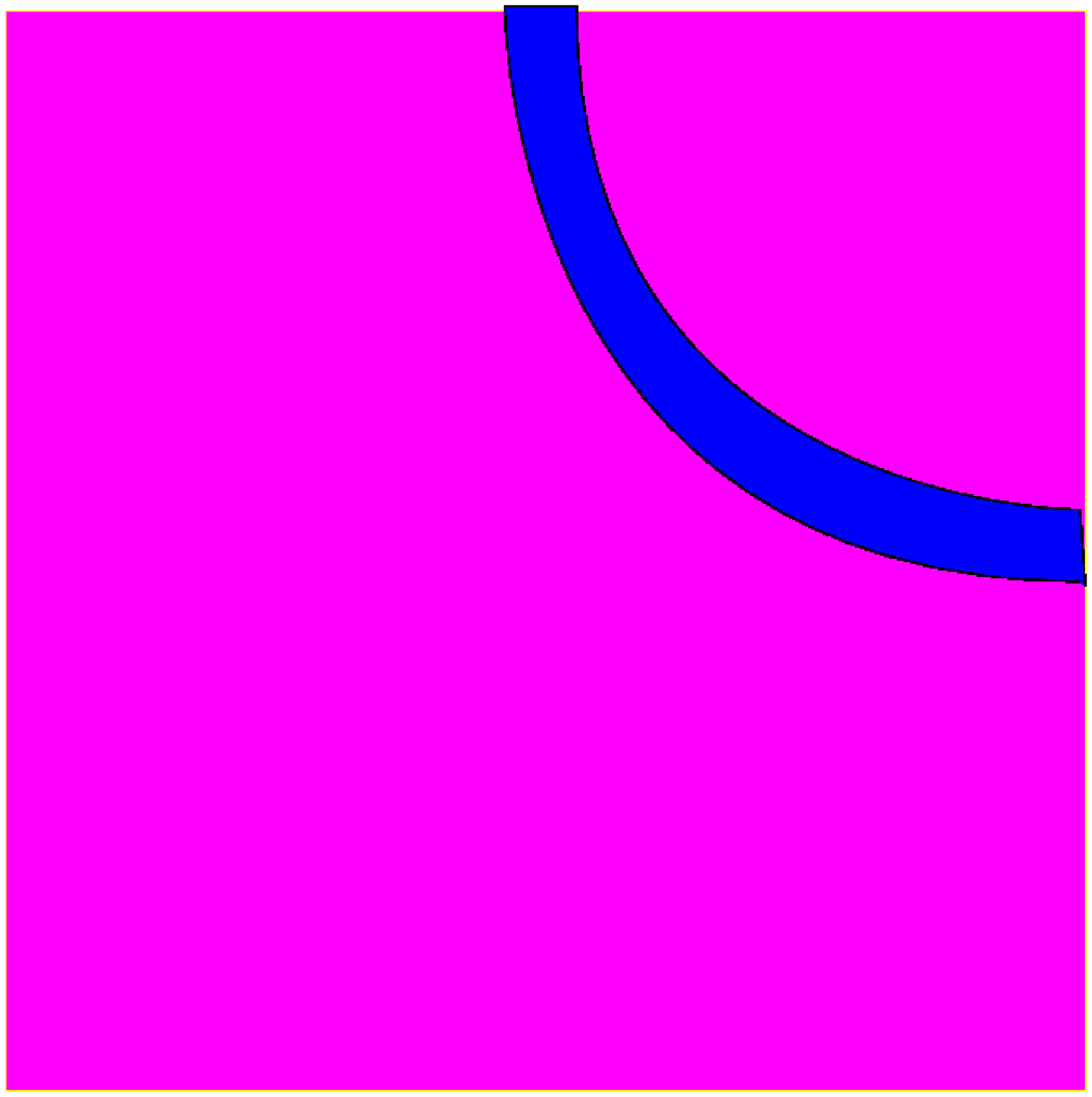}%
}%
&
{\includegraphics[
height=0.3in,
width=0.3in
]%
{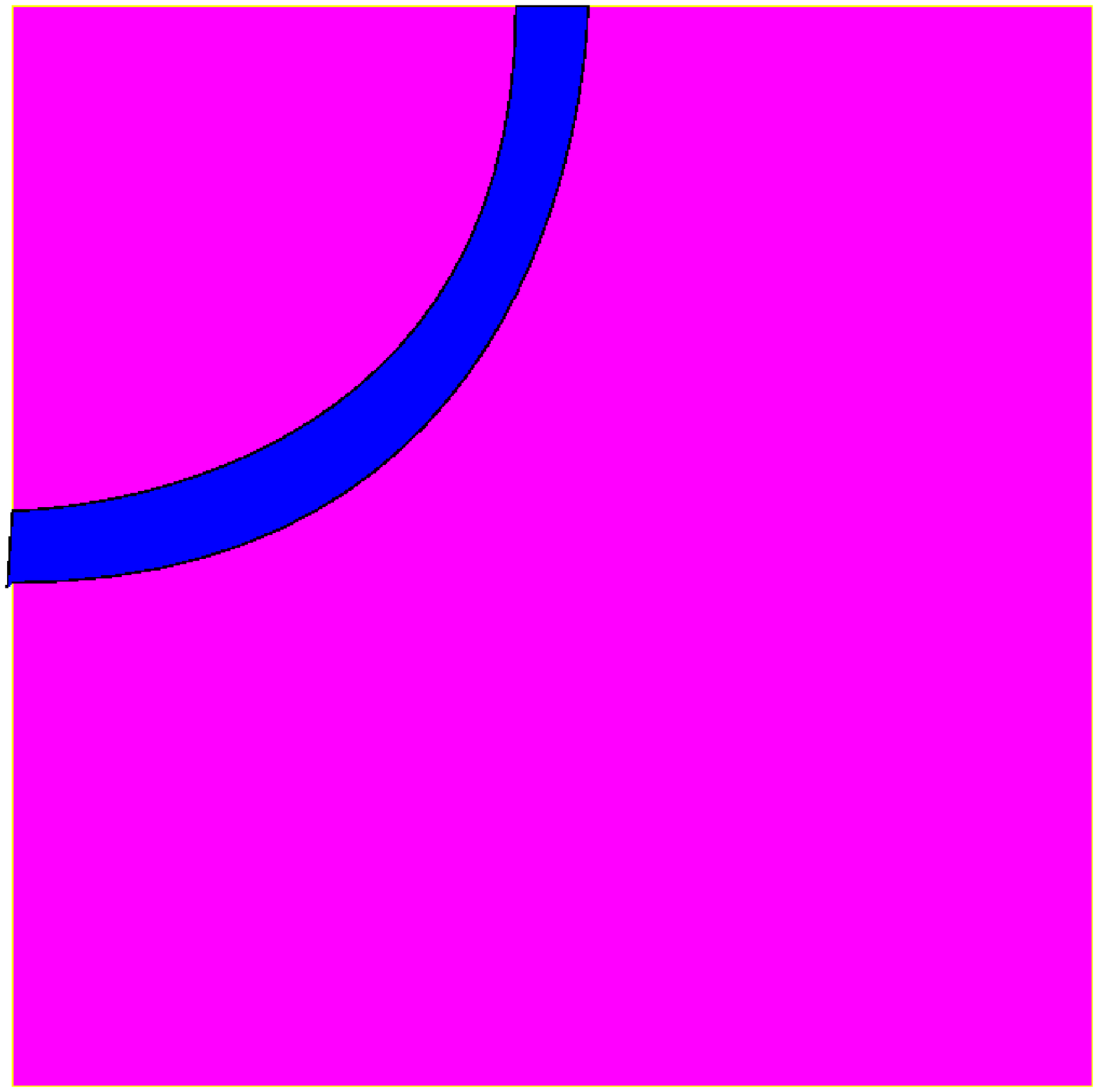}%
}%
&
{\includegraphics[
height=0.3in,
width=0.3in
]%
{T_0.eps}%
}%
\\%
{\includegraphics[
height=0.3in,
width=0.3in
]%
{T_0.eps}%
}%
&
{\includegraphics[
height=0.3in,
width=0.3in
]%
{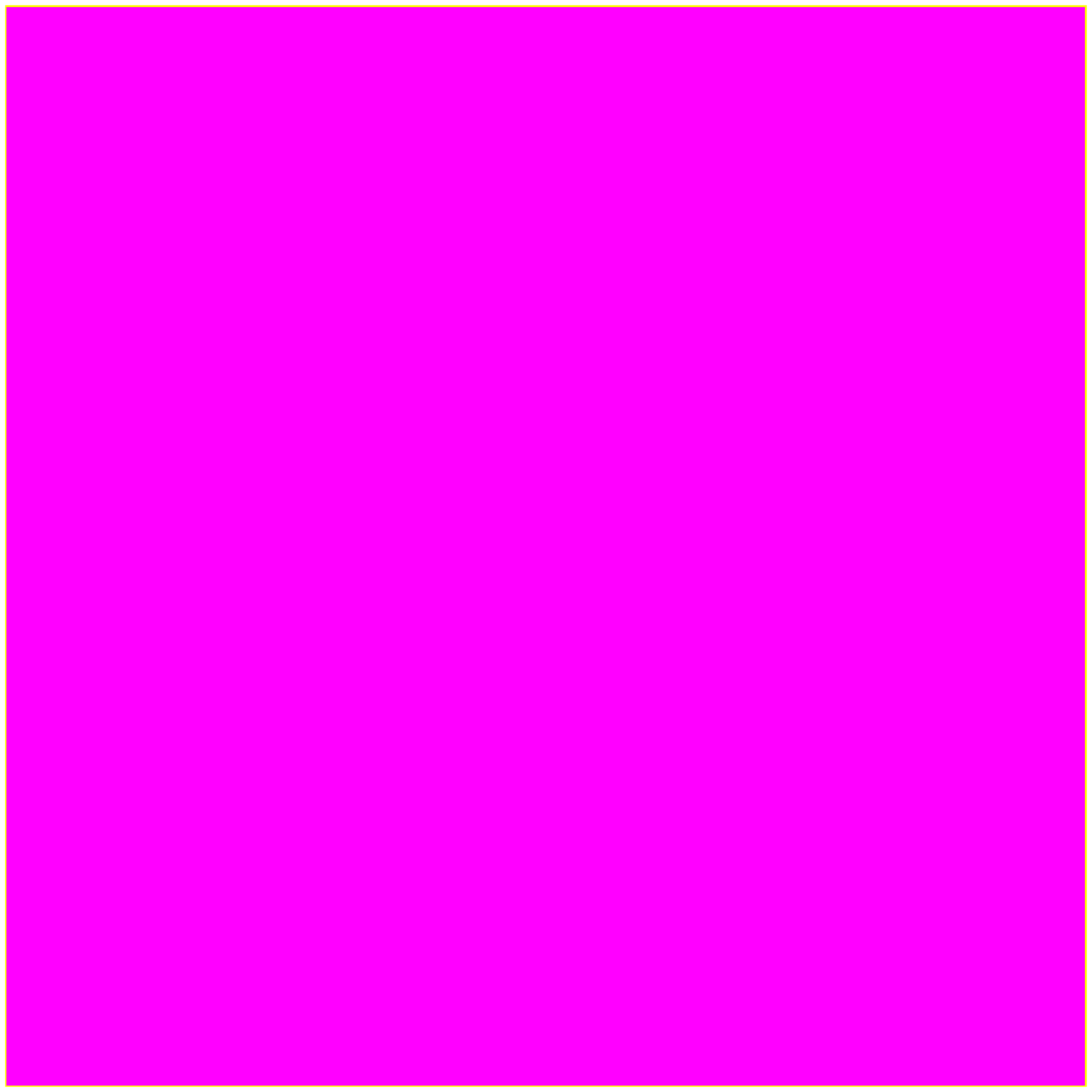}%
}%
&
{\includegraphics[
height=0.3in,
width=0.3in
]%
{B_0.eps}%
}%
&
{\includegraphics[
height=0.3in,
width=0.3in
]%
{T_0.eps}%
}%
\end{array}
\Leftrightarrow _{P_6}%
\begin{array}
[c]{cccc}%
{\includegraphics[
height=0.3in,
width=0.3in
]%
{T_2.eps}
}%
&
{\includegraphics[
height=0.3in,
width=0.3in
]%
{T_1.eps}
}%
&
{\includegraphics[
height=0.3in,
width=0.3in
]%
{T_0.eps}
}%
&
{\includegraphics[
height=0.3in,
width=0.3in
]%
{T_0.eps}%
}%
\\%
{\includegraphics[
height=0.3in,
width=0.3in
]%
{T_3.eps}%
}%
&
{\includegraphics[
height=0.3in,
width=0.3in
]%
{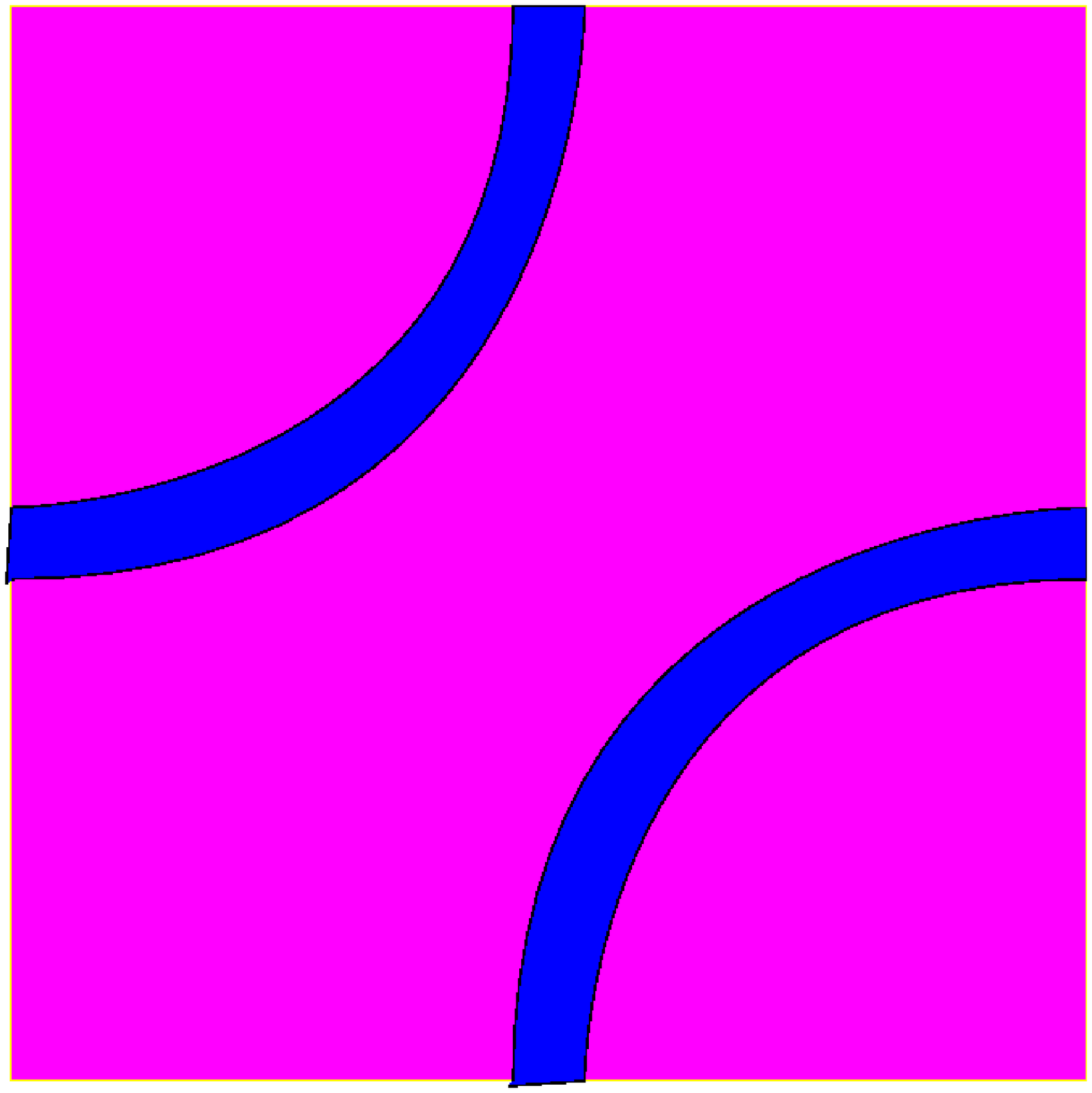}%
}%
&
{\includegraphics[
height=0.3in,
width=0.3in
]%
{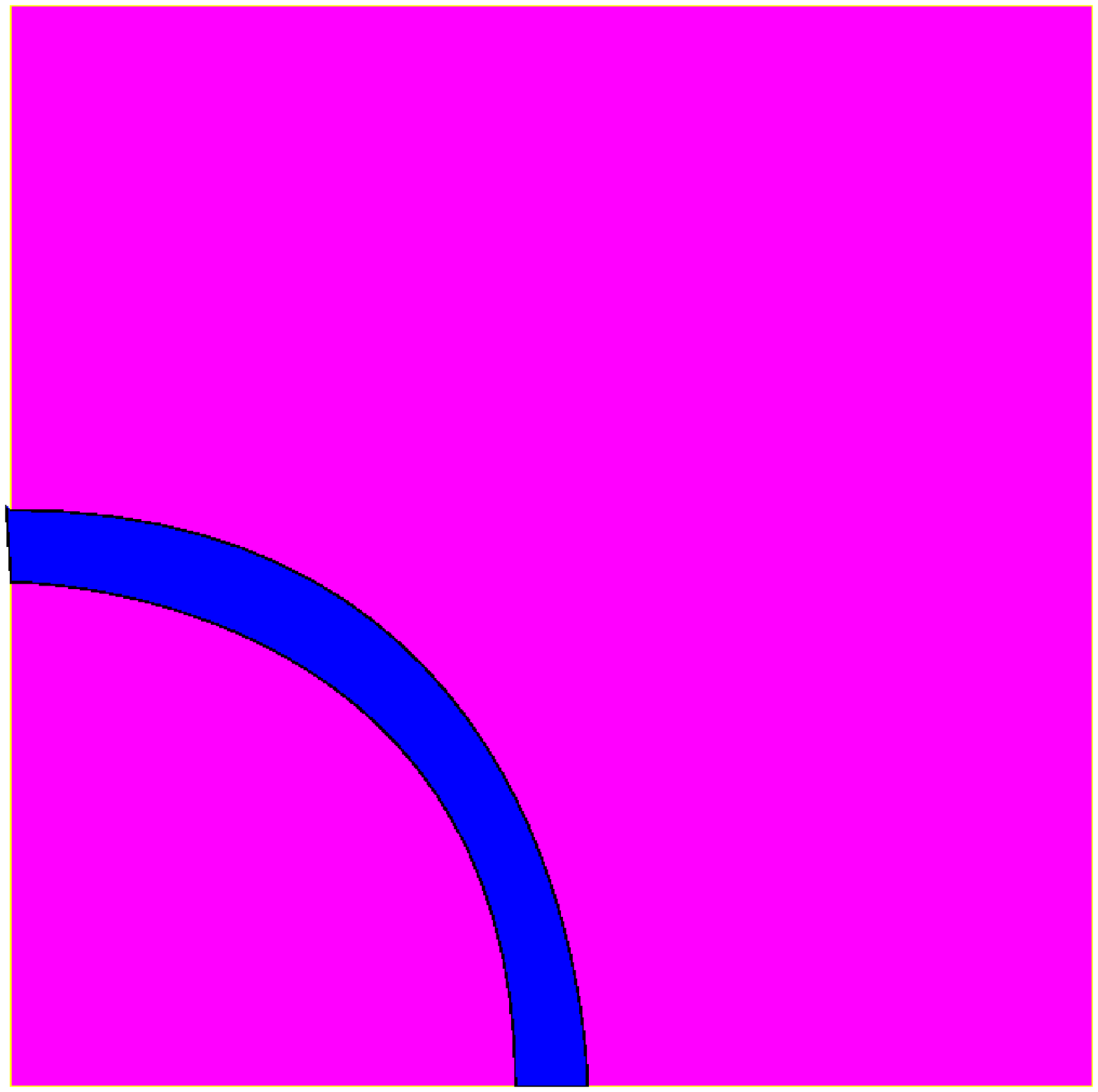}%
}%
&
{\includegraphics[
height=0.3in,
width=0.3in
]%
{T_0.eps}%
}%
\\%
{\includegraphics[
height=0.3in,
width=0.3in
]%
{T_0.eps}%
}%
&
{\includegraphics[
height=0.3in,
width=0.3in
]%
{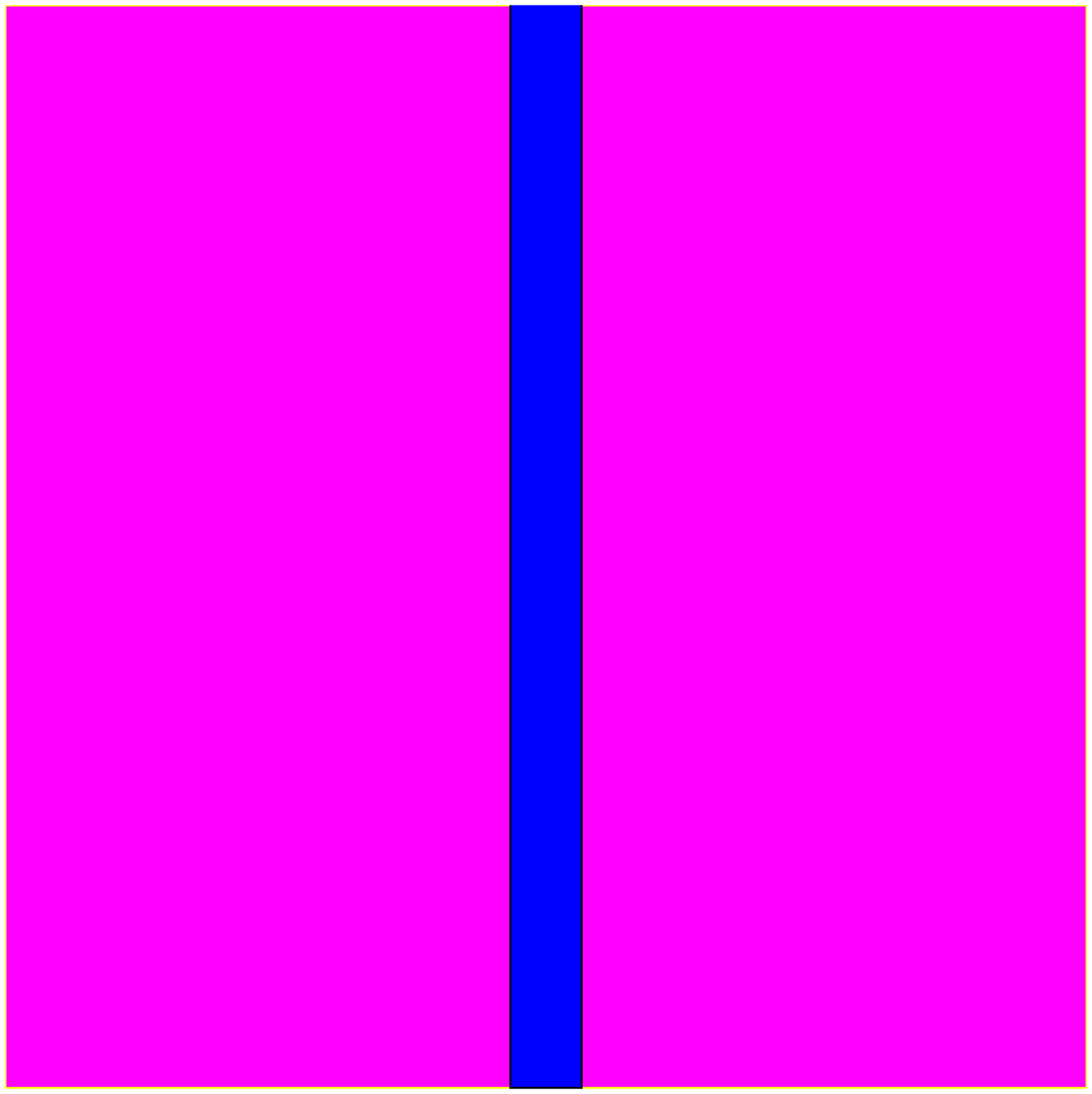}%
}%
&
{\includegraphics[
height=0.3in,
width=0.3in
]%
{B_6.eps}%
}%
&
{\includegraphics[
height=0.3in,
width=0.3in
]%
{T_0.eps}%
}%
\\%
{\includegraphics[
height=0.3in,
width=0.3in
]%
{T_0.eps}%
}%
&
{\includegraphics[
height=0.3in,
width=0.3in
]%
{T_3.eps}%
}%
&
{\includegraphics[
height=0.3in,
width=0.3in
]%
{T_4.eps}%
}%
&
{\includegraphics[
height=0.3in,
width=0.3in
]%
{T_0.eps}%
}%
\end{array}
\Leftrightarrow _{P_6}
\begin{array}
[c]{cccc}%
{\includegraphics[
height=0.3in,
width=0.3in
]%
{T_2.eps}
}%
&
{\includegraphics[
height=0.3in,
width=0.3in
]%
{T_1.eps}
}%
&
{\includegraphics[
height=0.3in,
width=0.3in
]%
{T_0.eps}
}%
&
{\includegraphics[
height=0.3in,
width=0.3in
]%
{T_0.eps}%
}%
\\%
{\includegraphics[
height=0.3in,
width=0.3in
]%
{T_3.eps}%
}%
&
{\includegraphics[
height=0.3in,
width=0.3in
]%
{T_4.eps}%
}%
&
{\includegraphics[
height=0.3in,
width=0.3in
]%
{T_0.eps}%
}%
&
{\includegraphics[
height=0.3in,
width=0.3in
]%
{T_0.eps}%
}%
\\%
{\includegraphics[
height=0.3in,
width=0.3in
]%
{B_0.eps}%
}%
&
{\includegraphics[
height=0.3in,
width=0.3in
]%
{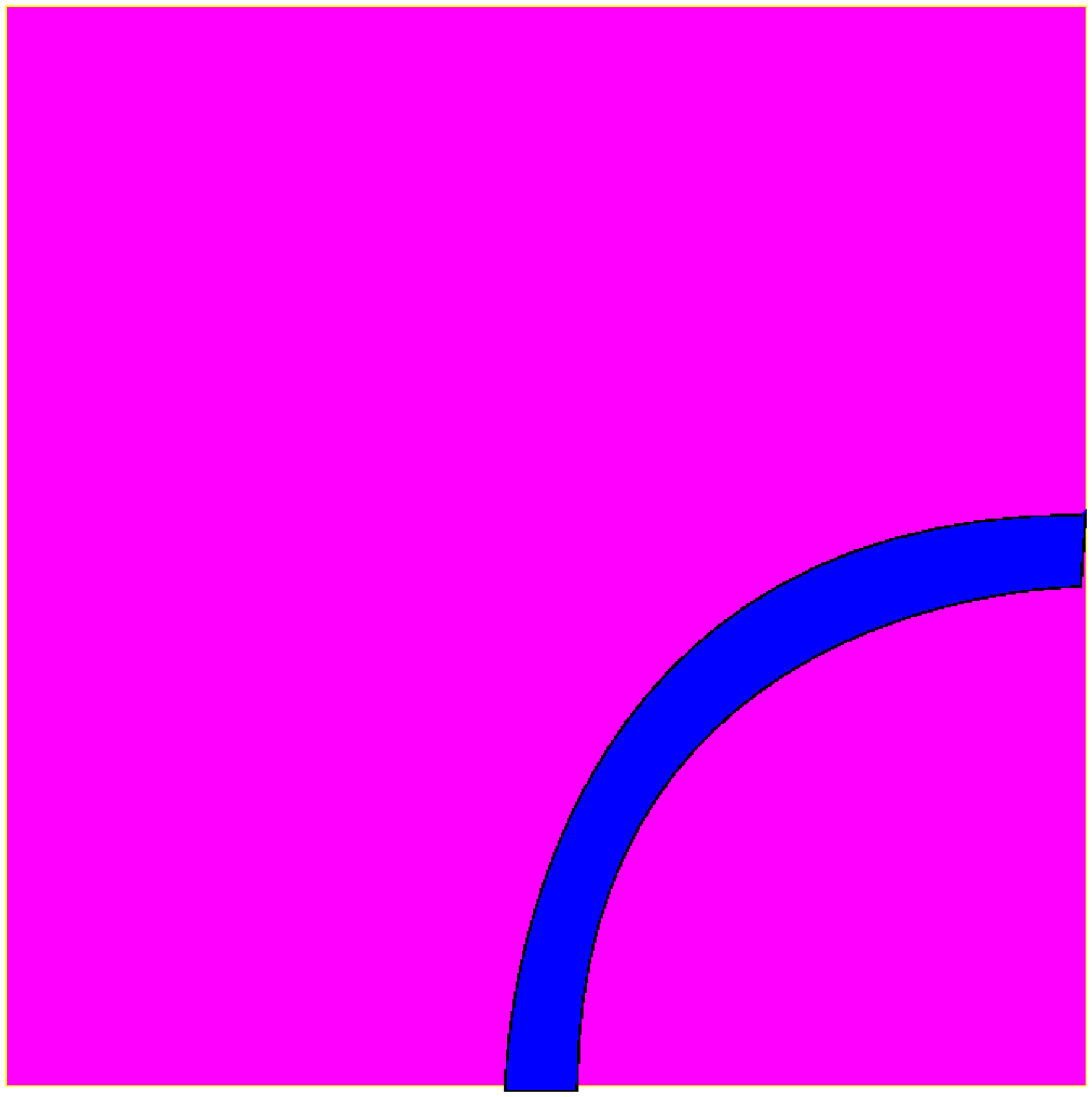}%
}%
&
{\includegraphics[
height=0.3in,
width=0.3in
]%
{T_1.eps}%
}%
&
{\includegraphics[
height=0.3in,
width=0.3in
]%
{T_0.eps}%
}%
\\%
{\includegraphics[
height=0.3in,
width=0.3in
]%
{B_0.eps}%
}%
&
{\includegraphics[
height=0.3in,
width=0.3in
]%
{B_3.eps}%
}%
&
{\includegraphics[
height=0.3in,
width=0.3in
]%
{T_4.eps}%
}%
&
{\includegraphics[
height=0.3in,
width=0.3in
]%
{T_0.eps}%
}%
\end{array}
\]

\[
\Leftrightarrow _{P_6}
\begin{array}
[c]{cccc}%
{\includegraphics[
height=0.3in,
width=0.3in
]%
{T_2.eps}
}%
&
{\includegraphics[
height=0.3in,
width=0.3in
]%
{T_1.eps}
}%
&
{\includegraphics[
height=0.3in,
width=0.3in
]%
{T_0.eps}
}%
&
{\includegraphics[
height=0.3in,
width=0.3in
]%
{T_0.eps}%
}%
\\%
{\includegraphics[
height=0.3in,
width=0.3in
]%
{T_3.eps}%
}%
&
{\includegraphics[
height=0.3in,
width=0.3in
]%
{T_4.eps}%
}%
&
{\includegraphics[
height=0.3in,
width=0.3in
]%
{T_0.eps}%
}%
&
{\includegraphics[
height=0.3in,
width=0.3in
]%
{T_0.eps}%
}%
\\%
{\includegraphics[
height=0.3in,
width=0.3in
]%
{T_2.eps}%
}%
&
{\includegraphics[
height=0.3in,
width=0.3in
]%
{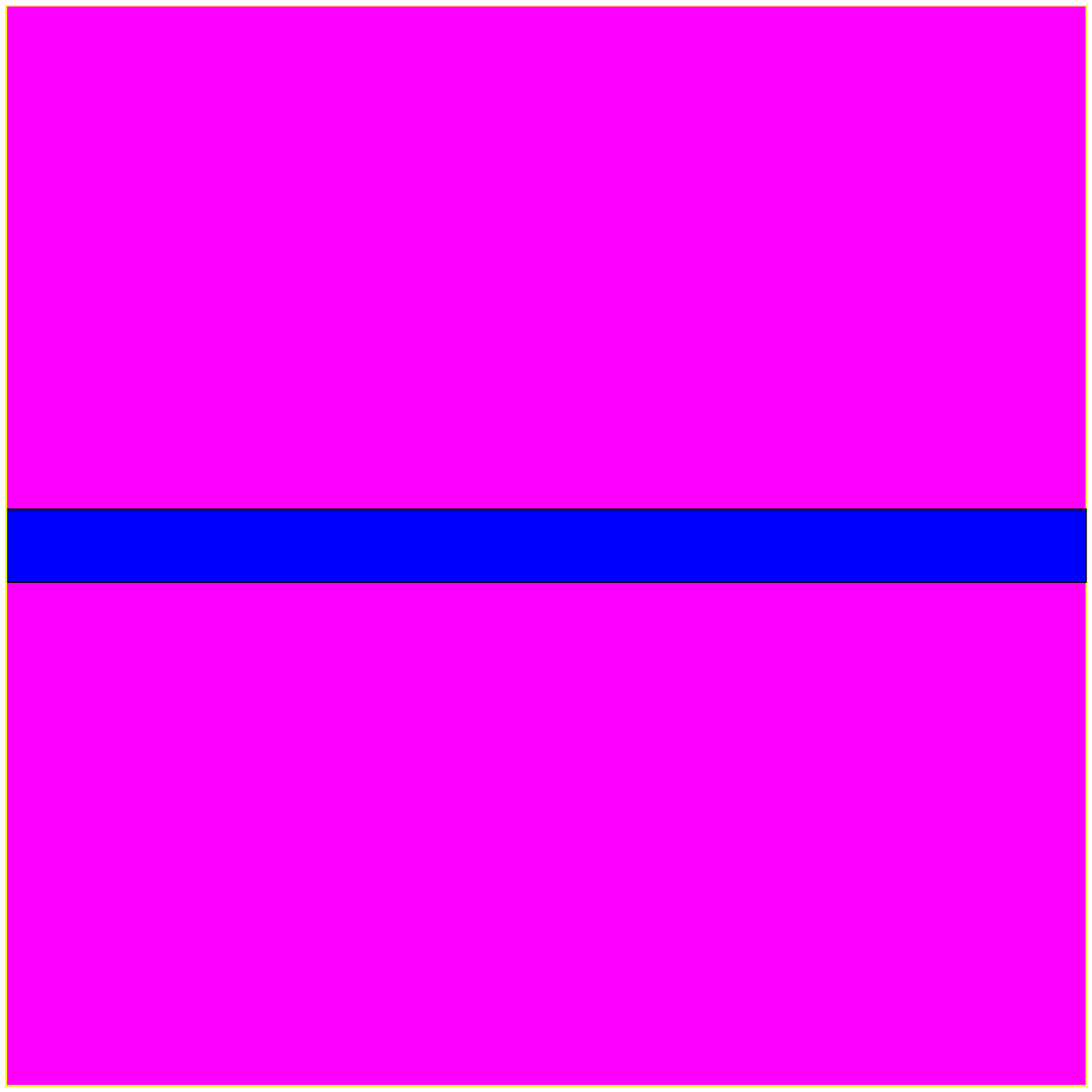}%
}%
&
{\includegraphics[
height=0.3in,
width=0.3in
]%
{B_1.eps}%
}%
&
{\includegraphics[
height=0.3in,
width=0.3in
]%
{T_0.eps}%
}%
\\%
{\includegraphics[
height=0.3in,
width=0.3in
]%
{T_3.eps}%
}%
&
{\includegraphics[
height=0.3in,
width=0.3in
]%
{B_5.eps}%
}%
&
{\includegraphics[
height=0.3in,
width=0.3in
]%
{B_4.eps}%
}%
&
{\includegraphics[
height=0.3in,
width=0.3in
]%
{T_0.eps}%
}%
\end{array}
\Leftrightarrow _{P_6}%
\begin{array}
[c]{cccc}%
{\includegraphics[
height=0.3in,
width=0.3in
]%
{T_2.eps}
}%
&
{\includegraphics[
height=0.3in,
width=0.3in
]%
{B_1.eps}
}%
&
{\includegraphics[
height=0.3in,
width=0.3in
]%
{B_0.eps}
}%
&
{\includegraphics[
height=0.3in,
width=0.3in
]%
{T_0.eps}%
}%
\\%
{\includegraphics[
height=0.3in,
width=0.3in
]%
{T_3.eps}%
}%
&
{\includegraphics[
height=0.3in,
width=0.3in
]%
{B_4.eps}%
}%
&
{\includegraphics[
height=0.3in,
width=0.3in
]%
{B_0.eps}%
}%
&
{\includegraphics[
height=0.3in,
width=0.3in
]%
{T_0.eps}%
}%
\\%
{\includegraphics[
height=0.3in,
width=0.3in
]%
{T_2.eps}%
}%
&
{\includegraphics[
height=0.3in,
width=0.3in
]%
{T_1.eps}%
}%
&
{\includegraphics[
height=0.3in,
width=0.3in
]%
{T_0.eps}%
}%
&
{\includegraphics[
height=0.3in,
width=0.3in
]%
{T_0.eps}%
}%
\\%
{\includegraphics[
height=0.3in,
width=0.3in
]%
{T_3.eps}%
}%
&
{\includegraphics[
height=0.3in,
width=0.3in
]%
{T_4.eps}%
}%
&
{\includegraphics[
height=0.3in,
width=0.3in
]%
{T_0.eps}%
}%
&
{\includegraphics[
height=0.3in,
width=0.3in
]%
{T_0.eps}%
}%
\end{array}
\Leftrightarrow _{P_6}%
\begin{array}
[c]{cccc}%
{\includegraphics[
height=0.3in,
width=0.3in
]%
{B_2.eps}
}%
&
{\includegraphics[
height=0.3in,
width=0.3in
]%
{B_5.eps}
}%
&
{\includegraphics[
height=0.3in,
width=0.3in
]%
{T_1.eps}
}%
&
{\includegraphics[
height=0.3in,
width=0.3in
]%
{T_0.eps}%
}%
\\%
{\includegraphics[
height=0.3in,
width=0.3in
]%
{B_3.eps}%
}%
&
{\includegraphics[
height=0.3in,
width=0.3in
]%
{B_5.eps}%
}%
&
{\includegraphics[
height=0.3in,
width=0.3in
]%
{T_4.eps}%
}%
&
{\includegraphics[
height=0.3in,
width=0.3in
]%
{T_0.eps}%
}%
\\%
{\includegraphics[
height=0.3in,
width=0.3in
]%
{T_2.eps}%
}%
&
{\includegraphics[
height=0.3in,
width=0.3in
]%
{T_1.eps}%
}%
&
{\includegraphics[
height=0.3in,
width=0.3in
]%
{T_0.eps}%
}%
&
{\includegraphics[
height=0.3in,
width=0.3in
]%
{T_0.eps}%
}%
\\%
{\includegraphics[
height=0.3in,
width=0.3in
]%
{T_3.eps}%
}%
&
{\includegraphics[
height=0.3in,
width=0.3in
]%
{T_4.eps}%
}%
&
{\includegraphics[
height=0.3in,
width=0.3in
]%
{T_0.eps}%
}%
&
{\includegraphics[
height=0.3in,
width=0.3in
]%
{T_0.eps}%
}%
\end{array}
\]

\[
\Leftrightarrow _{P_6}
\begin{array}
[c]{cccc}%
{\includegraphics[
height=0.3in,
width=0.3in
]%
{T_0.eps}
}%
&
{\includegraphics[
height=0.3in,
width=0.3in
]%
{T_2.eps}
}%
&
{\includegraphics[
height=0.3in,
width=0.3in
]%
{T_1.eps}
}%
&
{\includegraphics[
height=0.3in,
width=0.3in
]%
{T_0.eps}%
}%
\\%
{\includegraphics[
height=0.3in,
width=0.3in
]%
{B_0.eps}%
}%
&
{\includegraphics[
height=0.3in,
width=0.3in
]%
{B_3.eps}%
}%
&
{\includegraphics[
height=0.3in,
width=0.3in
]%
{T_4.eps}%
}%
&
{\includegraphics[
height=0.3in,
width=0.3in
]%
{T_0.eps}%
}%
\\%
{\includegraphics[
height=0.3in,
width=0.3in
]%
{B_2.eps}%
}%
&
{\includegraphics[
height=0.3in,
width=0.3in
]%
{B_1.eps}%
}%
&
{\includegraphics[
height=0.3in,
width=0.3in
]%
{T_0.eps}%
}%
&
{\includegraphics[
height=0.3in,
width=0.3in
]%
{T_0.eps}%
}%
\\%
{\includegraphics[
height=0.3in,
width=0.3in
]%
{T_3.eps}%
}%
&
{\includegraphics[
height=0.3in,
width=0.3in
]%
{T_4.eps}%
}%
&
{\includegraphics[
height=0.3in,
width=0.3in
]%
{T_0.eps}%
}%
&
{\includegraphics[
height=0.3in,
width=0.3in
]%
{T_0.eps}%
}%
\end{array}
\Leftrightarrow _{P_6}%
\begin{array}
[c]{cccc}%
{\includegraphics[
height=0.3in,
width=0.3in
]%
{T_0.eps}
}%
&
{\includegraphics[
height=0.3in,
width=0.3in
]%
{T_2.eps}
}%
&
{\includegraphics[
height=0.3in,
width=0.3in
]%
{T_1.eps}
}%
&
{\includegraphics[
height=0.3in,
width=0.3in
]%
{T_0.eps}%
}%
\\%
{\includegraphics[
height=0.3in,
width=0.3in
]%
{T_2.eps}%
}%
&
{\includegraphics[
height=0.3in,
width=0.3in
]%
{T_7.eps}%
}%
&
{\includegraphics[
height=0.3in,
width=0.3in
]%
{T_4.eps}%
}%
&
{\includegraphics[
height=0.3in,
width=0.3in
]%
{T_0.eps}%
}%
\\%
{\includegraphics[
height=0.3in,
width=0.3in
]%
{B_6.eps}%
}%
&
{\includegraphics[
height=0.3in,
width=0.3in
]%
{B_6.eps}%
}%
&
{\includegraphics[
height=0.3in,
width=0.3in
]%
{T_0.eps}%
}%
&
{\includegraphics[
height=0.3in,
width=0.3in
]%
{T_0.eps}%
}%
\\%
{\includegraphics[
height=0.3in,
width=0.3in
]%
{B_3.eps}%
}%
&
{\includegraphics[
height=0.3in,
width=0.3in
]%
{B_4.eps}%
}%
&
{\includegraphics[
height=0.3in,
width=0.3in
]%
{T_0.eps}%
}%
&
{\includegraphics[
height=0.3in,
width=0.3in
]%
{T_0.eps}%
}%
\end{array}
\Leftrightarrow _{P_6}%
\begin{array}
[c]{cccc}%
{\includegraphics[
height=0.3in,
width=0.3in
]%
{T_0.eps}
}%
&
{\includegraphics[
height=0.3in,
width=0.3in
]%
{T_2.eps}
}%
&
{\includegraphics[
height=0.3in,
width=0.3in
]%
{T_1.eps}
}%
&
{\includegraphics[
height=0.3in,
width=0.3in
]%
{T_0.eps}%
}%
\\%
{\includegraphics[
height=0.3in,
width=0.3in
]%
{T_2.eps}%
}%
&
{\includegraphics[
height=0.3in,
width=0.3in
]%
{T_7.eps}%
}%
&
{\includegraphics[
height=0.3in,
width=0.3in
]%
{T_4.eps}%
}%
&
{\includegraphics[
height=0.3in,
width=0.3in
]%
{T_0.eps}%
}%
\\%
{\includegraphics[
height=0.3in,
width=0.3in
]%
{T_3.eps}%
}%
&
{\includegraphics[
height=0.3in,
width=0.3in
]%
{T_4.eps}%
}%
&
{\includegraphics[
height=0.3in,
width=0.3in
]%
{T_0.eps}%
}%
&
{\includegraphics[
height=0.3in,
width=0.3in
]%
{T_0.eps}%
}%
\\%
{\includegraphics[
height=0.3in,
width=0.3in
]%
{T_0.eps}%
}%
&
{\includegraphics[
height=0.3in,
width=0.3in
]%
{T_0.eps}%
}%
&
{\includegraphics[
height=0.3in,
width=0.3in
]%
{T_0.eps}%
}%
&
{\includegraphics[
height=0.3in,
width=0.3in
]%
{T_0.eps}%
}%
\end{array}
=\iota K_2
\]

\end{exa}

\section{grid diagram}{\label {grid_diagram}}

Let $K$ be an oriented knot in $S^3$. Choose a knot diagram $D$ for $K$ such that

\begin{itemize}
\item $D$ is composed entirely of horizontal and vertical segments,

\item no two horizontal segments of $D$ have the same $y$-coordinate, and no two vertical
segments of D have the same $x$-coordinate, and

\item at each crossing, the vertical segment crosses over the horizontal segment.

\end{itemize}

Every knot admits such a diagram (Starting with
a knot diagram $D$, one approximates $D$ using horizontal and vertical segments,
so that crossings are always vertical over horizontal). The only data in such a diagram are the endpoints of the segments, which we record by placing $X$'s and $O$'s at these endpoints,
alternately around the knot, and so that the knot is oriented from $X$ to $O$ along vertical segments. The data of the knot is entirely encoded in these X's and O's, which we can see as sitting in the
middle of squares on a piece of graph paper. Notice that no two $X$'s (respectively $O$'s) lie on the same horizontal or vertical line (Perturb the result so that no segments lie on the same horizontal or vertical line).

\begin{figure}[ht]
\begin{minipage}{2.5cm}
\includegraphics[width=\hsize]{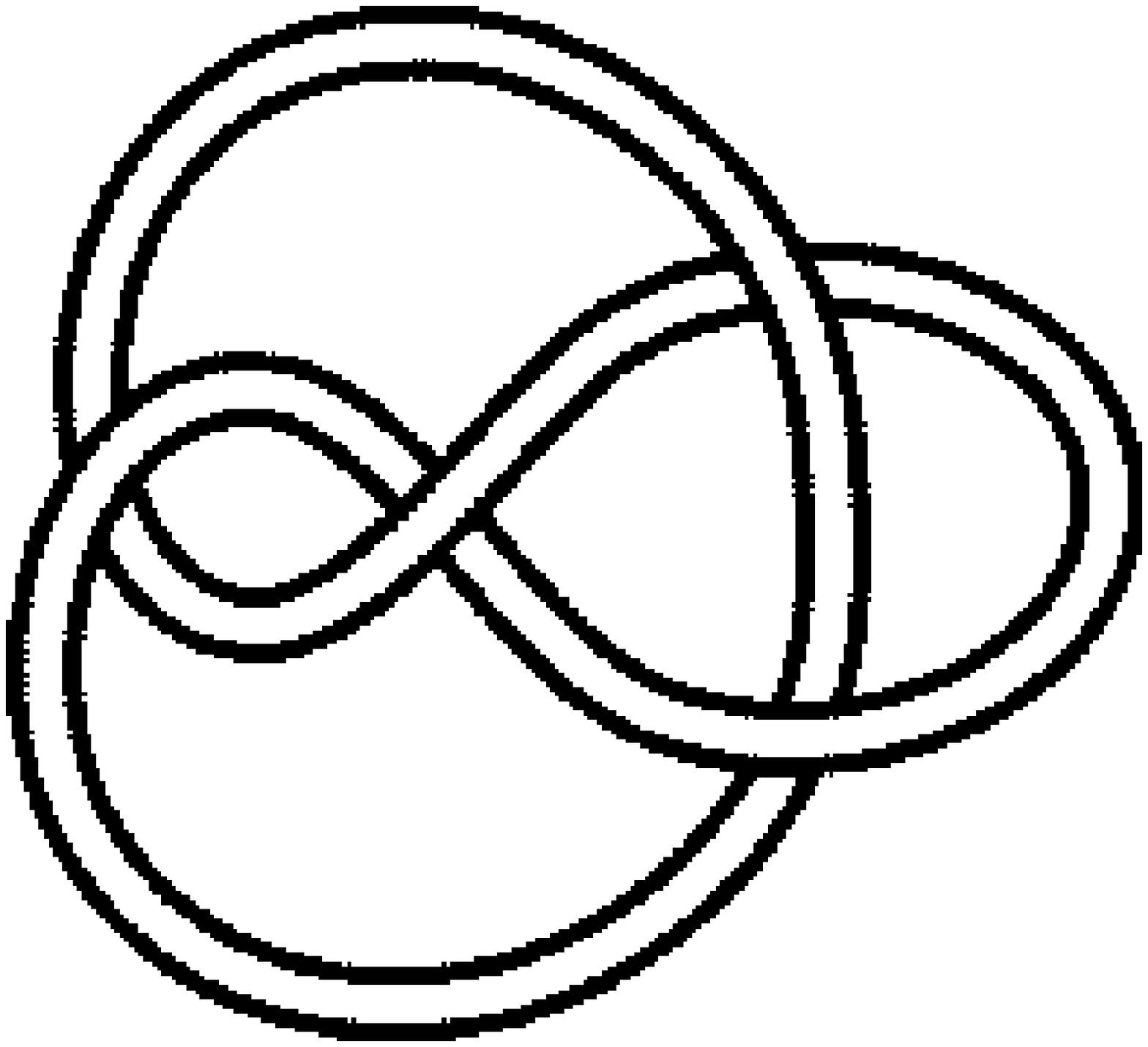}
\end{minipage}
\ $\Longrightarrow $ \ 
\begin{minipage}{2.5cm}
\includegraphics[width=\hsize]{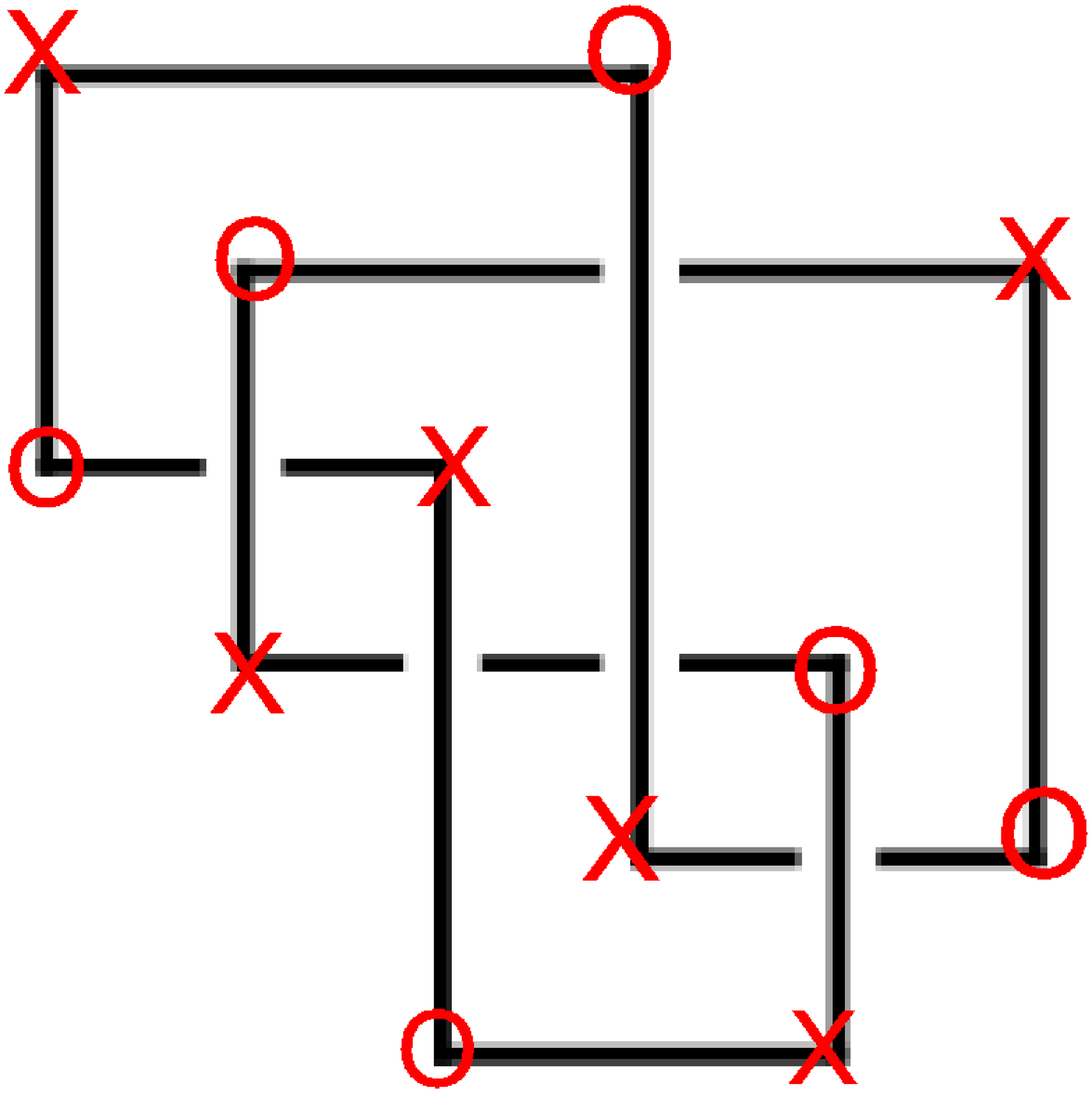}
\end{minipage}
\ $\Longrightarrow $ \ 
\begin{minipage}{2.5cm}
\includegraphics[width=\hsize]{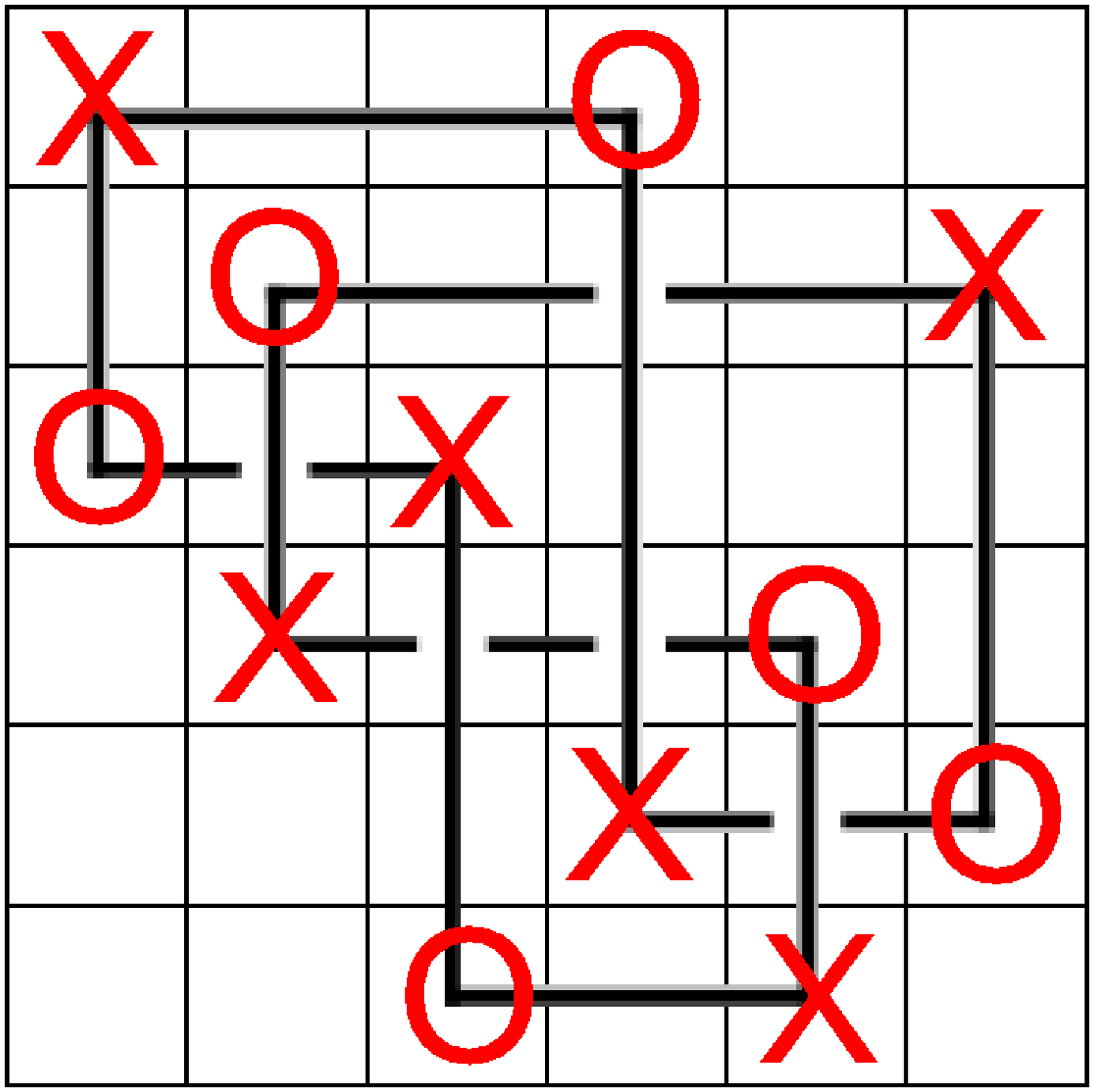}
\end{minipage}
\ $\Longrightarrow $ \ 
\begin{minipage}{2.5cm}
\includegraphics[width=\hsize]{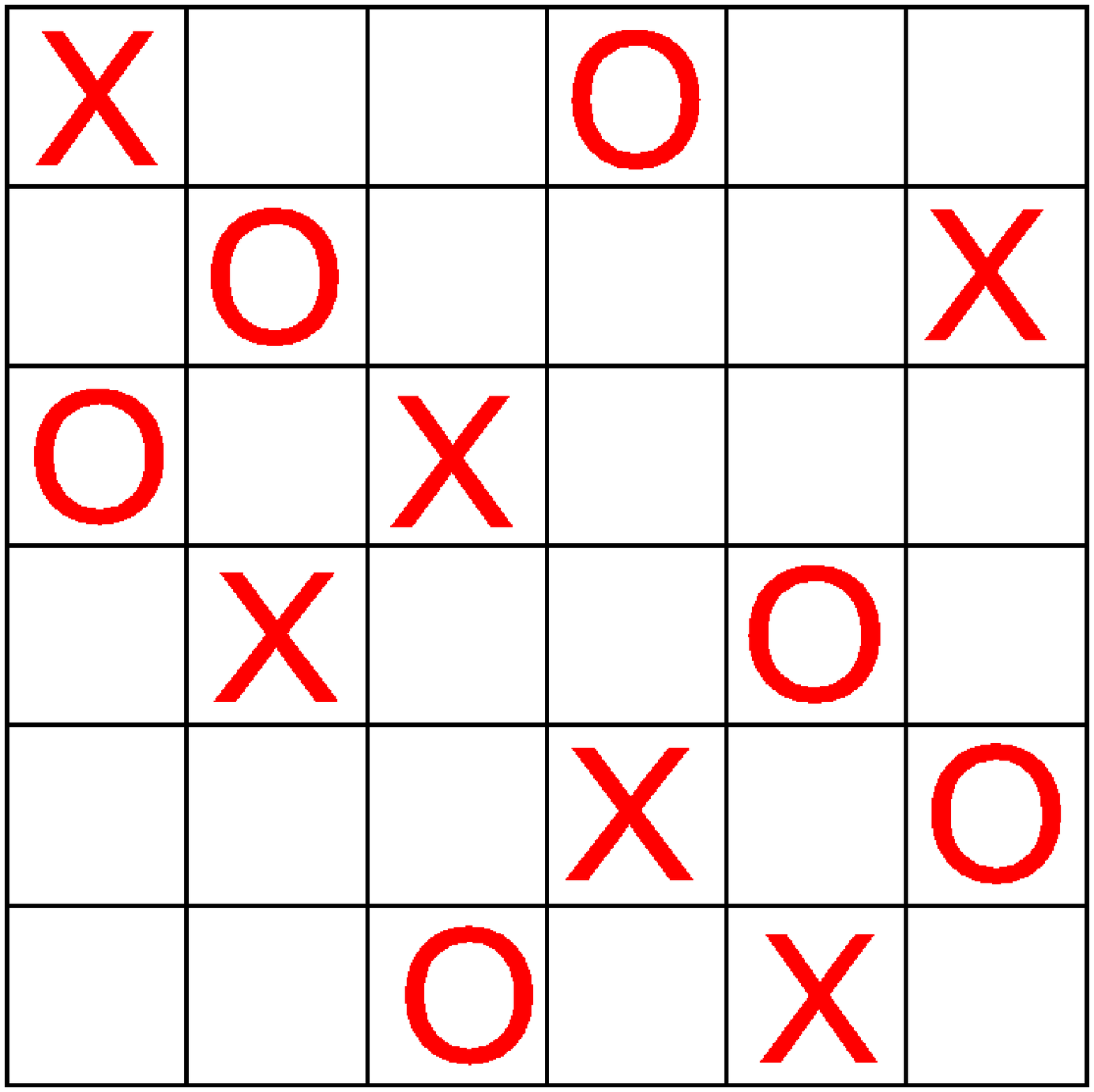}
\end{minipage}
\end{figure}

\begin{defn}
Let $\mathbb{X} = \{ X_i \}^N _{i=1}$ and $O = \{O_i \}^N
 _{i=1}$ denote the set of $X$'s and $O$'s, respectively. Up to
isotopy of the knot (and renumbering of the $X_i$), we may assume that the coordinates
of $X_i$ are $(i - \frac{1}{2}, \sigma_{\mathbb{X}}(i)- \frac{1}{2})$
for some permutation $\sigma_{\mathbb{X}} \in S_N$. Then (after renumbering),
the coordinates of $O_i$ are $(i - \frac{1}{2}, \sigma_{\mathbb{O}}(i)- \frac{1}{2})$ for some permutation $\sigma_{\mathbb{O}} \in S_N$. The data
$(\mathbb{R}^2,\mathbb{X},\mathbb{O})$ is a planar grid diagram for the knot $K$.
\end{defn}

Of course, different grid diagrams can lead to the same link. However, this is controlled by the
following theorem:

\begin{thm}{\label{Dy}}{\cite{Dy}}
Any two grid diagrams which describe the
same link can be connected by a finite sequence of the following elementary moves:
\begin{itemize}
\item {\bf Cyclic permutation} cyclic permutation of the columns (resp. rows);

\item {\bf Commutation} commutation of two adjacent columns (resp. rows) under the condition that all the decorations of one of the two commuting columns (resp. rows) are strictly above (resp. strictly on the right of ) the decorations of the other one, where a decoration is an element of $\mathbb{O} \cup \mathbb{X}$.;

\item {\bf Stabilization/Destabilization} addition (resp. removal) of one column and one row by replacing
(resp. substituting) locally a decorated square by (resp. to) a $(2 \times 2)$-grid containing three
decorations in such a way that it remains globally a grid diagram.

\end{itemize}

\end{thm}

\section{grid diagram as a knot mosaic diagram}{\label {grid_mosaic_diagram}}
We can regard a grid diagram as a knot mosaic diagram using following map:

\begin{figure}[ht]
\begin{minipage}{1.2cm}
\includegraphics[width=\hsize]{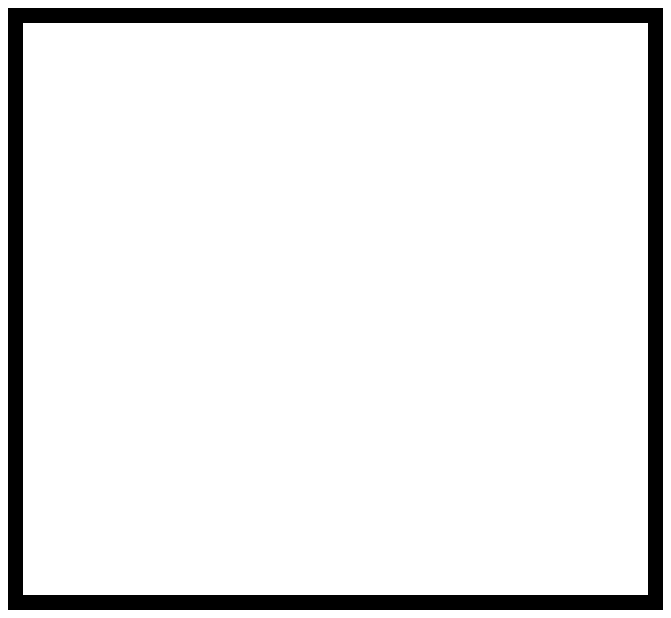}
\end{minipage}
\ $\longmapsto   $ \ 
\begin{minipage}{1cm}
\includegraphics[width=\hsize]{T_0.eps}
\end{minipage}

\begin{minipage}{1.2cm}
\includegraphics[width=\hsize]{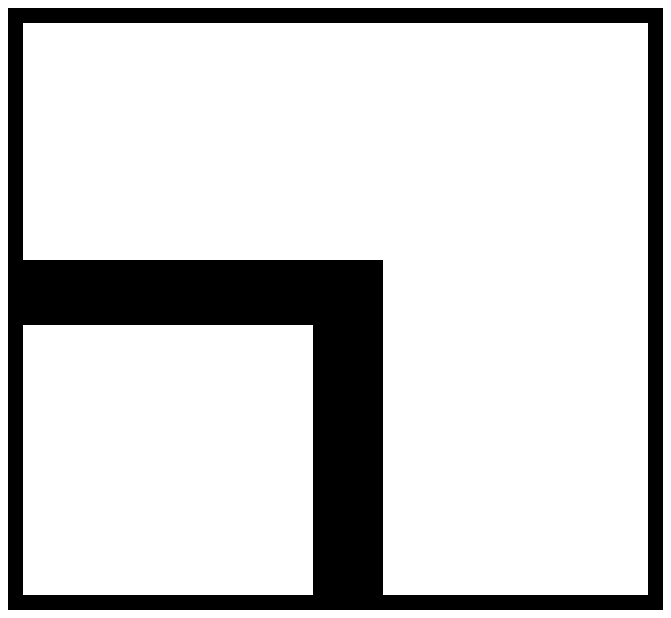}
\end{minipage}
\ $\longmapsto   $ \ 
\begin{minipage}{1cm}
\includegraphics[width=\hsize]{T_1.eps}
\end{minipage}

\begin{minipage}{1.2cm}
\includegraphics[width=\hsize]{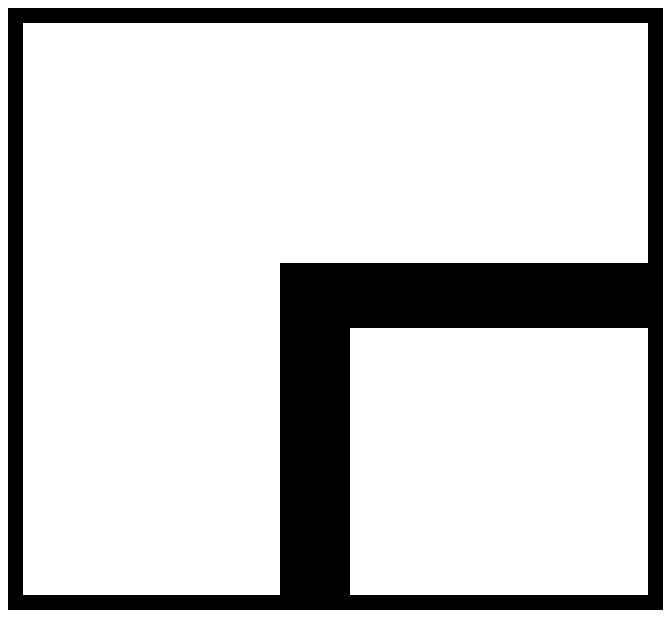}
\end{minipage}
\ $\longmapsto   $ \ 
\begin{minipage}{1cm}
\includegraphics[width=\hsize]{T_2.eps}
\end{minipage}

\begin{minipage}{1.2cm}
\includegraphics[width=\hsize]{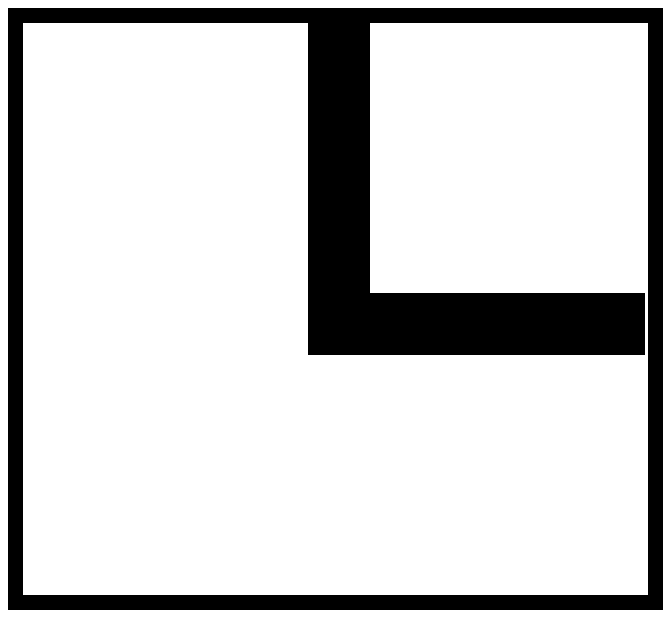}
\end{minipage}
\ $\longmapsto   $ \ 
\begin{minipage}{1cm}
\includegraphics[width=\hsize]{T_3.eps}
\end{minipage}

\begin{minipage}{1.2cm}
\includegraphics[width=\hsize]{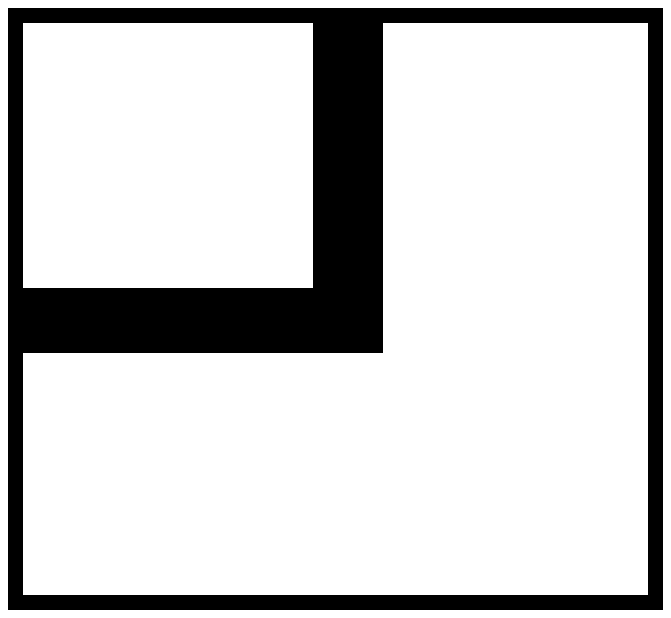}
\end{minipage}
\ $\longmapsto   $ \ 
\begin{minipage}{1cm}
\includegraphics[width=\hsize]{T_4.eps}
\end{minipage}

\begin{minipage}{1.2cm}
\includegraphics[width=\hsize]{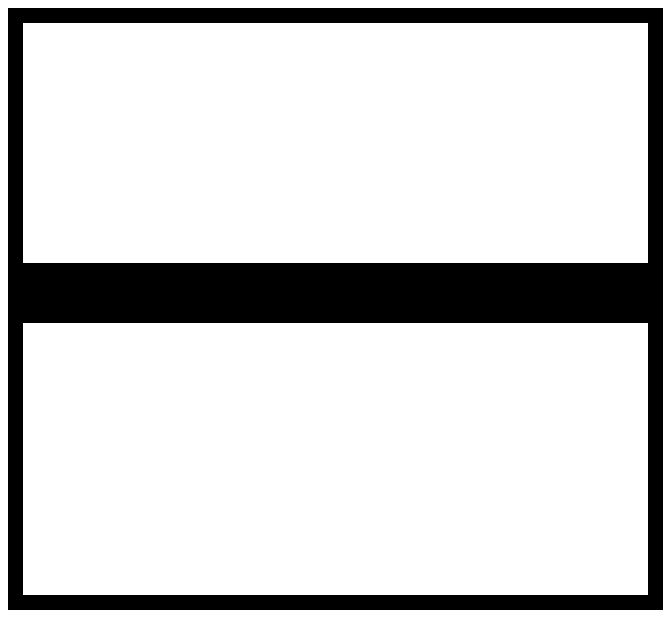}
\end{minipage}
\ $\longmapsto   $ \ 
\begin{minipage}{1cm}
\includegraphics[width=\hsize]{T_5.eps}
\end{minipage}

\begin{minipage}{1.2cm}
\includegraphics[width=\hsize]{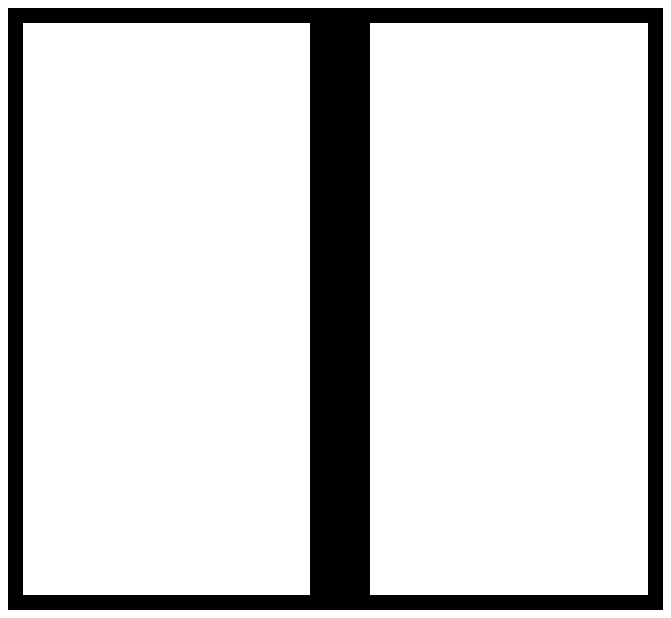}
\end{minipage}
\ $\longmapsto   $ \ 
\begin{minipage}{1cm}
\includegraphics[width=\hsize]{T_6.eps}
\end{minipage}

\begin{minipage}{1.2cm}
\includegraphics[width=\hsize]{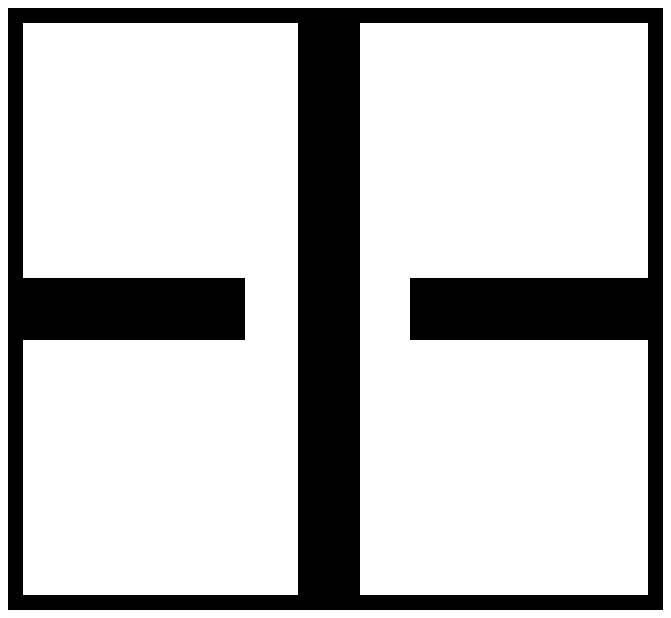}
\end{minipage}
\ $\longmapsto   $ \ 
\begin{minipage}{1cm}
\includegraphics[width=\hsize]{T_10.eps}
\end{minipage}
\end{figure}

Using this map, we identify the set of $(n \times n)$-grid diagrams with a subset of $\mathbb{K}^{(n)}$ generated by $\{T_0, T_1, T_2, T_3, T_4, T_5, T_6, T_{10} \}$ denoted by $\mathbb{G}^{(n)}$.

\section{elementary moves as mosaic moves}{\label {grid_mosaic_move}}

In this section, we will check elementary moves can be represented by mosaic moves under identification of previous section.
\begin{itemize}
\item {\bf Cyclic permutation}\\ 
In \cite{OST}, P. S. Ozsv$\acute{a}$th, Z. Sz$\acute{a}$bo, and D. P. Thurston proved the following fact, so it is sufficiently to check Commutation and Stabilization/Destabilization.

\begin{lem}{\label{OST}}
A cyclic permutation is equivalent to a sequence of commutations and (de)stabilizations.
\end{lem}
\ \\

\item {\bf Commutation}

\begin{figure}[ht]
\begin{minipage}{3cm}
\includegraphics[width=\hsize]{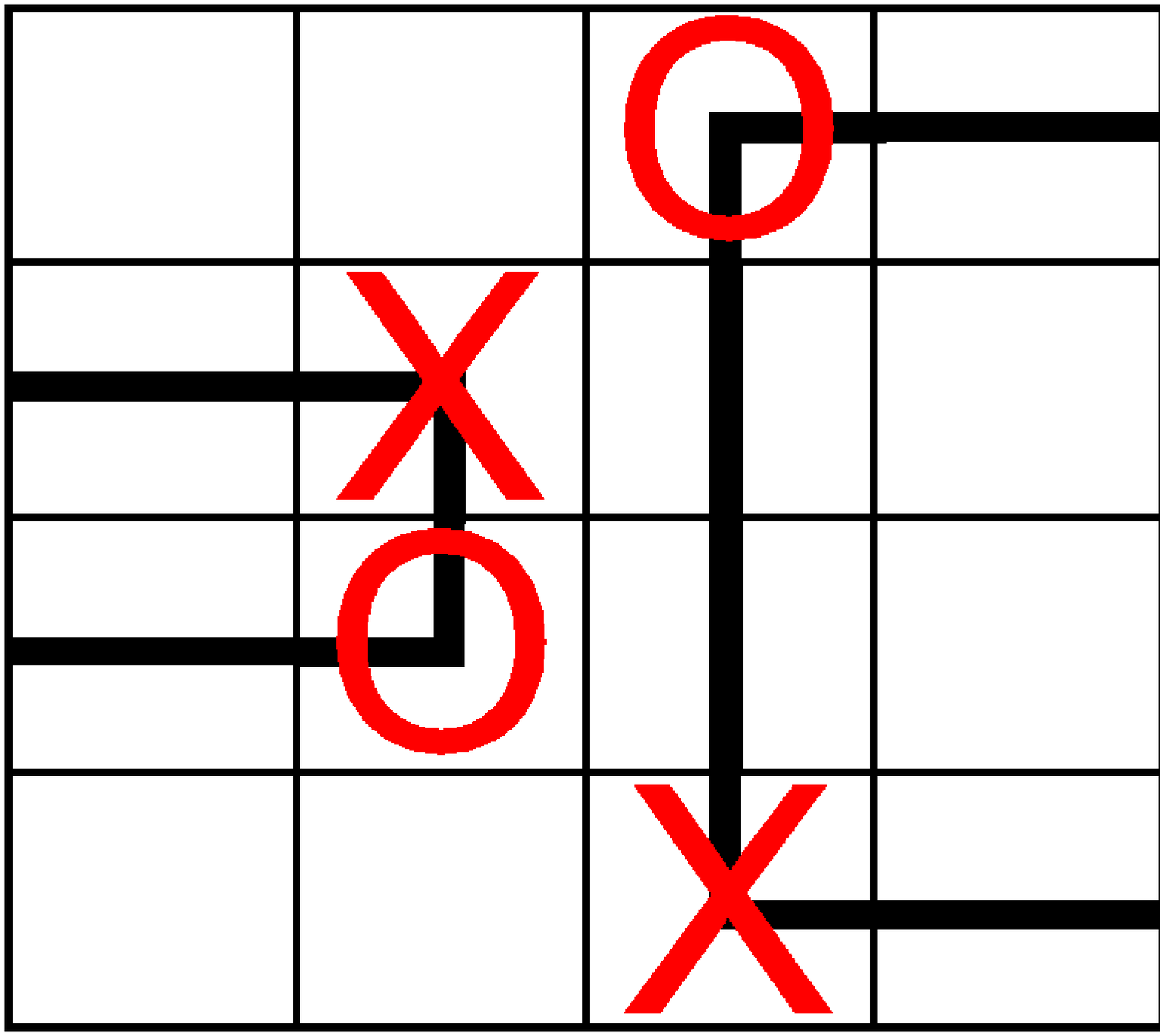}
\end{minipage}
\ $\Longleftrightarrow  $ \ 
\begin{minipage}{3cm}
\includegraphics[width=\hsize]{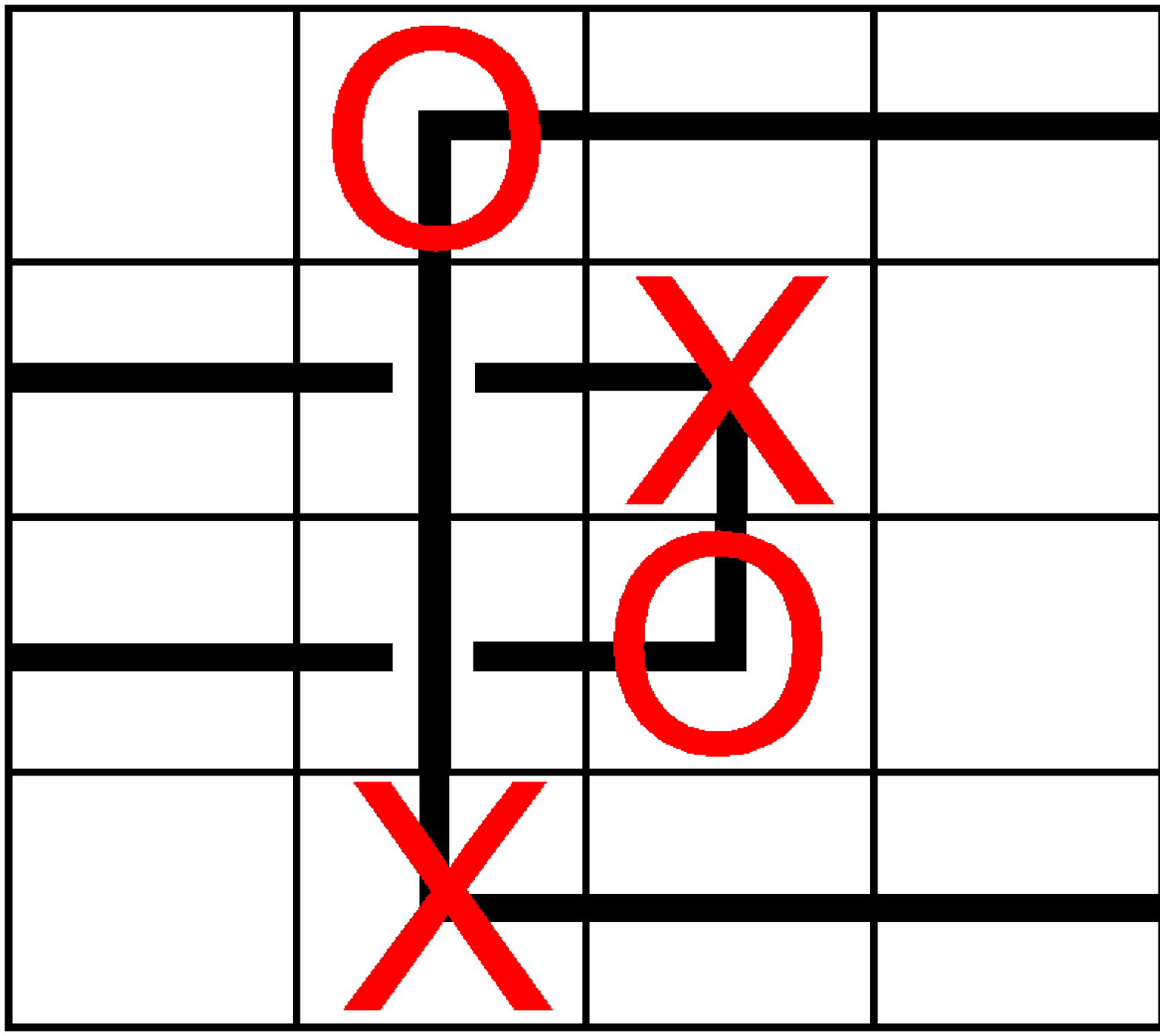}
\end{minipage}
\end{figure}

\[
\begin{array}
[c]{cccc}%
{\includegraphics[
height=0.3269in,
width=0.3269in
]%
{T_0.eps}
}%
&
{\includegraphics[
height=0.3269in,
width=0.3269in
]%
{T_0.eps}
}%
&
{\includegraphics[
height=0.3269in,
width=0.3269in
]%
{T_2.eps}
}%
&
{\includegraphics[
height=0.3269in,
width=0.3269in
]%
{T_5.eps}%
}%
\\%
{\includegraphics[
height=0.3269in,
width=0.3269in
]%
{T_5.eps}%
}%
&
{\includegraphics[
height=0.3269in,
width=0.3269in
]%
{T_1.eps}%
}%
&
{\includegraphics[
height=0.3269in,
width=0.3269in
]%
{T_6.eps}%
}%
&
{\includegraphics[
height=0.3269in,
width=0.3269in
]%
{T_0.eps}%
}%
\\%
{\includegraphics[
height=0.3269in,
width=0.3269in
]%
{T_5.eps}%
}%
&
{\includegraphics[
height=0.3269in,
width=0.3269in
]%
{T_4.eps}%
}%
&
{\includegraphics[
height=0.3269in,
width=0.3269in
]%
{T_6.eps}%
}%
&
{\includegraphics[
height=0.3269in,
width=0.3269in
]%
{T_0.eps}%
}%
\\%
{\includegraphics[
height=0.3269in,
width=0.3269in
]%
{T_0.eps}%
}%
&
{\includegraphics[
height=0.3269in,
width=0.3269in
]%
{T_0.eps}%
}%
&
{\includegraphics[
height=0.3269in,
width=0.3269in
]%
{T_3.eps}%
}%
&
{\includegraphics[
height=0.3269in,
width=0.3269in
]%
{T_5.eps}%
}%
\end{array}
\Longleftrightarrow 
\begin{array}
[c]{cccc}%
{\includegraphics[
height=0.3269in,
width=0.3269in
]%
{T_0.eps}
}%
&
{\includegraphics[
height=0.3269in,
width=0.3269in
]%
{T_2.eps}
}%
&
{\includegraphics[
height=0.3269in,
width=0.3269in
]%
{T_5.eps}
}%
&
{\includegraphics[
height=0.3269in,
width=0.3269in
]%
{T_5.eps}%
}%
\\%
{\includegraphics[
height=0.3269in,
width=0.3269in
]%
{T_5.eps}%
}%
&
{\includegraphics[
height=0.3269in,
width=0.3269in
]%
{T_7.eps}%
}%
&
{\includegraphics[
height=0.3269in,
width=0.3269in
]%
{T_1.eps}%
}%
&
{\includegraphics[
height=0.3269in,
width=0.3269in
]%
{T_0.eps}%
}%
\\%
{\includegraphics[
height=0.3269in,
width=0.3269in
]%
{T_5.eps}%
}%
&
{\includegraphics[
height=0.3269in,
width=0.3269in
]%
{T_8.eps}%
}%
&
{\includegraphics[
height=0.3269in,
width=0.3269in
]%
{T_4.eps}%
}%
&
{\includegraphics[
height=0.3269in,
width=0.3269in
]%
{T_0.eps}%
}%
\\%
{\includegraphics[
height=0.3269in,
width=0.3269in
]%
{T_0.eps}%
}%
&
{\includegraphics[
height=0.3269in,
width=0.3269in
]%
{T_3.eps}%
}%
&
{\includegraphics[
height=0.3269in,
width=0.3269in
]%
{T_5.eps}%
}%
&
{\includegraphics[
height=0.3269in,
width=0.3269in
]%
{T_5.eps}%
}%
\end{array}
\Longleftrightarrow 
\begin{array}
[c]{cccc}%
{\includegraphics[
height=0.3269in,
width=0.3269in
]%
{T_0.eps}
}%
&
{\includegraphics[
height=0.3269in,
width=0.3269in
]%
{T_2.eps}
}%
&
{\includegraphics[
height=0.3269in,
width=0.3269in
]%
{T_5.eps}
}%
&
{\includegraphics[
height=0.3269in,
width=0.3269in
]%
{T_5.eps}%
}%
\\%
{\includegraphics[
height=0.3269in,
width=0.3269in
]%
{T_5.eps}%
}%
&
{\includegraphics[
height=0.3269in,
width=0.3269in
]%
{T_10.eps}%
}%
&
{\includegraphics[
height=0.3269in,
width=0.3269in
]%
{T_1.eps}%
}%
&
{\includegraphics[
height=0.3269in,
width=0.3269in
]%
{T_0.eps}%
}%
\\%
{\includegraphics[
height=0.3269in,
width=0.3269in
]%
{T_5.eps}%
}%
&
{\includegraphics[
height=0.3269in,
width=0.3269in
]%
{T_10.eps}%
}%
&
{\includegraphics[
height=0.3269in,
width=0.3269in
]%
{T_4.eps}%
}%
&
{\includegraphics[
height=0.3269in,
width=0.3269in
]%
{T_0.eps}%
}%
\\%
{\includegraphics[
height=0.3269in,
width=0.3269in
]%
{T_0.eps}%
}%
&
{\includegraphics[
height=0.3269in,
width=0.3269in
]%
{T_3.eps}%
}%
&
{\includegraphics[
height=0.3269in,
width=0.3269in
]%
{T_5.eps}%
}%
&
{\includegraphics[
height=0.3269in,
width=0.3269in
]%
{T_5.eps}%
}%
\end{array}
\]

\begin{figure}[ht]
\begin{minipage}{3.5cm}
\includegraphics[width=\hsize]{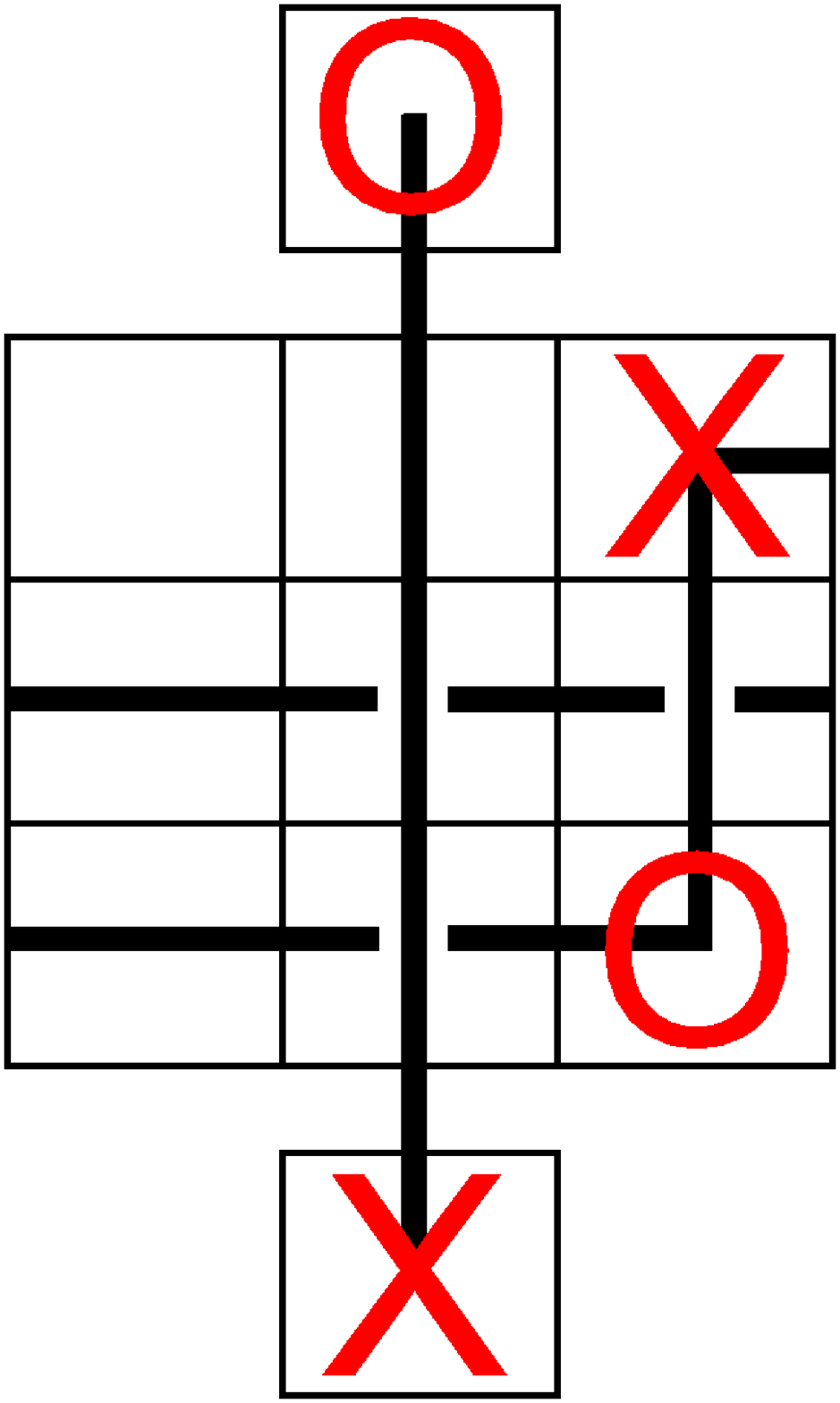}
\end{minipage}
\ $\Longleftrightarrow  $ \ 
\begin{minipage}{3.5cm}
\includegraphics[width=\hsize]{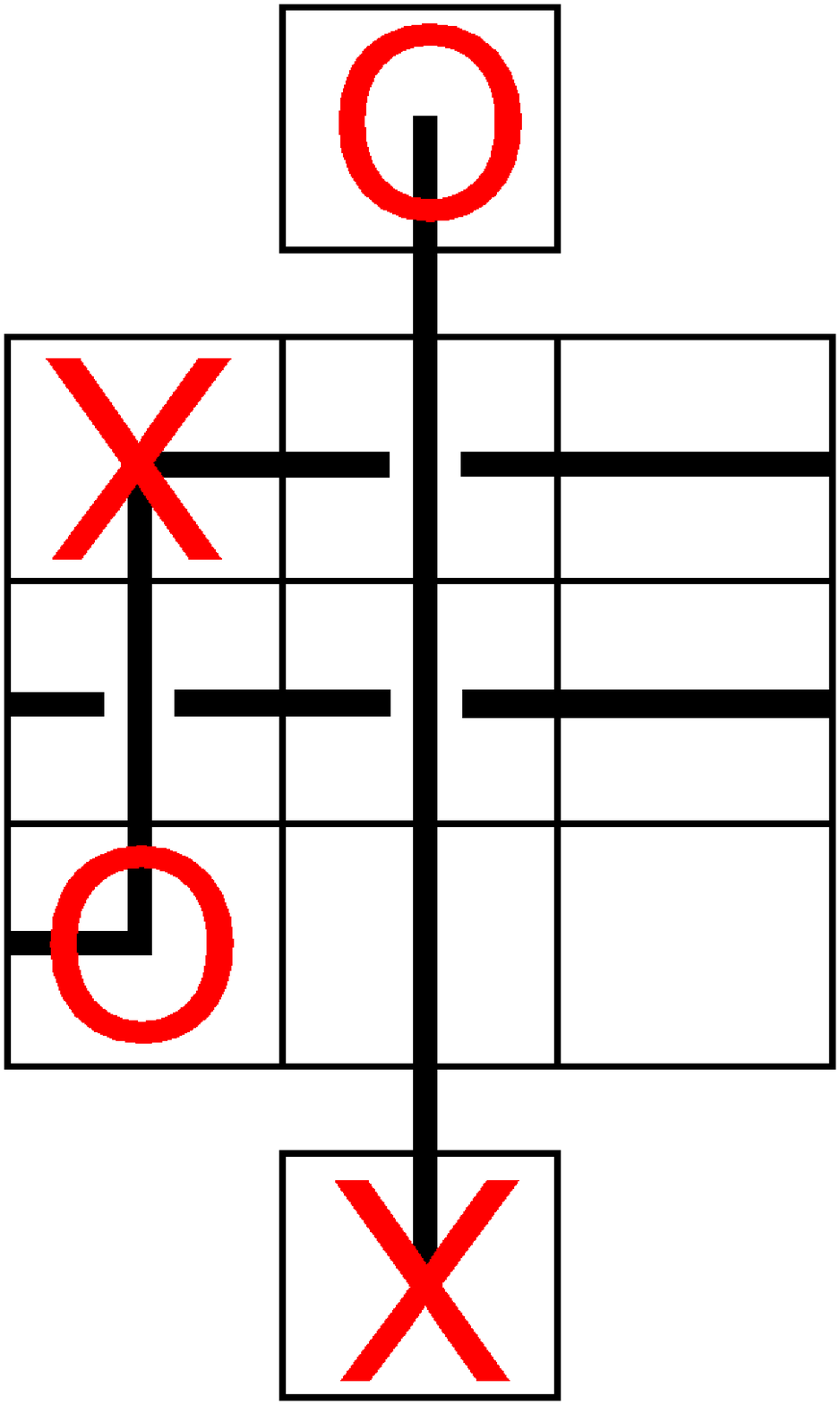}
\end{minipage}
\end{figure}

\[
\begin{array}
[c]{ccc}%
{\includegraphics[
height=0.3269in,
width=0.3269in
]%
{T_0.eps}
}%
&
{\includegraphics[
height=0.3269in,
width=0.3269in
]%
{T_6.eps}
}%
&
{\includegraphics[
height=0.3269in,
width=0.3269in
]%
{T_2.eps}
}%
\\%
{\includegraphics[
height=0.3269in,
width=0.3269in
]%
{T_5.eps}%
}%
&
{\includegraphics[
height=0.3269in,
width=0.3269in
]%
{T_10.eps}%
}%
&
{\includegraphics[
height=0.3269in,
width=0.3269in
]%
{T_10.eps}%
}%
\\%
{\includegraphics[
height=0.3269in,
width=0.3269in
]%
{T_5.eps}%
}%
&
{\includegraphics[
height=0.3269in,
width=0.3269in
]%
{T_10.eps}%
}%
&
{\includegraphics[
height=0.3269in,
width=0.3269in
]%
{T_4.eps}%
}%
\end{array}
\Longleftrightarrow 
\begin{array}
[c]{ccc}%
{\includegraphics[
height=0.3269in,
width=0.3269in
]%
{T_2.eps}
}%
&
{\includegraphics[
height=0.3269in,
width=0.3269in
]%
{T_10.eps}
}%
&
{\includegraphics[
height=0.3269in,
width=0.3269in
]%
{T_5.eps}
}%
\\%
{\includegraphics[
height=0.3269in,
width=0.3269in
]%
{T_10.eps}%
}%
&
{\includegraphics[
height=0.3269in,
width=0.3269in
]%
{T_10.eps}%
}%
&
{\includegraphics[
height=0.3269in,
width=0.3269in
]%
{T_5.eps}%
}%
\\%
{\includegraphics[
height=0.3269in,
width=0.3269in
]%
{T_4.eps}%
}%
&
{\includegraphics[
height=0.3269in,
width=0.3269in
]%
{T_6.eps}%
}%
&
{\includegraphics[
height=0.3269in,
width=0.3269in
]%
{T_0.eps}%
}%
\end{array}
\]

\item {\bf Stabilization/Destabilization}

\begin{figure}[ht]
\begin{minipage}{2.5cm}
\includegraphics[width=\hsize]{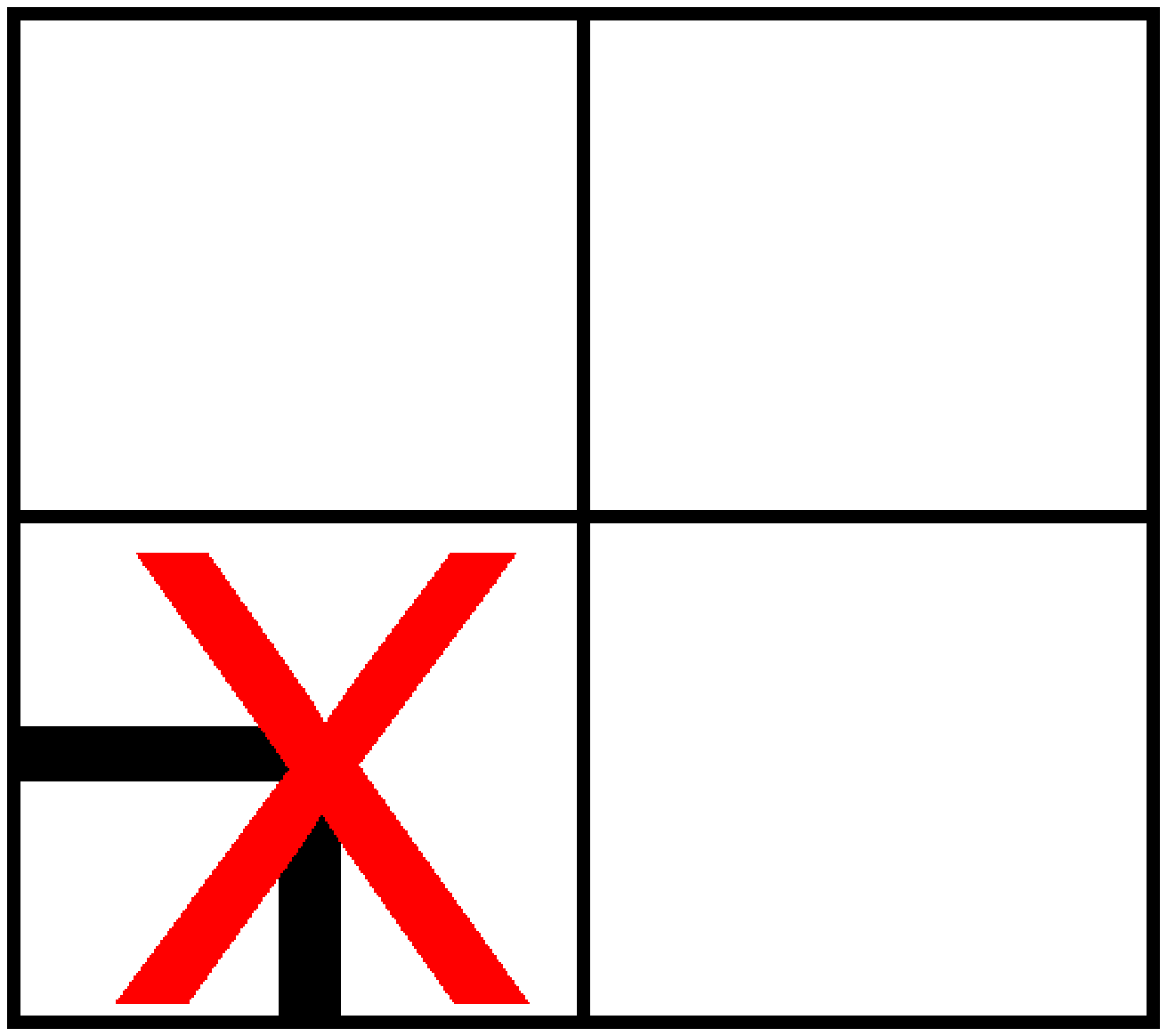}
\end{minipage}
\ $\Longleftrightarrow  $ \ 
\begin{minipage}{2.5cm}
\includegraphics[width=\hsize]{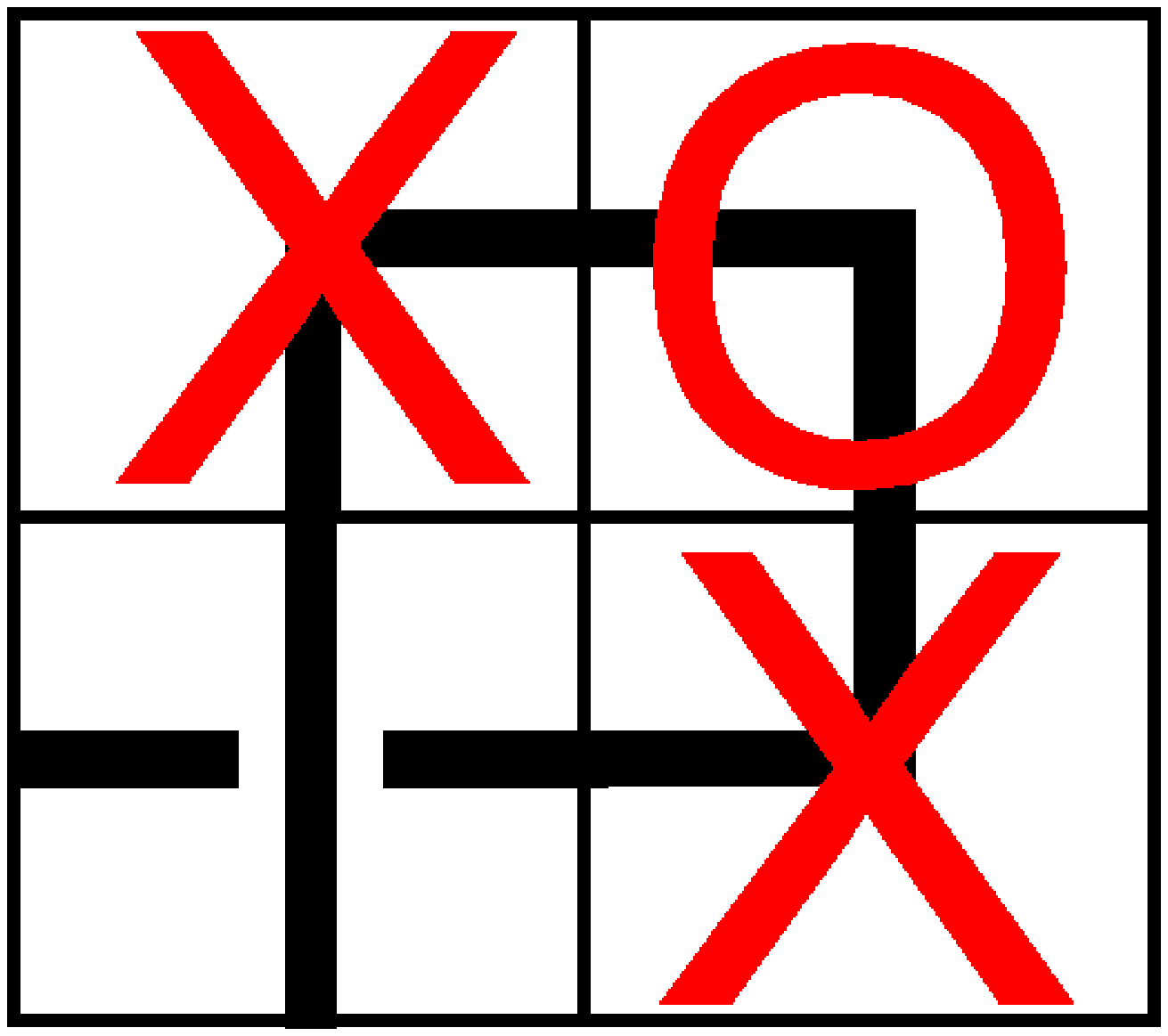}
\end{minipage}
\end{figure}

\[
\begin{array}
[c]{cc}%
{\includegraphics[
height=0.3269in,
width=0.3269in
]%
{T_0.eps}
}%
&
{\includegraphics[
height=0.3269in,
width=0.3269in
]%
{T_0.eps}
}%
\\%
{\includegraphics[
height=0.3269in,
width=0.3269in
]%
{T_1.eps}%
}%
&
{\includegraphics[
height=0.3269in,
width=0.3269in
]%
{T_0.eps}%
}%
\end{array}
\Longleftrightarrow 
\begin{array}
[c]{cc}%
{\includegraphics[
height=0.3269in,
width=0.3269in
]%
{T_2.eps}
}%
&
{\includegraphics[
height=0.3269in,
width=0.3269in
]%
{T_1.eps}
}%
\\%
{\includegraphics[
height=0.3269in,
width=0.3269in
]%
{T_8.eps}%
}%
&
{\includegraphics[
height=0.3269in,
width=0.3269in
]%
{T_4.eps}%
}%
\end{array}
\Longleftrightarrow 
\begin{array}
[c]{cc}%
{\includegraphics[
height=0.3269in,
width=0.3269in
]%
{T_2.eps}
}%
&
{\includegraphics[
height=0.3269in,
width=0.3269in
]%
{T_1.eps}
}%
\\%
{\includegraphics[
height=0.3269in,
width=0.3269in
]%
{T_10.eps}%
}%
&
{\includegraphics[
height=0.3269in,
width=0.3269in
]%
{T_4.eps}%
}%
\end{array}
\]

\begin{figure}[ht]
\begin{minipage}{2.5cm}
\includegraphics[width=\hsize]{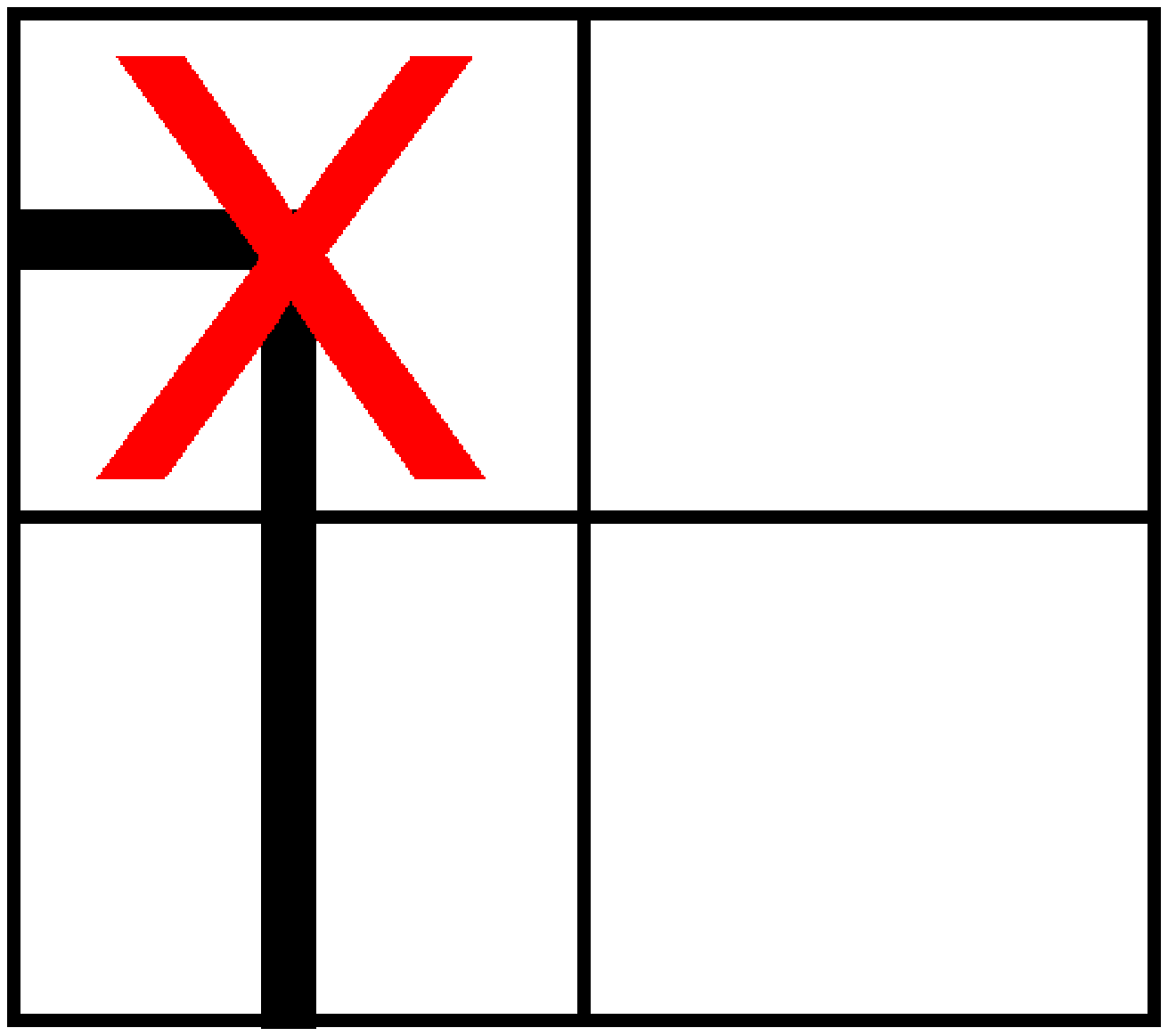}
\end{minipage}
\ $\Longleftrightarrow  $ \ 
\begin{minipage}{2.5cm}
\includegraphics[width=\hsize]{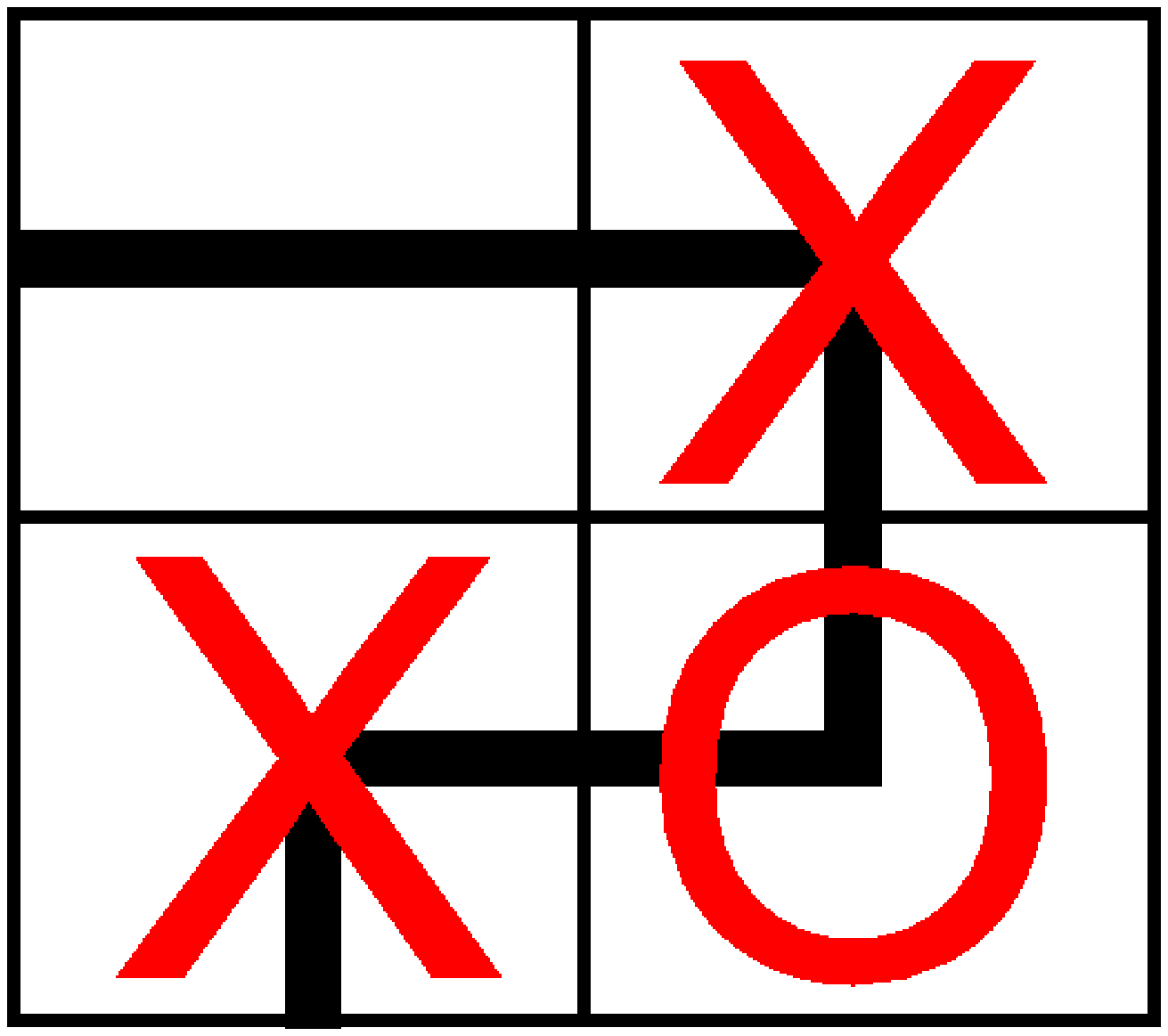}
\end{minipage}
\end{figure}

\[
\begin{array}
[c]{cc}%
{\includegraphics[
height=0.3269in,
width=0.3269in
]%
{T_1.eps}
}%
&
{\includegraphics[
height=0.3269in,
width=0.3269in
]%
{T_0.eps}
}%
\\%
{\includegraphics[
height=0.3269in,
width=0.3269in
]%
{T_6.eps}%
}%
&
{\includegraphics[
height=0.3269in,
width=0.3269in
]%
{T_0.eps}%
}%
\end{array}
\Longleftrightarrow 
\begin{array}
[c]{cc}%
{\includegraphics[
height=0.3269in,
width=0.3269in
]%
{T_5.eps}
}%
&
{\includegraphics[
height=0.3269in,
width=0.3269in
]%
{T_1.eps}
}%
\\%
{\includegraphics[
height=0.3269in,
width=0.3269in
]%
{T_2.eps}%
}%
&
{\includegraphics[
height=0.3269in,
width=0.3269in
]%
{T_4.eps}%
}%
\end{array}
\]

\begin{figure}[ht]
\begin{minipage}{2.3cm}
\includegraphics[width=\hsize]{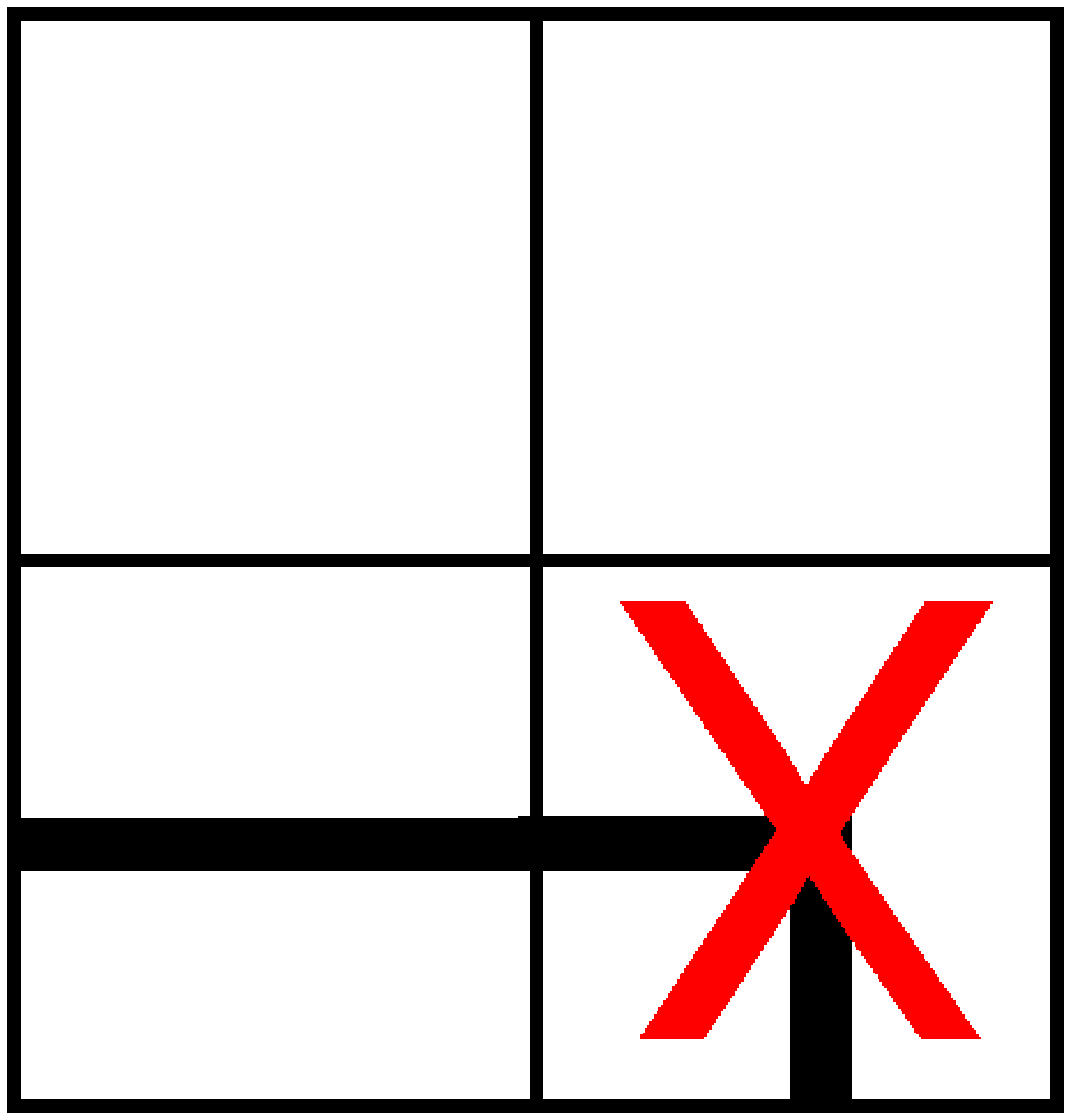}
\end{minipage}
\ $\Longleftrightarrow  $ \ 
\begin{minipage}{2.5cm}
\includegraphics[width=\hsize]{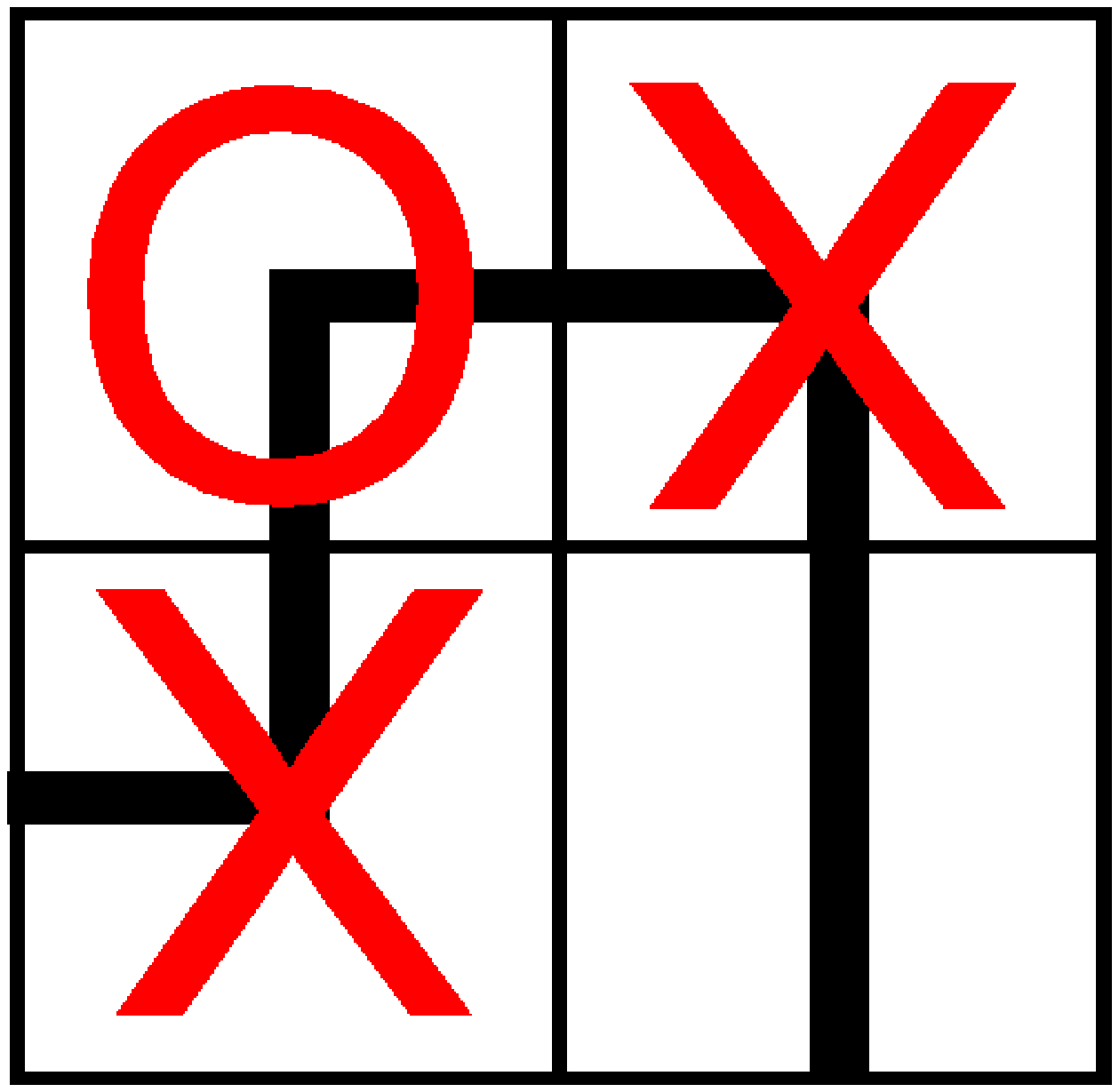}
\end{minipage}
\end{figure}

\[
\begin{array}
[c]{cc}%
{\includegraphics[
height=0.3269in,
width=0.3269in
]%
{T_0.eps}
}%
&
{\includegraphics[
height=0.3269in,
width=0.3269in
]%
{T_0.eps}
}%
\\%
{\includegraphics[
height=0.3269in,
width=0.3269in
]%
{T_5.eps}%
}%
&
{\includegraphics[
height=0.3269in,
width=0.3269in
]%
{T_1.eps}%
}%
\end{array}
\Longleftrightarrow 
\begin{array}
[c]{cc}%
{\includegraphics[
height=0.3269in,
width=0.3269in
]%
{T_2.eps}
}%
&
{\includegraphics[
height=0.3269in,
width=0.3269in
]%
{T_1.eps}
}%
\\%
{\includegraphics[
height=0.3269in,
width=0.3269in
]%
{T_4.eps}%
}%
&
{\includegraphics[
height=0.3269in,
width=0.3269in
]%
{T_6.eps}%
}%
\end{array}
\]

\begin{figure}[ht]
\begin{minipage}{2.5cm}
\includegraphics[width=\hsize]{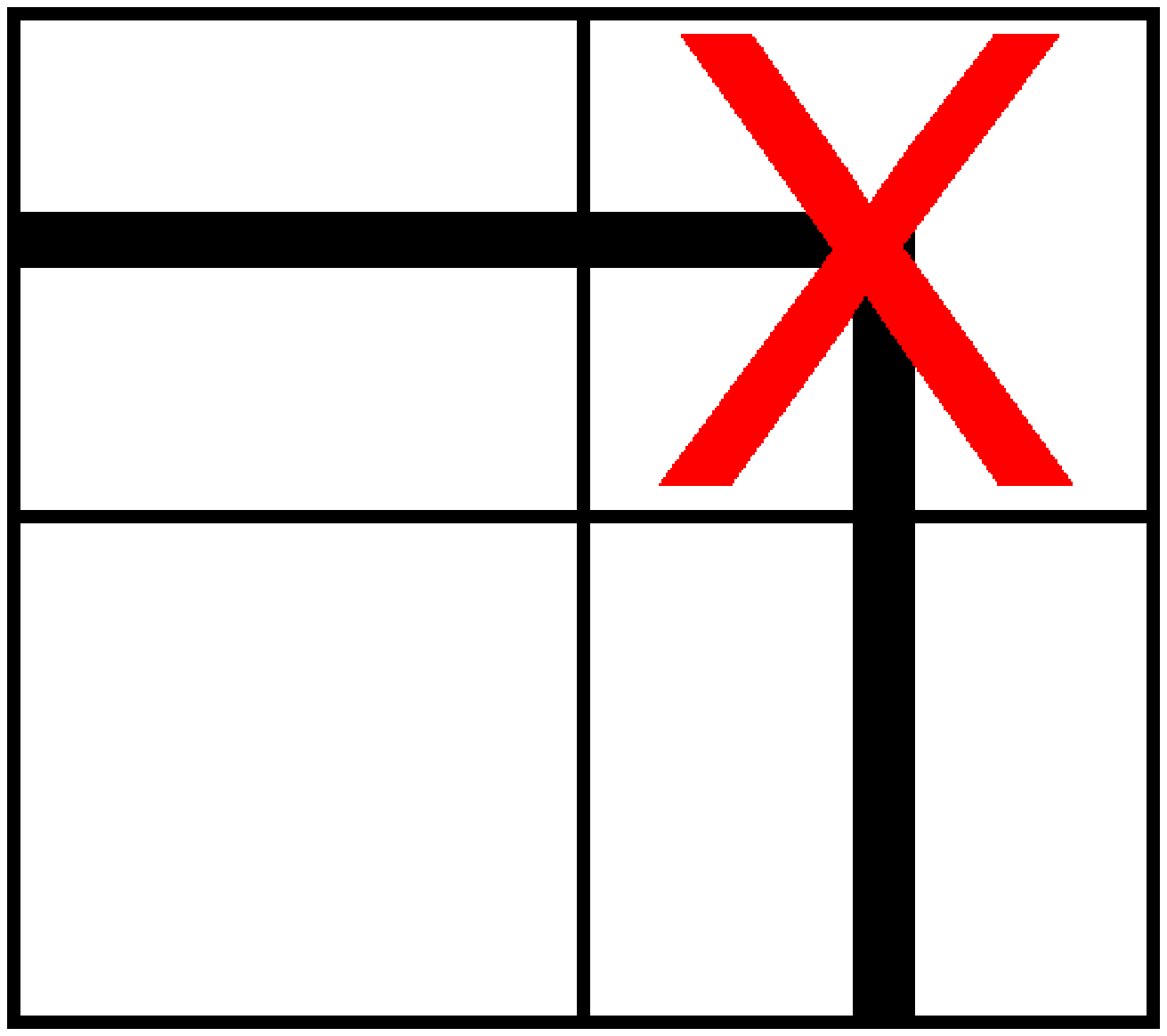}
\end{minipage}
\ $\Longleftrightarrow  $ \ 
\begin{minipage}{2.5cm}
\includegraphics[width=\hsize]{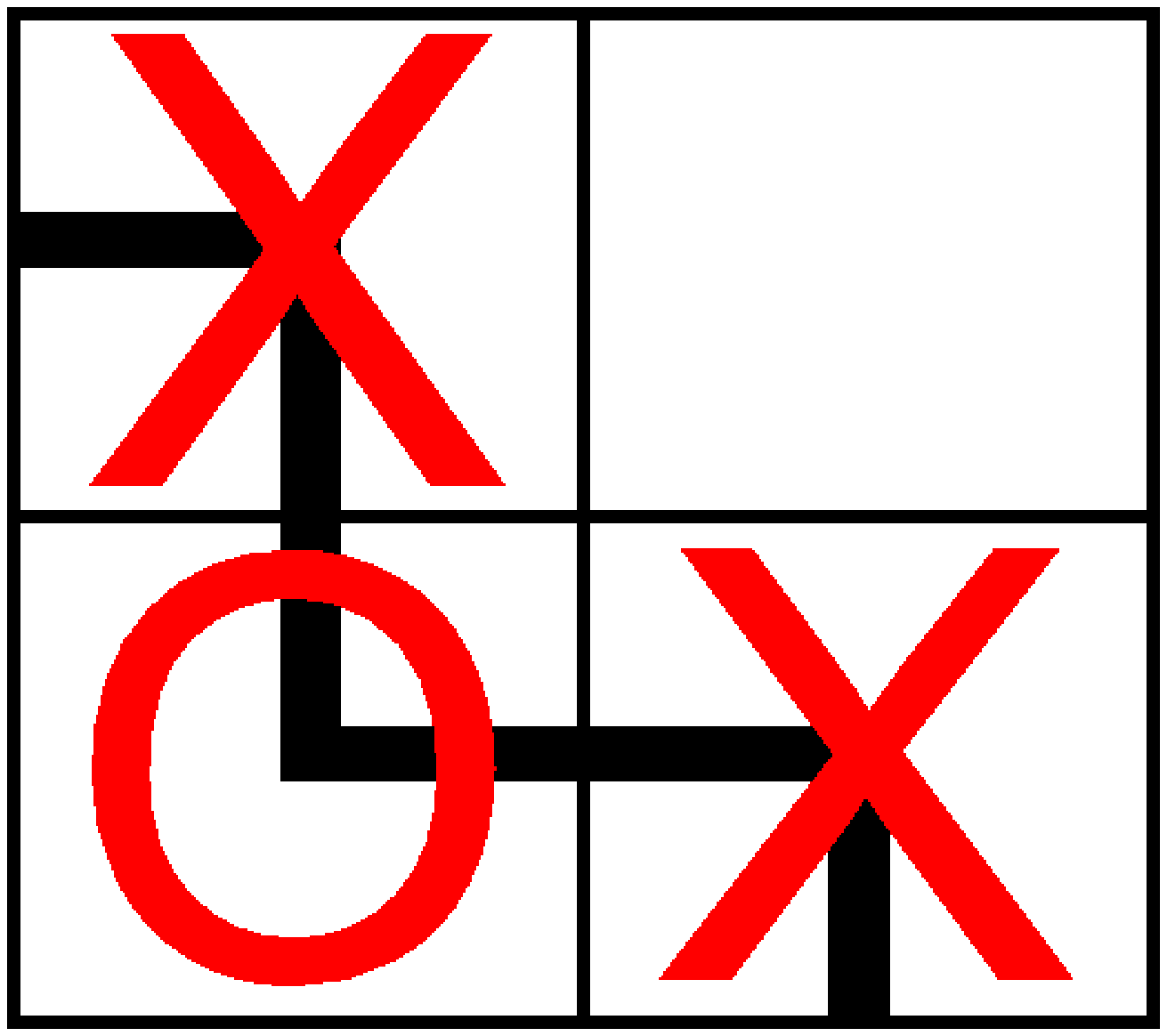}
\end{minipage}
\end{figure}

\[
\begin{array}
[c]{cc}%
{\includegraphics[
height=0.3269in,
width=0.3269in
]%
{T_5.eps}
}%
&
{\includegraphics[
height=0.3269in,
width=0.3269in
]%
{T_1.eps}
}%
\\%
{\includegraphics[
height=0.3269in,
width=0.3269in
]%
{T_0.eps}%
}%
&
{\includegraphics[
height=0.3269in,
width=0.3269in
]%
{T_6.eps}%
}%
\end{array}
\Longleftrightarrow 
\begin{array}
[c]{cc}%
{\includegraphics[
height=0.3269in,
width=0.3269in
]%
{T_1.eps}
}%
&
{\includegraphics[
height=0.3269in,
width=0.3269in
]%
{T_0.eps}
}%
\\%
{\includegraphics[
height=0.3269in,
width=0.3269in
]%
{T_3.eps}%
}%
&
{\includegraphics[
height=0.3269in,
width=0.3269in
]%
{T_1.eps}%
}%
\end{array}
\]

\end{itemize}

As an immediate corollary of previous and this sections, we get the following proposition.
\begin{pro}{\label{Ku1}}
$( \mathbb{G},\mathbb{E})\subset ( \mathbb{K},\mathbb{A}) $, where $\mathbb{E}$ is the group of all elementary moves of $\mathbb{G}$.
\end{pro}

\section{$5 \times$ zoom}{\label {zoom}}

\begin{defn}
We define the {\bf zoom with ratio of $5 \times$  }
\[%
\begin{array}
[c]{rrr}%
z_{5 \times}:\mathbb{M}^{(n)} & \longrightarrow & \mathbb{M}^{(5n)}\\
M^{(n)} & \longmapsto & M^{(5n)}%
\end{array}
\]
as%

\[
T_0=
\begin{array}
[c]{c}%
{\includegraphics[
height=0.25in,
width=0.25in
]%
{T_0.eps}
}
\end{array}
\longmapsto %
\begin{array}
[c]{ccccc}%
{\includegraphics[
height=0.25in,
width=0.25in
]%
{T_0.eps}
}%
&
{\includegraphics[
height=0.25in,
width=0.25in
]%
{T_0.eps}
}%
&
{\includegraphics[
height=0.25in,
width=0.25in
]%
{T_0.eps}
}%
&
{\includegraphics[
height=0.25in,
width=0.25in
]%
{T_0.eps}%
}%
&
{\includegraphics[
height=0.25in,
width=0.25in
]%
{T_0.eps}%
}%
\\%
{\includegraphics[
height=0.25in,
width=0.25in
]%
{T_0.eps}%
}%
&
{\includegraphics[
height=0.25in,
width=0.25in
]%
{T_0.eps}%
}%
&
{\includegraphics[
height=0.25in,
width=0.25in
]%
{T_0.eps}%
}%
&
{\includegraphics[
height=0.25in,
width=0.25in
]%
{T_0.eps}%
}%
&
{\includegraphics[
height=0.25in,
width=0.25in
]%
{T_0.eps}%
}%
\\%
{\includegraphics[
height=0.25in,
width=0.25in
]%
{T_0.eps}%
}%
&
{\includegraphics[
height=0.25in,
width=0.25in
]%
{T_0.eps}%
}%
&
{\includegraphics[
height=0.25in,
width=0.25in
]%
{T_0.eps}%
}%
&
{\includegraphics[
height=0.25in,
width=0.25in
]%
{T_0.eps}%
}%
&
{\includegraphics[
height=0.25in,
width=0.25in
]%
{T_0.eps}%
}%
\\%
{\includegraphics[
height=0.25in,
width=0.25in
]%
{T_0.eps}%
}%
&
{\includegraphics[
height=0.25in,
width=0.25in
]%
{T_0.eps}%
}%
&
{\includegraphics[
height=0.25in,
width=0.25in
]%
{T_0.eps}%
}%
&
{\includegraphics[
height=0.25in,
width=0.25in
]%
{T_0.eps}%
}%
&
{\includegraphics[
height=0.25in,
width=0.25in
]%
{T_0.eps}%
}%
\\%
{\includegraphics[
height=0.25in,
width=0.25in
]%
{T_0.eps}%
}%
&
{\includegraphics[
height=0.25in,
width=0.25in
]%
{T_0.eps}%
}%
&
{\includegraphics[
height=0.25in,
width=0.25in
]%
{T_0.eps}%
}%
&
{\includegraphics[
height=0.25in,
width=0.25in
]%
{T_0.eps}%
}%
&
{\includegraphics[
height=0.25in,
width=0.25in
]%
{T_0.eps}%
}%
\end{array}
\]

\[
T_1=
\begin{array}
[c]{c}%
{\includegraphics[
height=0.25in,
width=0.25in
]%
{T_1.eps}
}
\end{array}
\longmapsto %
\begin{array}
[c]{ccccc}%
{\includegraphics[
height=0.25in,
width=0.25in
]%
{T_0.eps}
}%
&
{\includegraphics[
height=0.25in,
width=0.25in
]%
{T_0.eps}
}%
&
{\includegraphics[
height=0.25in,
width=0.25in
]%
{T_0.eps}
}%
&
{\includegraphics[
height=0.25in,
width=0.25in
]%
{T_0.eps}%
}%
&
{\includegraphics[
height=0.25in,
width=0.25in
]%
{T_0.eps}%
}%
\\%
{\includegraphics[
height=0.25in,
width=0.25in
]%
{T_0.eps}%
}%
&
{\includegraphics[
height=0.25in,
width=0.25in
]%
{T_0.eps}%
}%
&
{\includegraphics[
height=0.25in,
width=0.25in
]%
{T_0.eps}%
}%
&
{\includegraphics[
height=0.25in,
width=0.25in
]%
{T_0.eps}%
}%
&
{\includegraphics[
height=0.25in,
width=0.25in
]%
{T_0.eps}%
}%
\\%
{\includegraphics[
height=0.25in,
width=0.25in
]%
{T_5.eps}%
}%
&
{\includegraphics[
height=0.25in,
width=0.25in
]%
{T_5.eps}%
}%
&
{\includegraphics[
height=0.25in,
width=0.25in
]%
{T_1.eps}%
}%
&
{\includegraphics[
height=0.25in,
width=0.25in
]%
{T_0.eps}%
}%
&
{\includegraphics[
height=0.25in,
width=0.25in
]%
{T_0.eps}%
}%
\\%
{\includegraphics[
height=0.25in,
width=0.25in
]%
{T_0.eps}%
}%
&
{\includegraphics[
height=0.25in,
width=0.25in
]%
{T_0.eps}%
}%
&
{\includegraphics[
height=0.25in,
width=0.25in
]%
{T_6.eps}%
}%
&
{\includegraphics[
height=0.25in,
width=0.25in
]%
{T_0.eps}%
}%
&
{\includegraphics[
height=0.25in,
width=0.25in
]%
{T_0.eps}%
}%
\\%
{\includegraphics[
height=0.25in,
width=0.25in
]%
{T_0.eps}%
}%
&
{\includegraphics[
height=0.25in,
width=0.25in
]%
{T_0.eps}%
}%
&
{\includegraphics[
height=0.25in,
width=0.25in
]%
{T_6.eps}%
}%
&
{\includegraphics[
height=0.25in,
width=0.25in
]%
{T_0.eps}%
}%
&
{\includegraphics[
height=0.25in,
width=0.25in
]%
{T_0.eps}%
}%
\end{array}
\]

\[
T_2=
\begin{array}
[c]{c}%
{\includegraphics[
height=0.25in,
width=0.25in
]%
{T_2.eps}
}
\end{array}
\longmapsto %
\begin{array}
[c]{ccccc}%
{\includegraphics[
height=0.25in,
width=0.25in
]%
{T_0.eps}
}%
&
{\includegraphics[
height=0.25in,
width=0.25in
]%
{T_0.eps}
}%
&
{\includegraphics[
height=0.25in,
width=0.25in
]%
{T_0.eps}
}%
&
{\includegraphics[
height=0.25in,
width=0.25in
]%
{T_0.eps}%
}%
&
{\includegraphics[
height=0.25in,
width=0.25in
]%
{T_0.eps}%
}%
\\%
{\includegraphics[
height=0.25in,
width=0.25in
]%
{T_0.eps}%
}%
&
{\includegraphics[
height=0.25in,
width=0.25in
]%
{T_0.eps}%
}%
&
{\includegraphics[
height=0.25in,
width=0.25in
]%
{T_0.eps}%
}%
&
{\includegraphics[
height=0.25in,
width=0.25in
]%
{T_0.eps}%
}%
&
{\includegraphics[
height=0.25in,
width=0.25in
]%
{T_0.eps}%
}%
\\%
{\includegraphics[
height=0.25in,
width=0.25in
]%
{T_0.eps}%
}%
&
{\includegraphics[
height=0.25in,
width=0.25in
]%
{T_0.eps}%
}%
&
{\includegraphics[
height=0.25in,
width=0.25in
]%
{T_2.eps}%
}%
&
{\includegraphics[
height=0.25in,
width=0.25in
]%
{T_5.eps}%
}%
&
{\includegraphics[
height=0.25in,
width=0.25in
]%
{T_5.eps}%
}%
\\%
{\includegraphics[
height=0.25in,
width=0.25in
]%
{T_0.eps}%
}%
&
{\includegraphics[
height=0.25in,
width=0.25in
]%
{T_0.eps}%
}%
&
{\includegraphics[
height=0.25in,
width=0.25in
]%
{T_6.eps}%
}%
&
{\includegraphics[
height=0.25in,
width=0.25in
]%
{T_0.eps}%
}%
&
{\includegraphics[
height=0.25in,
width=0.25in
]%
{T_0.eps}%
}%
\\%
{\includegraphics[
height=0.25in,
width=0.25in
]%
{T_0.eps}%
}%
&
{\includegraphics[
height=0.25in,
width=0.25in
]%
{T_0.eps}%
}%
&
{\includegraphics[
height=0.25in,
width=0.25in
]%
{T_6.eps}%
}%
&
{\includegraphics[
height=0.25in,
width=0.25in
]%
{T_0.eps}%
}%
&
{\includegraphics[
height=0.25in,
width=0.25in
]%
{T_0.eps}%
}%
\end{array}
\]

\[
T_3=
\begin{array}
[c]{c}%
{\includegraphics[
height=0.25in,
width=0.25in
]%
{T_3.eps}
}
\end{array}
\longmapsto %
\begin{array}
[c]{ccccc}%
{\includegraphics[
height=0.25in,
width=0.25in
]%
{T_0.eps}
}%
&
{\includegraphics[
height=0.25in,
width=0.25in
]%
{T_0.eps}
}%
&
{\includegraphics[
height=0.25in,
width=0.25in
]%
{T_6.eps}
}%
&
{\includegraphics[
height=0.25in,
width=0.25in
]%
{T_0.eps}%
}%
&
{\includegraphics[
height=0.25in,
width=0.25in
]%
{T_0.eps}%
}%
\\%
{\includegraphics[
height=0.25in,
width=0.25in
]%
{T_0.eps}%
}%
&
{\includegraphics[
height=0.25in,
width=0.25in
]%
{T_0.eps}%
}%
&
{\includegraphics[
height=0.25in,
width=0.25in
]%
{T_6.eps}%
}%
&
{\includegraphics[
height=0.25in,
width=0.25in
]%
{T_0.eps}%
}%
&
{\includegraphics[
height=0.25in,
width=0.25in
]%
{T_0.eps}%
}%
\\%
{\includegraphics[
height=0.25in,
width=0.25in
]%
{T_0.eps}%
}%
&
{\includegraphics[
height=0.25in,
width=0.25in
]%
{T_0.eps}%
}%
&
{\includegraphics[
height=0.25in,
width=0.25in
]%
{T_3.eps}%
}%
&
{\includegraphics[
height=0.25in,
width=0.25in
]%
{T_5.eps}%
}%
&
{\includegraphics[
height=0.25in,
width=0.25in
]%
{T_5.eps}%
}%
\\%
{\includegraphics[
height=0.25in,
width=0.25in
]%
{T_0.eps}%
}%
&
{\includegraphics[
height=0.25in,
width=0.25in
]%
{T_0.eps}%
}%
&
{\includegraphics[
height=0.25in,
width=0.25in
]%
{T_0.eps}%
}%
&
{\includegraphics[
height=0.25in,
width=0.25in
]%
{T_0.eps}%
}%
&
{\includegraphics[
height=0.25in,
width=0.25in
]%
{T_0.eps}%
}%
\\%
{\includegraphics[
height=0.25in,
width=0.25in
]%
{T_0.eps}%
}%
&
{\includegraphics[
height=0.25in,
width=0.25in
]%
{T_0.eps}%
}%
&
{\includegraphics[
height=0.25in,
width=0.25in
]%
{T_0.eps}%
}%
&
{\includegraphics[
height=0.25in,
width=0.25in
]%
{T_0.eps}%
}%
&
{\includegraphics[
height=0.25in,
width=0.25in
]%
{T_0.eps}%
}%
\end{array}
\]

\[
T_4=
\begin{array}
[c]{c}%
{\includegraphics[
height=0.25in,
width=0.25in
]%
{T_4.eps}
}
\end{array}
\longmapsto %
\begin{array}
[c]{ccccc}%
{\includegraphics[
height=0.25in,
width=0.25in
]%
{T_0.eps}
}%
&
{\includegraphics[
height=0.25in,
width=0.25in
]%
{T_0.eps}
}%
&
{\includegraphics[
height=0.25in,
width=0.25in
]%
{T_6.eps}
}%
&
{\includegraphics[
height=0.25in,
width=0.25in
]%
{T_0.eps}%
}%
&
{\includegraphics[
height=0.25in,
width=0.25in
]%
{T_0.eps}%
}%
\\%
{\includegraphics[
height=0.25in,
width=0.25in
]%
{T_0.eps}%
}%
&
{\includegraphics[
height=0.25in,
width=0.25in
]%
{T_0.eps}%
}%
&
{\includegraphics[
height=0.25in,
width=0.25in
]%
{T_6.eps}%
}%
&
{\includegraphics[
height=0.25in,
width=0.25in
]%
{T_0.eps}%
}%
&
{\includegraphics[
height=0.25in,
width=0.25in
]%
{T_0.eps}%
}%
\\%
{\includegraphics[
height=0.25in,
width=0.25in
]%
{T_5.eps}%
}%
&
{\includegraphics[
height=0.25in,
width=0.25in
]%
{T_5.eps}%
}%
&
{\includegraphics[
height=0.25in,
width=0.25in
]%
{T_4.eps}%
}%
&
{\includegraphics[
height=0.25in,
width=0.25in
]%
{T_0.eps}%
}%
&
{\includegraphics[
height=0.25in,
width=0.25in
]%
{T_0.eps}%
}%
\\%
{\includegraphics[
height=0.25in,
width=0.25in
]%
{T_0.eps}%
}%
&
{\includegraphics[
height=0.25in,
width=0.25in
]%
{T_0.eps}%
}%
&
{\includegraphics[
height=0.25in,
width=0.25in
]%
{T_0.eps}%
}%
&
{\includegraphics[
height=0.25in,
width=0.25in
]%
{T_0.eps}%
}%
&
{\includegraphics[
height=0.25in,
width=0.25in
]%
{T_0.eps}%
}%
\\%
{\includegraphics[
height=0.25in,
width=0.25in
]%
{T_0.eps}%
}%
&
{\includegraphics[
height=0.25in,
width=0.25in
]%
{T_0.eps}%
}%
&
{\includegraphics[
height=0.25in,
width=0.25in
]%
{T_0.eps}%
}%
&
{\includegraphics[
height=0.25in,
width=0.25in
]%
{T_0.eps}%
}%
&
{\includegraphics[
height=0.25in,
width=0.25in
]%
{T_0.eps}%
}%
\end{array}
\]

\[
T_5=
\begin{array}
[c]{c}%
{\includegraphics[
height=0.25in,
width=0.25in
]%
{T_5.eps}
}
\end{array}
\longmapsto %
\begin{array}
[c]{ccccc}%
{\includegraphics[
height=0.25in,
width=0.25in
]%
{T_0.eps}
}%
&
{\includegraphics[
height=0.25in,
width=0.25in
]%
{T_0.eps}
}%
&
{\includegraphics[
height=0.25in,
width=0.25in
]%
{T_0.eps}
}%
&
{\includegraphics[
height=0.25in,
width=0.25in
]%
{T_0.eps}%
}%
&
{\includegraphics[
height=0.25in,
width=0.25in
]%
{T_0.eps}%
}%
\\%
{\includegraphics[
height=0.25in,
width=0.25in
]%
{T_0.eps}%
}%
&
{\includegraphics[
height=0.25in,
width=0.25in
]%
{T_0.eps}%
}%
&
{\includegraphics[
height=0.25in,
width=0.25in
]%
{T_0.eps}%
}%
&
{\includegraphics[
height=0.25in,
width=0.25in
]%
{T_0.eps}%
}%
&
{\includegraphics[
height=0.25in,
width=0.25in
]%
{T_0.eps}%
}%
\\%
{\includegraphics[
height=0.25in,
width=0.25in
]%
{T_5.eps}%
}%
&
{\includegraphics[
height=0.25in,
width=0.25in
]%
{T_5.eps}%
}%
&
{\includegraphics[
height=0.25in,
width=0.25in
]%
{T_5.eps}%
}%
&
{\includegraphics[
height=0.25in,
width=0.25in
]%
{T_5.eps}%
}%
&
{\includegraphics[
height=0.25in,
width=0.25in
]%
{T_5.eps}%
}%
\\%
{\includegraphics[
height=0.25in,
width=0.25in
]%
{T_0.eps}%
}%
&
{\includegraphics[
height=0.25in,
width=0.25in
]%
{T_0.eps}%
}%
&
{\includegraphics[
height=0.25in,
width=0.25in
]%
{T_0.eps}%
}%
&
{\includegraphics[
height=0.25in,
width=0.25in
]%
{T_0.eps}%
}%
&
{\includegraphics[
height=0.25in,
width=0.25in
]%
{T_0.eps}%
}%
\\%
{\includegraphics[
height=0.25in,
width=0.25in
]%
{T_0.eps}%
}%
&
{\includegraphics[
height=0.25in,
width=0.25in
]%
{T_0.eps}%
}%
&
{\includegraphics[
height=0.25in,
width=0.25in
]%
{T_0.eps}%
}%
&
{\includegraphics[
height=0.25in,
width=0.25in
]%
{T_0.eps}%
}%
&
{\includegraphics[
height=0.25in,
width=0.25in
]%
{T_0.eps}%
}%
\end{array}
\]

\[
T_6=
\begin{array}
[c]{c}%
{\includegraphics[
height=0.25in,
width=0.25in
]%
{T_6.eps}
}
\end{array}
\longmapsto %
\begin{array}
[c]{ccccc}%
{\includegraphics[
height=0.25in,
width=0.25in
]%
{T_0.eps}
}%
&
{\includegraphics[
height=0.25in,
width=0.25in
]%
{T_0.eps}
}%
&
{\includegraphics[
height=0.25in,
width=0.25in
]%
{T_6.eps}
}%
&
{\includegraphics[
height=0.25in,
width=0.25in
]%
{T_0.eps}%
}%
&
{\includegraphics[
height=0.25in,
width=0.25in
]%
{T_0.eps}%
}%
\\%
{\includegraphics[
height=0.25in,
width=0.25in
]%
{T_0.eps}%
}%
&
{\includegraphics[
height=0.25in,
width=0.25in
]%
{T_0.eps}%
}%
&
{\includegraphics[
height=0.25in,
width=0.25in
]%
{T_6.eps}%
}%
&
{\includegraphics[
height=0.25in,
width=0.25in
]%
{T_0.eps}%
}%
&
{\includegraphics[
height=0.25in,
width=0.25in
]%
{T_0.eps}%
}%
\\%
{\includegraphics[
height=0.25in,
width=0.25in
]%
{T_0.eps}%
}%
&
{\includegraphics[
height=0.25in,
width=0.25in
]%
{T_0.eps}%
}%
&
{\includegraphics[
height=0.25in,
width=0.25in
]%
{T_6.eps}%
}%
&
{\includegraphics[
height=0.25in,
width=0.25in
]%
{T_0.eps}%
}%
&
{\includegraphics[
height=0.25in,
width=0.25in
]%
{T_0.eps}%
}%
\\%
{\includegraphics[
height=0.25in,
width=0.25in
]%
{T_0.eps}%
}%
&
{\includegraphics[
height=0.25in,
width=0.25in
]%
{T_0.eps}%
}%
&
{\includegraphics[
height=0.25in,
width=0.25in
]%
{T_6.eps}%
}%
&
{\includegraphics[
height=0.25in,
width=0.25in
]%
{T_0.eps}%
}%
&
{\includegraphics[
height=0.25in,
width=0.25in
]%
{T_0.eps}%
}%
\\%
{\includegraphics[
height=0.25in,
width=0.25in
]%
{T_0.eps}%
}%
&
{\includegraphics[
height=0.25in,
width=0.25in
]%
{T_0.eps}%
}%
&
{\includegraphics[
height=0.25in,
width=0.25in
]%
{T_6.eps}%
}%
&
{\includegraphics[
height=0.25in,
width=0.25in
]%
{T_0.eps}%
}%
&
{\includegraphics[
height=0.25in,
width=0.25in
]%
{T_0.eps}%
}%
\end{array}
\]

\[
T_7=
\begin{array}
[c]{c}%
{\includegraphics[
height=0.25in,
width=0.25in
]%
{T_7.eps}
}
\end{array}
\longmapsto %
\begin{array}
[c]{ccccc}%
{\includegraphics[
height=0.25in,
width=0.25in
]%
{T_0.eps}
}%
&
{\includegraphics[
height=0.25in,
width=0.25in
]%
{T_0.eps}
}%
&
{\includegraphics[
height=0.25in,
width=0.25in
]%
{T_6.eps}
}%
&
{\includegraphics[
height=0.25in,
width=0.25in
]%
{T_0.eps}%
}%
&
{\includegraphics[
height=0.25in,
width=0.25in
]%
{T_0.eps}%
}%
\\%
{\includegraphics[
height=0.25in,
width=0.25in
]%
{T_0.eps}%
}%
&
{\includegraphics[
height=0.25in,
width=0.25in
]%
{T_0.eps}%
}%
&
{\includegraphics[
height=0.25in,
width=0.25in
]%
{T_3.eps}%
}%
&
{\includegraphics[
height=0.25in,
width=0.25in
]%
{T_1.eps}%
}%
&
{\includegraphics[
height=0.25in,
width=0.25in
]%
{T_0.eps}%
}%
\\%
{\includegraphics[
height=0.25in,
width=0.25in
]%
{T_5.eps}%
}%
&
{\includegraphics[
height=0.25in,
width=0.25in
]%
{T_1.eps}%
}%
&
{\includegraphics[
height=0.25in,
width=0.25in
]%
{T_0.eps}%
}%
&
{\includegraphics[
height=0.25in,
width=0.25in
]%
{T_3.eps}%
}%
&
{\includegraphics[
height=0.25in,
width=0.25in
]%
{T_5.eps}%
}%
\\%
{\includegraphics[
height=0.25in,
width=0.25in
]%
{T_0.eps}%
}%
&
{\includegraphics[
height=0.25in,
width=0.25in
]%
{T_3.eps}%
}%
&
{\includegraphics[
height=0.25in,
width=0.25in
]%
{T_1.eps}%
}%
&
{\includegraphics[
height=0.25in,
width=0.25in
]%
{T_0.eps}%
}%
&
{\includegraphics[
height=0.25in,
width=0.25in
]%
{T_0.eps}%
}%
\\%
{\includegraphics[
height=0.25in,
width=0.25in
]%
{T_0.eps}%
}%
&
{\includegraphics[
height=0.25in,
width=0.25in
]%
{T_0.eps}%
}%
&
{\includegraphics[
height=0.25in,
width=0.25in
]%
{T_6.eps}%
}%
&
{\includegraphics[
height=0.25in,
width=0.25in
]%
{T_0.eps}%
}%
&
{\includegraphics[
height=0.25in,
width=0.25in
]%
{T_0.eps}%
}%
\end{array}
\]

\[
T_8=
\begin{array}
[c]{c}%
{\includegraphics[
height=0.25in,
width=0.25in
]%
{T_8.eps}
}
\end{array}
\longmapsto %
\begin{array}
[c]{ccccc}%
{\includegraphics[
height=0.25in,
width=0.25in
]%
{T_0.eps}
}%
&
{\includegraphics[
height=0.25in,
width=0.25in
]%
{T_0.eps}
}%
&
{\includegraphics[
height=0.25in,
width=0.25in
]%
{T_6.eps}
}%
&
{\includegraphics[
height=0.25in,
width=0.25in
]%
{T_0.eps}%
}%
&
{\includegraphics[
height=0.25in,
width=0.25in
]%
{T_0.eps}%
}%
\\%
{\includegraphics[
height=0.25in,
width=0.25in
]%
{T_0.eps}%
}%
&
{\includegraphics[
height=0.25in,
width=0.25in
]%
{T_2.eps}%
}%
&
{\includegraphics[
height=0.25in,
width=0.25in
]%
{T_4.eps}%
}%
&
{\includegraphics[
height=0.25in,
width=0.25in
]%
{T_0.eps}%
}%
&
{\includegraphics[
height=0.25in,
width=0.25in
]%
{T_0.eps}%
}%
\\%
{\includegraphics[
height=0.25in,
width=0.25in
]%
{T_5.eps}%
}%
&
{\includegraphics[
height=0.25in,
width=0.25in
]%
{T_4.eps}%
}%
&
{\includegraphics[
height=0.25in,
width=0.25in
]%
{T_0.eps}%
}%
&
{\includegraphics[
height=0.25in,
width=0.25in
]%
{T_2.eps}%
}%
&
{\includegraphics[
height=0.25in,
width=0.25in
]%
{T_5.eps}%
}%
\\%
{\includegraphics[
height=0.25in,
width=0.25in
]%
{T_0.eps}%
}%
&
{\includegraphics[
height=0.25in,
width=0.25in
]%
{T_0.eps}%
}%
&
{\includegraphics[
height=0.25in,
width=0.25in
]%
{T_2.eps}%
}%
&
{\includegraphics[
height=0.25in,
width=0.25in
]%
{T_4.eps}%
}%
&
{\includegraphics[
height=0.25in,
width=0.25in
]%
{T_0.eps}%
}%
\\%
{\includegraphics[
height=0.25in,
width=0.25in
]%
{T_0.eps}%
}%
&
{\includegraphics[
height=0.25in,
width=0.25in
]%
{T_0.eps}%
}%
&
{\includegraphics[
height=0.25in,
width=0.25in
]%
{T_6.eps}%
}%
&
{\includegraphics[
height=0.25in,
width=0.25in
]%
{T_0.eps}%
}%
&
{\includegraphics[
height=0.25in,
width=0.25in
]%
{T_0.eps}%
}%
\end{array}
\]

\[
T_9=
\begin{array}
[c]{c}%
{\includegraphics[
height=0.25in,
width=0.25in
]%
{T_9.eps}
}
\end{array}
\longmapsto %
\begin{array}
[c]{ccccc}%
{\includegraphics[
height=0.25in,
width=0.25in
]%
{T_0.eps}
}%
&
{\includegraphics[
height=0.25in,
width=0.25in
]%
{T_0.eps}
}%
&
{\includegraphics[
height=0.25in,
width=0.25in
]%
{T_3.eps}
}%
&
{\includegraphics[
height=0.25in,
width=0.25in
]%
{T_1.eps}%
}%
&
{\includegraphics[
height=0.25in,
width=0.25in
]%
{T_0.eps}%
}%
\\%
{\includegraphics[
height=0.25in,
width=0.25in
]%
{T_2.eps}%
}%
&
{\includegraphics[
height=0.25in,
width=0.25in
]%
{T_5.eps}%
}%
&
{\includegraphics[
height=0.25in,
width=0.25in
]%
{T_1.eps}%
}%
&
{\includegraphics[
height=0.25in,
width=0.25in
]%
{T_6.eps}%
}%
&
{\includegraphics[
height=0.25in,
width=0.25in
]%
{T_0.eps}%
}%
\\%
{\includegraphics[
height=0.25in,
width=0.25in
]%
{T_4.eps}%
}%
&
{\includegraphics[
height=0.25in,
width=0.25in
]%
{T_2.eps}%
}%
&
{\includegraphics[
height=0.25in,
width=0.25in
]%
{T_10.eps}%
}%
&
{\includegraphics[
height=0.25in,
width=0.25in
]%
{T_4.eps}%
}%
&
{\includegraphics[
height=0.25in,
width=0.25in
]%
{T_2.eps}%
}%
\\%
{\includegraphics[
height=0.25in,
width=0.25in
]%
{T_0.eps}%
}%
&
{\includegraphics[
height=0.25in,
width=0.25in
]%
{T_6.eps}%
}%
&
{\includegraphics[
height=0.25in,
width=0.25in
]%
{T_3.eps}%
}%
&
{\includegraphics[
height=0.25in,
width=0.25in
]%
{T_5.eps}%
}%
&
{\includegraphics[
height=0.25in,
width=0.25in
]%
{T_4.eps}%
}%
\\%
{\includegraphics[
height=0.25in,
width=0.25in
]%
{T_0.eps}%
}%
&
{\includegraphics[
height=0.25in,
width=0.25in
]%
{T_3.eps}%
}%
&
{\includegraphics[
height=0.25in,
width=0.25in
]%
{T_1.eps}%
}%
&
{\includegraphics[
height=0.25in,
width=0.25in
]%
{T_0.eps}%
}%
&
{\includegraphics[
height=0.25in,
width=0.25in
]%
{T_0.eps}%
}%
\end{array}
\]

\[
T_{10}=
\begin{array}
[c]{c}%
{\includegraphics[
height=0.25in,
width=0.25in
]%
{T_10.eps}
}
\end{array}
\longmapsto %
\begin{array}
[c]{ccccc}%
{\includegraphics[
height=0.25in,
width=0.25in
]%
{T_0.eps}
}%
&
{\includegraphics[
height=0.25in,
width=0.25in
]%
{T_0.eps}
}%
&
{\includegraphics[
height=0.25in,
width=0.25in
]%
{T_6.eps}
}%
&
{\includegraphics[
height=0.25in,
width=0.25in
]%
{T_0.eps}%
}%
&
{\includegraphics[
height=0.25in,
width=0.25in
]%
{T_0.eps}%
}%
\\%
{\includegraphics[
height=0.25in,
width=0.25in
]%
{T_0.eps}%
}%
&
{\includegraphics[
height=0.25in,
width=0.25in
]%
{T_0.eps}%
}%
&
{\includegraphics[
height=0.25in,
width=0.25in
]%
{T_6.eps}%
}%
&
{\includegraphics[
height=0.25in,
width=0.25in
]%
{T_0.eps}%
}%
&
{\includegraphics[
height=0.25in,
width=0.25in
]%
{T_0.eps}%
}%
\\%
{\includegraphics[
height=0.25in,
width=0.25in
]%
{T_5.eps}%
}%
&
{\includegraphics[
height=0.25in,
width=0.25in
]%
{T_5.eps}%
}%
&
{\includegraphics[
height=0.25in,
width=0.25in
]%
{T_10.eps}%
}%
&
{\includegraphics[
height=0.25in,
width=0.25in
]%
{T_5.eps}%
}%
&
{\includegraphics[
height=0.25in,
width=0.25in
]%
{T_5.eps}%
}%
\\%
{\includegraphics[
height=0.25in,
width=0.25in
]%
{T_0.eps}%
}%
&
{\includegraphics[
height=0.25in,
width=0.25in
]%
{T_0.eps}%
}%
&
{\includegraphics[
height=0.25in,
width=0.25in
]%
{T_6.eps}%
}%
&
{\includegraphics[
height=0.25in,
width=0.25in
]%
{T_0.eps}%
}%
&
{\includegraphics[
height=0.25in,
width=0.25in
]%
{T_0.eps}%
}%
\\%
{\includegraphics[
height=0.25in,
width=0.25in
]%
{T_0.eps}%
}%
&
{\includegraphics[
height=0.25in,
width=0.25in
]%
{T_0.eps}%
}%
&
{\includegraphics[
height=0.25in,
width=0.25in
]%
{T_6.eps}%
}%
&
{\includegraphics[
height=0.25in,
width=0.25in
]%
{T_0.eps}%
}%
&
{\includegraphics[
height=0.25in,
width=0.25in
]%
{T_0.eps}%
}%
\end{array}
\]

\end{defn}

\begin{exa}
\[
z_{5 \times}(
\begin{array}
[c]{cccc}%
{\includegraphics[
height=0.15in,
width=0.15in
]%
{T_2.eps}
}%
&
{\includegraphics[
height=0.15in,
width=0.15in
]%
{T_5.eps}
}%
&
{\includegraphics[
height=0.15in,
width=0.15in
]%
{T_1.eps}
}%
&
{\includegraphics[
height=0.15in,
width=0.15in
]%
{T_0.eps}
}%
\\%
{\includegraphics[
height=0.15in,
width=0.15in
]%
{T_6.eps}%
}%
&
{\includegraphics[
height=0.15in,
width=0.15in
]%
{T_2.eps}%
}%
&
{\includegraphics[
height=0.15in,
width=0.15in
]%
{T_9.eps}
}%
&
{\includegraphics[
height=0.15in,
width=0.15in
]%
{T_1.eps}
}%
\\%
{\includegraphics[
height=0.15in,
width=0.15in
]%
{T_3.eps}
}%
&
{\includegraphics[
height=0.15in,
width=0.15in
]%
{T_9.eps}
}%
&
{\includegraphics[
height=0.15in,
width=0.15in
]%
{T_8.eps}
}%
&
{\includegraphics[
height=0.15in,
width=0.15in
]%
{T_4.eps}
}%
\\%
{\includegraphics[
height=0.15in,
width=0.15in
]%
{T_0.eps}
}%
&
{\includegraphics[
height=0.15in,
width=0.15in
]%
{T_3.eps}
}%
&
{\includegraphics[
height=0.15in,
width=0.15in
]%
{T_4.eps}
}%
&
{\includegraphics[
height=0.15in,
width=0.15in
]%
{T_0.eps}
}%
\end{array}
)
\]


\[
=
\begin{array}
[c]{cccccccccccccccccccc}%
{\includegraphics[
height=0.15in,
width=0.15in
]%
{T_0.eps}
}%
&
{\includegraphics[
height=0.15in,
width=0.15in
]%
{T_0.eps}
}%
&
{\includegraphics[
height=0.15in,
width=0.15in
]%
{T_0.eps}
}%
&
{\includegraphics[
height=0.15in,
width=0.15in
]%
{T_0.eps}
}%
&
{\includegraphics[
height=0.15in,
width=0.15in
]%
{T_0.eps}%
}%
&
{\includegraphics[
height=0.15in,
width=0.15in
]%
{T_0.eps}%
}%
&
{\includegraphics[
height=0.15in,
width=0.15in
]%
{T_0.eps}
}%
&
{\includegraphics[
height=0.15in,
width=0.15in
]%
{T_0.eps}
}%
&
{\includegraphics[
height=0.15in,
width=0.15in
]%
{T_0.eps}
}%
&
{\includegraphics[
height=0.15in,
width=0.15in
]%
{T_0.eps}
}%
&
{\includegraphics[
height=0.15in,
width=0.15in
]%
{T_0.eps}
}%
&
{\includegraphics[
height=0.15in,
width=0.15in
]%
{T_0.eps}
}%
&
{\includegraphics[
height=0.15in,
width=0.15in
]%
{T_0.eps}
}%
&
{\includegraphics[
height=0.15in,
width=0.15in
]%
{T_0.eps}
}%
&
{\includegraphics[
height=0.15in,
width=0.15in
]%
{T_0.eps}
}%
&
{\includegraphics[
height=0.15in,
width=0.15in
]%
{T_0.eps}
}%
&
{\includegraphics[
height=0.15in,
width=0.15in
]%
{T_0.eps}
}%
&
{\includegraphics[
height=0.15in,
width=0.15in
]%
{T_0.eps}
}%
&
{\includegraphics[
height=0.15in,
width=0.15in
]%
{T_0.eps}
}%
&
{\includegraphics[
height=0.15in,
width=0.15in
]%
{T_0.eps}
}%
\\
\end{array}
\]


\end{exa}

\begin{lem}
For any positive integer $n$ and arbitrary knot $n$-mosaic $K \in \mathbb{K}^{(n)}$, $z_{5 \times}(K) \sim _{5n} \iota^{4n} K$. Thus, $z_{5 \times}(K) \sim K$.
\end{lem}

\begin{proof}
Under mosaic planar isotopy moves, 
\[
z_{5 \times}(T_7)=z_{5 \times}(
\begin{array}
[c]{c}%
{\includegraphics[
height=0.25in,
width=0.25in
]%
{T_7.eps}
}
\end{array}
)= %
\begin{array}
[c]{ccccc}%
{\includegraphics[
height=0.25in,
width=0.25in
]%
{T_0.eps}
}%
&
{\includegraphics[
height=0.25in,
width=0.25in
]%
{T_0.eps}
}%
&
{\includegraphics[
height=0.25in,
width=0.25in
]%
{T_6.eps}
}%
&
{\includegraphics[
height=0.25in,
width=0.25in
]%
{T_0.eps}%
}%
&
{\includegraphics[
height=0.25in,
width=0.25in
]%
{T_0.eps}%
}%
\\%
{\includegraphics[
height=0.25in,
width=0.25in
]%
{T_0.eps}%
}%
&
{\includegraphics[
height=0.25in,
width=0.25in
]%
{T_0.eps}%
}%
&
{\includegraphics[
height=0.25in,
width=0.25in
]%
{T_3.eps}%
}%
&
{\includegraphics[
height=0.25in,
width=0.25in
]%
{T_1.eps}%
}%
&
{\includegraphics[
height=0.25in,
width=0.25in
]%
{T_0.eps}%
}%
\\%
{\includegraphics[
height=0.25in,
width=0.25in
]%
{T_5.eps}%
}%
&
{\includegraphics[
height=0.25in,
width=0.25in
]%
{T_1.eps}%
}%
&
{\includegraphics[
height=0.25in,
width=0.25in
]%
{T_0.eps}%
}%
&
{\includegraphics[
height=0.25in,
width=0.25in
]%
{T_3.eps}%
}%
&
{\includegraphics[
height=0.25in,
width=0.25in
]%
{T_5.eps}%
}%
\\%
{\includegraphics[
height=0.25in,
width=0.25in
]%
{T_0.eps}%
}%
&
{\includegraphics[
height=0.25in,
width=0.25in
]%
{T_3.eps}%
}%
&
{\includegraphics[
height=0.25in,
width=0.25in
]%
{T_1.eps}%
}%
&
{\includegraphics[
height=0.25in,
width=0.25in
]%
{T_0.eps}%
}%
&
{\includegraphics[
height=0.25in,
width=0.25in
]%
{T_0.eps}%
}%
\\%
{\includegraphics[
height=0.25in,
width=0.25in
]%
{T_0.eps}%
}%
&
{\includegraphics[
height=0.25in,
width=0.25in
]%
{T_0.eps}%
}%
&
{\includegraphics[
height=0.25in,
width=0.25in
]%
{T_6.eps}%
}%
&
{\includegraphics[
height=0.25in,
width=0.25in
]%
{T_0.eps}%
}%
&
{\includegraphics[
height=0.25in,
width=0.25in
]%
{T_0.eps}%
}%
\end{array}
\sim  %
\begin{array}
[c]{ccccc}%
{\includegraphics[
height=0.25in,
width=0.25in
]%
{T_0.eps}
}%
&
{\includegraphics[
height=0.25in,
width=0.25in
]%
{T_0.eps}
}%
&
{\includegraphics[
height=0.25in,
width=0.25in
]%
{T_6.eps}
}%
&
{\includegraphics[
height=0.25in,
width=0.25in
]%
{T_0.eps}%
}%
&
{\includegraphics[
height=0.25in,
width=0.25in
]%
{T_0.eps}%
}%
\\%
{\includegraphics[
height=0.25in,
width=0.25in
]%
{T_0.eps}%
}%
&
{\includegraphics[
height=0.25in,
width=0.25in
]%
{T_0.eps}%
}%
&
{\includegraphics[
height=0.25in,
width=0.25in
]%
{T_6.eps}%
}%
&
{\includegraphics[
height=0.25in,
width=0.25in
]%
{T_0.eps}%
}%
&
{\includegraphics[
height=0.25in,
width=0.25in
]%
{T_0.eps}%
}%
\\%
{\includegraphics[
height=0.25in,
width=0.25in
]%
{T_5.eps}%
}%
&
{\includegraphics[
height=0.25in,
width=0.25in
]%
{T_5.eps}%
}%
&
{\includegraphics[
height=0.25in,
width=0.25in
]%
{T_7.eps}%
}%
&
{\includegraphics[
height=0.25in,
width=0.25in
]%
{T_5.eps}%
}%
&
{\includegraphics[
height=0.25in,
width=0.25in
]%
{T_5.eps}%
}%
\\%
{\includegraphics[
height=0.25in,
width=0.25in
]%
{T_0.eps}%
}%
&
{\includegraphics[
height=0.25in,
width=0.25in
]%
{T_0.eps}%
}%
&
{\includegraphics[
height=0.25in,
width=0.25in
]%
{T_6.eps}%
}%
&
{\includegraphics[
height=0.25in,
width=0.25in
]%
{T_0.eps}%
}%
&
{\includegraphics[
height=0.25in,
width=0.25in
]%
{T_0.eps}%
}%
\\%
{\includegraphics[
height=0.25in,
width=0.25in
]%
{T_0.eps}%
}%
&
{\includegraphics[
height=0.25in,
width=0.25in
]%
{T_0.eps}%
}%
&
{\includegraphics[
height=0.25in,
width=0.25in
]%
{T_6.eps}%
}%
&
{\includegraphics[
height=0.25in,
width=0.25in
]%
{T_0.eps}%
}%
&
{\includegraphics[
height=0.25in,
width=0.25in
]%
{T_0.eps}%
}%
\end{array}
\]

\[
z_{5 \times}(T_8)=z_{5 \times}(
\begin{array}
[c]{c}%
{\includegraphics[
height=0.25in,
width=0.25in
]%
{T_8.eps}
}
\end{array}
)= %
\begin{array}
[c]{ccccc}%
{\includegraphics[
height=0.25in,
width=0.25in
]%
{T_0.eps}
}%
&
{\includegraphics[
height=0.25in,
width=0.25in
]%
{T_0.eps}
}%
&
{\includegraphics[
height=0.25in,
width=0.25in
]%
{T_6.eps}
}%
&
{\includegraphics[
height=0.25in,
width=0.25in
]%
{T_0.eps}%
}%
&
{\includegraphics[
height=0.25in,
width=0.25in
]%
{T_0.eps}%
}%
\\%
{\includegraphics[
height=0.25in,
width=0.25in
]%
{T_0.eps}%
}%
&
{\includegraphics[
height=0.25in,
width=0.25in
]%
{T_2.eps}%
}%
&
{\includegraphics[
height=0.25in,
width=0.25in
]%
{T_4.eps}%
}%
&
{\includegraphics[
height=0.25in,
width=0.25in
]%
{T_0.eps}%
}%
&
{\includegraphics[
height=0.25in,
width=0.25in
]%
{T_0.eps}%
}%
\\%
{\includegraphics[
height=0.25in,
width=0.25in
]%
{T_5.eps}%
}%
&
{\includegraphics[
height=0.25in,
width=0.25in
]%
{T_4.eps}%
}%
&
{\includegraphics[
height=0.25in,
width=0.25in
]%
{T_0.eps}%
}%
&
{\includegraphics[
height=0.25in,
width=0.25in
]%
{T_2.eps}%
}%
&
{\includegraphics[
height=0.25in,
width=0.25in
]%
{T_5.eps}%
}%
\\%
{\includegraphics[
height=0.25in,
width=0.25in
]%
{T_0.eps}%
}%
&
{\includegraphics[
height=0.25in,
width=0.25in
]%
{T_0.eps}%
}%
&
{\includegraphics[
height=0.25in,
width=0.25in
]%
{T_2.eps}%
}%
&
{\includegraphics[
height=0.25in,
width=0.25in
]%
{T_4.eps}%
}%
&
{\includegraphics[
height=0.25in,
width=0.25in
]%
{T_0.eps}%
}%
\\%
{\includegraphics[
height=0.25in,
width=0.25in
]%
{T_0.eps}%
}%
&
{\includegraphics[
height=0.25in,
width=0.25in
]%
{T_0.eps}%
}%
&
{\includegraphics[
height=0.25in,
width=0.25in
]%
{T_6.eps}%
}%
&
{\includegraphics[
height=0.25in,
width=0.25in
]%
{T_0.eps}%
}%
&
{\includegraphics[
height=0.25in,
width=0.25in
]%
{T_0.eps}%
}%
\end{array}
\sim  %
\begin{array}
[c]{ccccc}%
{\includegraphics[
height=0.25in,
width=0.25in
]%
{T_0.eps}
}%
&
{\includegraphics[
height=0.25in,
width=0.25in
]%
{T_0.eps}
}%
&
{\includegraphics[
height=0.25in,
width=0.25in
]%
{T_6.eps}
}%
&
{\includegraphics[
height=0.25in,
width=0.25in
]%
{T_0.eps}%
}%
&
{\includegraphics[
height=0.25in,
width=0.25in
]%
{T_0.eps}%
}%
\\%
{\includegraphics[
height=0.25in,
width=0.25in
]%
{T_0.eps}%
}%
&
{\includegraphics[
height=0.25in,
width=0.25in
]%
{T_0.eps}%
}%
&
{\includegraphics[
height=0.25in,
width=0.25in
]%
{T_6.eps}%
}%
&
{\includegraphics[
height=0.25in,
width=0.25in
]%
{T_0.eps}%
}%
&
{\includegraphics[
height=0.25in,
width=0.25in
]%
{T_0.eps}%
}%
\\%
{\includegraphics[
height=0.25in,
width=0.25in
]%
{T_5.eps}%
}%
&
{\includegraphics[
height=0.25in,
width=0.25in
]%
{T_5.eps}%
}%
&
{\includegraphics[
height=0.25in,
width=0.25in
]%
{T_8.eps}%
}%
&
{\includegraphics[
height=0.25in,
width=0.25in
]%
{T_5.eps}%
}%
&
{\includegraphics[
height=0.25in,
width=0.25in
]%
{T_5.eps}%
}%
\\%
{\includegraphics[
height=0.25in,
width=0.25in
]%
{T_0.eps}%
}%
&
{\includegraphics[
height=0.25in,
width=0.25in
]%
{T_0.eps}%
}%
&
{\includegraphics[
height=0.25in,
width=0.25in
]%
{T_6.eps}%
}%
&
{\includegraphics[
height=0.25in,
width=0.25in
]%
{T_0.eps}%
}%
&
{\includegraphics[
height=0.25in,
width=0.25in
]%
{T_0.eps}%
}%
\\%
{\includegraphics[
height=0.25in,
width=0.25in
]%
{T_0.eps}%
}%
&
{\includegraphics[
height=0.25in,
width=0.25in
]%
{T_0.eps}%
}%
&
{\includegraphics[
height=0.25in,
width=0.25in
]%
{T_6.eps}%
}%
&
{\includegraphics[
height=0.25in,
width=0.25in
]%
{T_0.eps}%
}%
&
{\includegraphics[
height=0.25in,
width=0.25in
]%
{T_0.eps}%
}%
\end{array}
\]

\[
z_{5 \times}(T_9)=z_{5 \times}(
\begin{array}
[c]{c}%
{\includegraphics[
height=0.25in,
width=0.25in
]%
{T_9.eps}
}
\end{array}
)= %
\begin{array}
[c]{ccccc}%
{\includegraphics[
height=0.25in,
width=0.25in
]%
{T_0.eps}
}%
&
{\includegraphics[
height=0.25in,
width=0.25in
]%
{T_0.eps}
}%
&
{\includegraphics[
height=0.25in,
width=0.25in
]%
{T_3.eps}
}%
&
{\includegraphics[
height=0.25in,
width=0.25in
]%
{T_1.eps}%
}%
&
{\includegraphics[
height=0.25in,
width=0.25in
]%
{T_0.eps}%
}%
\\%
{\includegraphics[
height=0.25in,
width=0.25in
]%
{T_2.eps}%
}%
&
{\includegraphics[
height=0.25in,
width=0.25in
]%
{T_5.eps}%
}%
&
{\includegraphics[
height=0.25in,
width=0.25in
]%
{T_1.eps}%
}%
&
{\includegraphics[
height=0.25in,
width=0.25in
]%
{T_6.eps}%
}%
&
{\includegraphics[
height=0.25in,
width=0.25in
]%
{T_0.eps}%
}%
\\%
{\includegraphics[
height=0.25in,
width=0.25in
]%
{T_4.eps}%
}%
&
{\includegraphics[
height=0.25in,
width=0.25in
]%
{T_2.eps}%
}%
&
{\includegraphics[
height=0.25in,
width=0.25in
]%
{T_10.eps}%
}%
&
{\includegraphics[
height=0.25in,
width=0.25in
]%
{T_4.eps}%
}%
&
{\includegraphics[
height=0.25in,
width=0.25in
]%
{T_2.eps}%
}%
\\%
{\includegraphics[
height=0.25in,
width=0.25in
]%
{T_0.eps}%
}%
&
{\includegraphics[
height=0.25in,
width=0.25in
]%
{T_6.eps}%
}%
&
{\includegraphics[
height=0.25in,
width=0.25in
]%
{T_3.eps}%
}%
&
{\includegraphics[
height=0.25in,
width=0.25in
]%
{T_5.eps}%
}%
&
{\includegraphics[
height=0.25in,
width=0.25in
]%
{T_4.eps}%
}%
\\%
{\includegraphics[
height=0.25in,
width=0.25in
]%
{T_0.eps}%
}%
&
{\includegraphics[
height=0.25in,
width=0.25in
]%
{T_3.eps}%
}%
&
{\includegraphics[
height=0.25in,
width=0.25in
]%
{T_1.eps}%
}%
&
{\includegraphics[
height=0.25in,
width=0.25in
]%
{T_0.eps}%
}%
&
{\includegraphics[
height=0.25in,
width=0.25in
]%
{T_0.eps}%
}%
\end{array}
\sim  %
\begin{array}
[c]{ccccc}%
{\includegraphics[
height=0.25in,
width=0.25in
]%
{T_0.eps}
}%
&
{\includegraphics[
height=0.25in,
width=0.25in
]%
{T_0.eps}
}%
&
{\includegraphics[
height=0.25in,
width=0.25in
]%
{T_6.eps}
}%
&
{\includegraphics[
height=0.25in,
width=0.25in
]%
{T_0.eps}%
}%
&
{\includegraphics[
height=0.25in,
width=0.25in
]%
{T_0.eps}%
}%
\\%
{\includegraphics[
height=0.25in,
width=0.25in
]%
{T_0.eps}%
}%
&
{\includegraphics[
height=0.25in,
width=0.25in
]%
{T_0.eps}%
}%
&
{\includegraphics[
height=0.25in,
width=0.25in
]%
{T_6.eps}%
}%
&
{\includegraphics[
height=0.25in,
width=0.25in
]%
{T_0.eps}%
}%
&
{\includegraphics[
height=0.25in,
width=0.25in
]%
{T_0.eps}%
}%
\\%
{\includegraphics[
height=0.25in,
width=0.25in
]%
{T_5.eps}%
}%
&
{\includegraphics[
height=0.25in,
width=0.25in
]%
{T_5.eps}%
}%
&
{\includegraphics[
height=0.25in,
width=0.25in
]%
{T_9.eps}%
}%
&
{\includegraphics[
height=0.25in,
width=0.25in
]%
{T_5.eps}%
}%
&
{\includegraphics[
height=0.25in,
width=0.25in
]%
{T_5.eps}%
}%
\\%
{\includegraphics[
height=0.25in,
width=0.25in
]%
{T_0.eps}%
}%
&
{\includegraphics[
height=0.25in,
width=0.25in
]%
{T_0.eps}%
}%
&
{\includegraphics[
height=0.25in,
width=0.25in
]%
{T_6.eps}%
}%
&
{\includegraphics[
height=0.25in,
width=0.25in
]%
{T_0.eps}%
}%
&
{\includegraphics[
height=0.25in,
width=0.25in
]%
{T_0.eps}%
}%
\\%
{\includegraphics[
height=0.25in,
width=0.25in
]%
{T_0.eps}%
}%
&
{\includegraphics[
height=0.25in,
width=0.25in
]%
{T_0.eps}%
}%
&
{\includegraphics[
height=0.25in,
width=0.25in
]%
{T_6.eps}%
}%
&
{\includegraphics[
height=0.25in,
width=0.25in
]%
{T_0.eps}%
}%
&
{\includegraphics[
height=0.25in,
width=0.25in
]%
{T_0.eps}%
}%
\end{array}
.
\]
After this arrangement, we can contract segments. As parallel translation can be realized by mosaic moves, we have required result.

\end{proof}

Using this Lemma, we have the following:
\begin{pro}{\label{Ku2}}
Let $k$ be tame knot (or link) and let $K$ be arbitrary chosen mosaic representative of $k$, then there is a grid diagram $G$ of $k$ such that $K \sim G$.
\end{pro}

\begin{proof}
$z_{5 \times}(K) $ has no tiles like $T_7,T_8,T_9$. Using mosaic planar isotopy moves (and mosaic injection $\iota$, if necessary), we can perturb the result so that no segments lie on the same horizontal or vertical line. Thus, we get a grid diagram of $k$.
\end{proof}

\section{Proof of main theorem}{\label {proof}}

Let $k_{1}$ and $k_{2}$ be two tame knots (or links), let $K_{1}$ and
$K_{2}$ be two arbitrary chosen mosaic representatives of $k_{1}$ and $k_{2}$,
respectively, and let $G_1$ and $G_2$ be two arbitrary chosen grid diagrams of $k_{1}$ and $k_{2}$, respectively, then the following conditions are equivalent.
\begin{description}
\item[(1)] $k_{1}$ and $k_{2}$ are of the same knot type.
\item[(2)] $K_{1}$ and $K_{2}$ are of the same knot mosaic
type.
\item[(3)] $G_{1}$ and $G_{2}$ can be connected by a finite sequence of the elementary moves.
\end{description}

\begin{proof}[Proof of main theorem]
\ 
\begin{description}
\item[(1)$\Rightarrow $ (3)] Theorem {\ref{Dy}}.
\item[(3)$\Rightarrow $ (2)] By Proposition {\ref{Ku1}}, assumption means $G_1 \sim G_2$. By Proposition {\ref{Ku2}}, $\exists G(K_i) \in \mathbb{G}$ s.t. $K_i \sim G(K_i) \ (i=1,2)$. Here, $G(K_i)$ and $G_i$ are grid diagrams of the same knot $k_{i}$, so, by using Theorem {\ref{Dy}} and Proposition {\ref{Ku1}}, $G(K_i) \sim G_i$. From these facts, result follows.

\item[(2)$\Rightarrow $ (1)] By the forgetting map of the square tiling of knot mosaic, $K_i$ can naturally be identified with a diagram of the
knot $k_i$ , and mosaic moves can be regarded as ambient istopies. 
\end{description}
\end{proof}

\section{ $\clubsuit$ appendix $\clubsuit$ \ \  A list of $\mathbb{K}^{(n)}/\mathbb{A}(n) \ (n \leq 4)$ and mosaic number of some knots }{\label {appendix}}

\begin{thm}{\label{list}}
\ 
\begin{enumerate}
\item
\[ 
\mathbb{K}^{(1)}/\mathbb{A}(1) = \{ 
\begin{array}
[c]{c}%
{\includegraphics[
height=0.25in,
width=0.25in
]%
{T_0.eps}
}%
\end{array}
\} 
\]

\item 
\[ 
\mathbb{K}^{(2)}/\mathbb{A}(2) = \{ 
\begin{array}
[c]{cc}%
{\includegraphics[
height=0.25in,
width=0.25in
]%
{T_0.eps}
}%
&
{\includegraphics[
height=0.25in,
width=0.25in
]%
{T_0.eps}
}%
\\%
{\includegraphics[
height=0.25in,
width=0.25in
]%
{T_0.eps}%
}%
&
{\includegraphics[
height=0.25in,
width=0.25in
]%
{T_0.eps}%
}%
\end{array}
,\ 
\begin{array}
[c]{cc}%
{\includegraphics[
height=0.25in,
width=0.25in
]%
{T_2.eps}
}%
&
{\includegraphics[
height=0.25in,
width=0.25in
]%
{T_1.eps}
}%
\\%
{\includegraphics[
height=0.25in,
width=0.25in
]%
{T_3.eps}%
}%
&
{\includegraphics[
height=0.25in,
width=0.25in
]%
{T_4.eps}%
}%
\end{array}
\} 
\]

\item 
\[ 
\mathbb{K}^{(3)}/\mathbb{A}(3) = \{ 
\begin{array}
[c]{ccc}%
{\includegraphics[
height=0.17in,
width=0.17in
]%
{T_0.eps}
}%
&
{\includegraphics[
height=0.17in,
width=0.17in
]%
{T_0.eps}
}%
&
{\includegraphics[
height=0.17in,
width=0.17in
]%
{T_0.eps}
}%
\\%
{\includegraphics[
height=0.17in,
width=0.17in
]%
{T_0.eps}%
}%
&
{\includegraphics[
height=0.17in,
width=0.17in
]%
{T_0.eps}%
}%
&
{\includegraphics[
height=0.17in,
width=0.17in
]%
{T_0.eps}
}%
\\%
{\includegraphics[
height=0.17in,
width=0.17in
]%
{T_0.eps}
}%
&
{\includegraphics[
height=0.17in,
width=0.17in
]%
{T_0.eps}
}%
&
{\includegraphics[
height=0.17in,
width=0.17in
]%
{T_0.eps}
}%
\end{array}
, 
\begin{array}
[c]{ccc}%
{\includegraphics[
height=0.17in,
width=0.17in
]%
{T_2.eps}
}%
&
{\includegraphics[
height=0.17in,
width=0.17in
]%
{T_1.eps}
}%
&
{\includegraphics[
height=0.17in,
width=0.17in
]%
{T_0.eps}
}%
\\%
{\includegraphics[
height=0.17in,
width=0.17in
]%
{T_3.eps}%
}%
&
{\includegraphics[
height=0.17in,
width=0.17in
]%
{T_4.eps}%
}%
&
{\includegraphics[
height=0.17in,
width=0.17in
]%
{T_0.eps}
}%
\\%
{\includegraphics[
height=0.17in,
width=0.17in
]%
{T_0.eps}
}%
&
{\includegraphics[
height=0.17in,
width=0.17in
]%
{T_0.eps}
}%
&
{\includegraphics[
height=0.17in,
width=0.17in
]%
{T_0.eps}
}%
\end{array}
,
\begin{array}
[c]{ccc}%
{\includegraphics[
height=0.17in,
width=0.17in
]%
{T_2.eps}
}%
&
{\includegraphics[
height=0.17in,
width=0.17in
]%
{T_1.eps}
}%
&
{\includegraphics[
height=0.17in,
width=0.17in
]%
{T_0.eps}
}%
\\%
{\includegraphics[
height=0.17in,
width=0.17in
]%
{T_3.eps}%
}%
&
{\includegraphics[
height=0.17in,
width=0.17in
]%
{T_8.eps}%
}%
&
{\includegraphics[
height=0.17in,
width=0.17in
]%
{T_1.eps}
}%
\\%
{\includegraphics[
height=0.17in,
width=0.17in
]%
{T_0.eps}
}%
&
{\includegraphics[
height=0.17in,
width=0.17in
]%
{T_3.eps}
}%
&
{\includegraphics[
height=0.17in,
width=0.17in
]%
{T_4.eps}
}%
\end{array}
,
\begin{array}
[c]{ccc}%
{\includegraphics[
height=0.17in,
width=0.17in
]%
{T_0.eps}
}%
&
{\includegraphics[
height=0.17in,
width=0.17in
]%
{T_2.eps}
}%
&
{\includegraphics[
height=0.17in,
width=0.17in
]%
{T_1.eps}
}%
\\%
{\includegraphics[
height=0.17in,
width=0.17in
]%
{T_2.eps}%
}%
&
{\includegraphics[
height=0.17in,
width=0.17in
]%
{T_7.eps}%
}%
&
{\includegraphics[
height=0.17in,
width=0.17in
]%
{T_4.eps}
}%
\\%
{\includegraphics[
height=0.17in,
width=0.17in
]%
{T_3.eps}
}%
&
{\includegraphics[
height=0.17in,
width=0.17in
]%
{T_4.eps}
}%
&
{\includegraphics[
height=0.17in,
width=0.17in
]%
{T_0.eps}
}%
\end{array}
\}
\]

\item 
\[ 
\mathbb{K}^{(4)}/\mathbb{A}(4) = \{ 
\begin{array}
[c]{cccc}%
{\includegraphics[
height=0.15in,
width=0.15in
]%
{T_0.eps}
}%
&
{\includegraphics[
height=0.15in,
width=0.15in
]%
{T_0.eps}
}%
&
{\includegraphics[
height=0.15in,
width=0.15in
]%
{T_0.eps}
}%
&
{\includegraphics[
height=0.15in,
width=0.15in
]%
{T_0.eps}
}%
\\%
{\includegraphics[
height=0.15in,
width=0.15in
]%
{T_0.eps}%
}%
&
{\includegraphics[
height=0.15in,
width=0.15in
]%
{T_0.eps}%
}%
&
{\includegraphics[
height=0.15in,
width=0.15in
]%
{T_0.eps}
}%
&
{\includegraphics[
height=0.15in,
width=0.15in
]%
{T_0.eps}
}%
\\%
{\includegraphics[
height=0.15in,
width=0.15in
]%
{T_0.eps}
}%
&
{\includegraphics[
height=0.15in,
width=0.15in
]%
{T_0.eps}
}%
&
{\includegraphics[
height=0.15in,
width=0.15in
]%
{T_0.eps}
}%
&
{\includegraphics[
height=0.15in,
width=0.15in
]%
{T_0.eps}
}%
\\%
{\includegraphics[
height=0.15in,
width=0.15in
]%
{T_0.eps}
}%
&
{\includegraphics[
height=0.15in,
width=0.15in
]%
{T_0.eps}
}%
&
{\includegraphics[
height=0.15in,
width=0.15in
]%
{T_0.eps}
}%
&
{\includegraphics[
height=0.15in,
width=0.15in
]%
{T_0.eps}
}%
\end{array}
,
\begin{array}
[c]{cccc}%
{\includegraphics[
height=0.15in,
width=0.15in
]%
{T_2.eps}
}%
&
{\includegraphics[
height=0.15in,
width=0.15in
]%
{T_1.eps}
}%
&
{\includegraphics[
height=0.15in,
width=0.15in
]%
{T_0.eps}
}%
&
{\includegraphics[
height=0.15in,
width=0.15in
]%
{T_0.eps}
}%
\\%
{\includegraphics[
height=0.15in,
width=0.15in
]%
{T_3.eps}%
}%
&
{\includegraphics[
height=0.15in,
width=0.15in
]%
{T_4.eps}%
}%
&
{\includegraphics[
height=0.15in,
width=0.15in
]%
{T_0.eps}
}%
&
{\includegraphics[
height=0.15in,
width=0.15in
]%
{T_0.eps}
}%
\\%
{\includegraphics[
height=0.15in,
width=0.15in
]%
{T_0.eps}
}%
&
{\includegraphics[
height=0.15in,
width=0.15in
]%
{T_0.eps}
}%
&
{\includegraphics[
height=0.15in,
width=0.15in
]%
{T_0.eps}
}%
&
{\includegraphics[
height=0.15in,
width=0.15in
]%
{T_0.eps}
}%
\\%
{\includegraphics[
height=0.15in,
width=0.15in
]%
{T_0.eps}
}%
&
{\includegraphics[
height=0.15in,
width=0.15in
]%
{T_0.eps}
}%
&
{\includegraphics[
height=0.15in,
width=0.15in
]%
{T_0.eps}
}%
&
{\includegraphics[
height=0.15in,
width=0.15in
]%
{T_0.eps}
}%
\end{array}
,
\begin{array}
[c]{cccc}%
{\includegraphics[
height=0.15in,
width=0.15in
]%
{T_2.eps}
}%
&
{\includegraphics[
height=0.15in,
width=0.15in
]%
{T_1.eps}
}%
&
{\includegraphics[
height=0.15in,
width=0.15in
]%
{T_0.eps}
}%
&
{\includegraphics[
height=0.15in,
width=0.15in
]%
{T_0.eps}
}%
\\%
{\includegraphics[
height=0.15in,
width=0.15in
]%
{T_3.eps}%
}%
&
{\includegraphics[
height=0.15in,
width=0.15in
]%
{T_8.eps}%
}%
&
{\includegraphics[
height=0.15in,
width=0.15in
]%
{T_1.eps}
}%
&
{\includegraphics[
height=0.15in,
width=0.15in
]%
{T_0.eps}
}%
\\%
{\includegraphics[
height=0.15in,
width=0.15in
]%
{T_0.eps}
}%
&
{\includegraphics[
height=0.15in,
width=0.15in
]%
{T_3.eps}
}%
&
{\includegraphics[
height=0.15in,
width=0.15in
]%
{T_4.eps}
}%
&
{\includegraphics[
height=0.15in,
width=0.15in
]%
{T_0.eps}
}%
\\%
{\includegraphics[
height=0.15in,
width=0.15in
]%
{T_0.eps}
}%
&
{\includegraphics[
height=0.15in,
width=0.15in
]%
{T_0.eps}
}%
&
{\includegraphics[
height=0.15in,
width=0.15in
]%
{T_0.eps}
}%
&
{\includegraphics[
height=0.15in,
width=0.15in
]%
{T_0.eps}
}%
\end{array}
,
\begin{array}
[c]{cccc}%
{\includegraphics[
height=0.15in,
width=0.15in
]%
{T_2.eps}
}%
&
{\includegraphics[
height=0.15in,
width=0.15in
]%
{T_1.eps}
}%
&
{\includegraphics[
height=0.15in,
width=0.15in
]%
{T_0.eps}
}%
&
{\includegraphics[
height=0.15in,
width=0.15in
]%
{T_0.eps}
}%
\\%
{\includegraphics[
height=0.15in,
width=0.15in
]%
{T_3.eps}%
}%
&
{\includegraphics[
height=0.15in,
width=0.15in
]%
{T_8.eps}%
}%
&
{\includegraphics[
height=0.15in,
width=0.15in
]%
{T_1.eps}
}%
&
{\includegraphics[
height=0.15in,
width=0.15in
]%
{T_0.eps}
}%
\\%
{\includegraphics[
height=0.15in,
width=0.15in
]%
{T_0.eps}
}%
&
{\includegraphics[
height=0.15in,
width=0.15in
]%
{T_3.eps}
}%
&
{\includegraphics[
height=0.15in,
width=0.15in
]%
{T_8.eps}
}%
&
{\includegraphics[
height=0.15in,
width=0.15in
]%
{T_1.eps}
}%
\\%
{\includegraphics[
height=0.15in,
width=0.15in
]%
{T_0.eps}
}%
&
{\includegraphics[
height=0.15in,
width=0.15in
]%
{T_0.eps}
}%
&
{\includegraphics[
height=0.15in,
width=0.15in
]%
{T_3.eps}
}%
&
{\includegraphics[
height=0.15in,
width=0.15in
]%
{T_4.eps}
}%
\end{array}
\]
\[
,
\begin{array}
[c]{cccc}%
{\includegraphics[
height=0.15in,
width=0.15in
]%
{T_2.eps}
}%
&
{\includegraphics[
height=0.15in,
width=0.15in
]%
{T_1.eps}
}%
&
{\includegraphics[
height=0.15in,
width=0.15in
]%
{T_2.eps}
}%
&
{\includegraphics[
height=0.15in,
width=0.15in
]%
{T_1.eps}
}%
\\%
{\includegraphics[
height=0.15in,
width=0.15in
]%
{T_3.eps}%
}%
&
{\includegraphics[
height=0.15in,
width=0.15in
]%
{T_4.eps}%
}%
&
{\includegraphics[
height=0.15in,
width=0.15in
]%
{T_3.eps}
}%
&
{\includegraphics[
height=0.15in,
width=0.15in
]%
{T_4.eps}
}%
\\%
{\includegraphics[
height=0.15in,
width=0.15in
]%
{T_2.eps}
}%
&
{\includegraphics[
height=0.15in,
width=0.15in
]%
{T_1.eps}
}%
&
{\includegraphics[
height=0.15in,
width=0.15in
]%
{T_2.eps}
}%
&
{\includegraphics[
height=0.15in,
width=0.15in
]%
{T_1.eps}
}%
\\%
{\includegraphics[
height=0.15in,
width=0.15in
]%
{T_3.eps}
}%
&
{\includegraphics[
height=0.15in,
width=0.15in
]%
{T_4.eps}
}%
&
{\includegraphics[
height=0.15in,
width=0.15in
]%
{T_3.eps}
}%
&
{\includegraphics[
height=0.15in,
width=0.15in
]%
{T_4.eps}
}%
\end{array}%
\} 
\]

\end{enumerate}

\end{thm}

\begin{defn}
Define the {\bf mosaic number $m(k)$} of a knot $k$ as the
smallest integer $n$ for which $k$ is representable as a knot $n$-mosaic.
\ For example, $m(\text{unknot})=2$ and $m(\text{trefoil})=4$.

\end{defn}
As a corollary of Theorem {\ref{list}}, we can get some knots with mosaic number 5.

\begin{cor}
\[
m(4_1)=m(5_1)=m(5_2)=m(6_2)=m(7_4)=5
\]
\end{cor}

\begin{proof}
\[
4_1\longmapsto 
%
\]

These knots did not appear in the list of Theorem {\ref{list}}, so we have required results.

\end{proof}

\begin{conj}
$m(6_3)=6$
\end{conj}
\[
6_3\longmapsto 
%
\]

\end{document}